\documentclass{article}
\usepackage{graphicx} % Required for inserting images

\usepackage[margin=1in]{geometry}
\usepackage[utf8]{inputenc}
\usepackage[T1]{fontenc}
\usepackage[english]{babel}
\usepackage{amsmath,amsfonts,amsthm, mathrsfs}
\usepackage{hyperref}
\usepackage{tikz-cd}
\usepackage{bbm}
\usepackage[capitalize]{cleveref}
\crefname{theorem}{Theorem}{Theorems}

\usepackage{biblatex}
\theoremstyle{plain}
\newtheorem{theorem}{Theorem}[section]
\newtheorem{corollary}[theorem]{Corollary}
\newtheorem{lemma}[theorem]{Lemma}
\newtheorem{proposition}[theorem]{Proposition}

\theoremstyle{definition}
\newtheorem{definition}[theorem]{Definition}

\newtheorem{remark}[theorem]{Remark}
\numberwithin{equation}{section}
\numberwithin{table}{section}
\DeclareMathOperator{\des}{des}
\DeclareMathOperator{\E}{\mathbb{E}}
\DeclareMathOperator{\Prob}{\mathbb{P}}
\DeclareMathOperator{\Var}{Var}
\DeclareMathOperator{\Cov}{Cov}

\title{A Central Limit Theorem on Two-Sided Descents of Mallows Distributed Elements of Finite Coxeter Groups}
\author{Maxwell Sun}

\begin{document}

\maketitle

\begin{abstract}
The Mallows distribution is a non-uniform distribution, first introduced over permutations to study non-ranked data, in which permutations are weighted according to their length. It can be generalized to any Coxeter group, and we study the distribution of $\des(w) + \des(w^{-1})$ where $w$ is a Mallows distributed element of a finite irreducible Coxeter group. We show that the asymptotic behavior of this statistic is Guassian. The proof uses a size-bias coupling with Stein's method.
\end{abstract}

\section{Introduction}

Asymptotic normality of statistics over the group of permutations (where we view the random variables as these statistics on $S_n$, taking $n \to \infty$) is a topic in probability that has been studied for many years \cite{Gon44, Ho51}. Chatterjee and Diaconis have proved a general asymptotic normality result for local statistics on $S_n$, including the two-sided descent (which we will discuss more later) \cite{CD17}. There is interest in finding general results for such statistics over an arbitrary Coxeter group. Given a sequence of finite Coxeter groups $(W_n)_{n \in \mathbb{N}}$, let $w_n$ be a random uniformly distributed element of $W_n$. We are interested in the behavior of the distributions of statistics on the $w_n$.

Kahle and Stump showed that $(\ell(w_n))_n$ is asymptotically normal if and only if $(W_n)_{n \in \mathbb{N}}$ satisfies some condition that requires the variance of $\ell(w_n)$ to grow sufficiently quickly, where $\ell(w)$ denotes the number of inversions in $w$ for a Coxeter group element $w$. They also showed a similar limit theorem for $(\des(w_n))_n$ where $\des(w_n)$ denotes the descent statistic, namely that $(\des(w_n))_n$ is asymptotically normal if and only if the variance $\Var(t(w_n))$ approaches infinity \cite{KS20}. Br\"uck and R\"ottger showed a similar limit theorem for $(t(w_n))_n$ under some assumptions about the sequence $(W_n)_{n \in \mathbb{N}}$ (which they refer to as being "well-behaved") where $t(w) = \des(w) + \des(w^{-1})$ is the two-sided descent statistic of a Coxeter group element $w$ \cite{BR22}. F\'eray then proved that this assumption on the behavior of $(W_n)_n$ is unnecessary, i.e. that a central limit theorem holds for $(t(w_n))_n$ if and only if $\Var(t(w_n)) \to \infty$ as $n \to \infty$. We generalize this result to cases when $w_n$ is not uniformly distributed in $W_n$, but rather when it is distributed according to the product of Mallows measures. He recently proved a central limit theorem for the two-side descents of Mallows permutations, covering the special case of $W_n$ all being symmetric groups \cite{He_2022}. We use his results and methods to prove our generalization.

The definition of a Mallows measure is given in Section \ref{subsection: Mallows}. The Mallows distribution was introduced by Mallows to study non-uniform ranked data \cite{Mal57}. The Mallows distribution on permutations is heavily studied in applied statistics \cite{Muk16a, Tang19}. Mallows permutations also have applications in studying stable matchings \cite{AHHL21} and one-dependent processes \cite{HHL}. The Mallows model also appears as the stationary distribution of the biased interchange process on the line \cite{LL19}, and it’s projection to particle systems is the stationary distribution of ASEP. It additionally appears as the stationary distribution of random walks related to Hecke algebras \cite{B20, DR00}.

\subsection{Results}

Coxeter groups are a class of groups with a specific kind of presentation in which, in particular, each generator has degree $2$. Irreducible Coxeter groups are those that cannot be decomposed into direct products of other Coxeter groups. More details can be found in \Cref{subsec: Coxeter}. For a finite irreducible Coxeter group, we define the Mallows measure over the group to be such that the measure of the singleton containing an element of length $\ell$ is proportional to $q^\ell$ where $q$ is some fixed parameter. Our main result is \Cref{CLT: General Coxeter Groups}.

\begin{theorem} \label{CLT: General Coxeter Groups}
Let $(W_n)_{n \ge 1}$ be a sequence of finite Coxeter groups and let $w_n$ be a random element of $W_n$ such that if a $W_n$ is the product of irreducible groups $W_n^1 \times W_n^2 \times \cdots \times W_n^{k_n}$ then the distribution of $w_n$ is the product of Mallows distributions on these factors. Then, $\frac{t(w_n) - \E(t(w_n))}{\sqrt{\Var(t(w_n))}}$ tends to $Z$ in distribution if and only if $\lim\limits_{n \to \infty}\Var(t(w_n)) = \infty$.
\end{theorem}

The proof of our result involves first proving a central limit theorem for the three infinite families of finite irreducible Coxeter groups. One of these families is $(S_n)_n$ and thus we can directly use the results in \cite{He_2022}. For the other families, we use similar techniques. We then convert the bounds to bounds involving Wasserstein $2$-distance and closely mirror the proof used by F\'eray for the uniform distribution case to obtain \Cref{CLT: General Coxeter Groups}.

First, we show analogous results to \cite[Thoerem 1.1]{He_2022} for the families of groups $(B_n)_{n \in \mathbb{N}}$ and $(D_n)_{n \in \mathbb{N}}$ using similar methods. These families of groups are closely related to the permutations. Both $B_n$ and $D_n$ are obtained by adding a generator $s_0$ to $S_n$, so the combinatorics involved is very similar to that used by He. We use a size-bias coupling for the two-sided descent statistic with Stein's method to obtain these bounds, which show that the two-sided descent of a Mallows distributed element of $B_n$ asymptotically approaches a standard normal, in some sense (and similarly for $D_n$). The main difficulty comes in proving that the error terms yielded by Stein's method are small. This involves bounding lots of covariance terms.

\begin{theorem}
\label{thm: main theorem 1, B_n}
Let $w\in B_n$ for $n\geq 4$ be Mallows distributed with parameter $q\in (0,\infty)$ and let $\mu$ and $\sigma^2$ denote the mean and variance of $\des(w)+\des(w^{-1})$. Let $Z$ denote a standard normal random variable. Then for all piecewise continuously differentiable functions $h:\mathbb{R}\to \mathbb{R}$,
\begin{equation*}
\begin{split}
    &\left|\E h\left(\frac{\des(w)+\des(w^{-1})-\mu}{\sigma}\right)-\E h(Z)\right|
    \leq \left(360\|h\|_\infty+236\|h'\|_\infty\max(q^{-1/2},q^{1/2})\right)n^{-\frac{1}{2}}.
\end{split}
\end{equation*}
\end{theorem}

\begin{theorem}
\label{thm: main theorem 2, D_n}
Let $w\in D_n$ for $n\geq 30$ be Mallows distributed with parameter $q\in (0,\infty)$ and let $\mu$ and $\sigma^2$ denote the mean and variance of $\des(w)+\des(w^{-1})$. Let $Z$ denote a standard normal random variable. Then for all piecewise continuously differentiable functions $h:\mathbb{R}\to \mathbb{R}$,
\begin{equation*}
\begin{split}
    &\left|\E h\left(\frac{\des(w)+\des(w^{-1})-\mu}{\sigma}\right)-\E h(Z)\right|
    \leq \left(768\|h\|_\infty+666\|h'\|_\infty\max(q^{-1/2},q^{1/2})\right)n^{-\frac{1}{2}}.
\end{split}
\end{equation*}
\end{theorem}

In proving the above results, we also get a bound on the Wasserstein 1-distance between a standard normal and the normalized two-sided descent. We let $d_W$ denote the Wasserstein $1$-distance and define it as follows:

\begin{definition}
Let $A$ and $B$ be real-valued random variables. Then the \emph{Wasserstein $1$-distance} between $A$ and $B$ is given by
\[
d_W(A,B) = \inf_{(A',B') \in \Gamma(A,B)} \E|A'-B'|
\]
where $\Gamma(A,B)$ denotes the set of couplings of $(A,B)$.
\end{definition}

Another distance on real-valued random variables is the following:
\begin{definition}
Let $A$ and $B$ be real-valued random variables. Then the \emph{Wasserstein $2$-distance} between $A$ and $B$ is given by
\[
d_2(A,B) = \inf_{(A',B') \in \Gamma(A,B)} \left(\E[(A'-B')^2]\right)^{\frac{1}{2}}
\]
where $\Gamma(A,B)$ denotes the set of couplings of $(A,B)$.
\end{definition}

The Wasserstein $1$-distance bound is as follows:

\begin{theorem} \label{thm: Wasserstein 1-distance}
Let $w\in B_n$ for $n\geq 4$ be Mallows distributed with parameter $q\in (0,\infty)$. Then,
$$
d_W\left(\frac{\des(w)+\des(w^{-1})-\mu}{\sigma}, Z \right) \le \left(180+236\max(q^{-1/2},q^{1/2})\right)n^{-\frac{1}{2}}.
$$
Similarly, let $w\in D_n$ for $n\geq 30$ be Mallows distributed with parameter $q\in (0,\infty)$. Then,
$$
d_W\left(\frac{\des(w)+\des(w^{-1})-\mu}{\sigma}, Z \right) \le \left(384+666\max(q^{-1/2},q^{1/2})\right)n^{-\frac{1}{2}}.
$$
\end{theorem}

We then bound the Wasserstein $2$-distance using the Wasserstein $1$-distance bounds via tail estimates on the two-sided descent distribution.
\begin{theorem} \label{thm: Wasserstein 2-distance}
Let $w\in B_n$ for $n\geq 4$ be Mallows distributed with parameter $q\in (0,\infty)$. Let $k = \min(q, 1/q)$. Assume that $nk \ge 50$. Then,
$$
d_2\left(\frac{\des(w)+\des(w^{-1})-\mu}{\sigma}, Z \right) \le 100(nk)^{-1/4}(\log (nk))^{1/2}
$$
The same is true if $w \in D_n$ or $A_n$ is Mallows distributed with parameter $q \in (0, \infty)$. 
\end{theorem}

This bound leads to \Cref{CLT: General Coxeter Groups}, using a similar proof to that in \cite{ferayCLT}. The strategy is to consider separately the $W_n^i$ that are "small" and those that are "large." The approach relies on the classification of finite Coxeter groups.

The product of Mallows distributions can be viewed as a "generalized" Mallows distribution on the Coxeter group $W_n$ where there is a parameter for each conjugacy class of generators (as it is not hard to see that two generators are conjugate if and only if they are in the same irreducible component of $W_n$). It is not hard to see that the product of Mallows measures all with the same parameter $q$ is also a Mallows measure. It then follows that the above holds also when $w_n$ is assumed simply to be Mallows distributed over $W_n$ (as a subset of the theorem).

\subsection{Further Directions}

It is natural to ask now if similar central limit theorems exist for some families of infinite Coxeter groups, over which the Mallows measure can be defined. For example, the affine symmetric group have similar combinatorial structures to symmetric groups (just as $B_n$ and $D_n$ do), so there may be great potential for results regarding them. We also wonder whether there exists a joint central limit theorem for $(\des(w), \des(w^{-1}))$, where $w$ is a Mallows distributed element of some Coxeter group. Such a result exists for the family $(A_n)_n$ of groups of permutations, as shown in \cite{He_2022}. He proves this by using a regenerative process related to Mallows permutations. As $B_n$ and $D_n$ are combinatorially structured rather similarly to permutations, we ask if the techniques used by He can be adapted.

\subsection{Outline of Paper}

In \Cref{sec: background}, we give some definitions and introduce the properties of Coxeter groups and Mallows distributions that are important to us. This is mostly just material regarding Coxeter groups. Then, we present some preliminary results in \Cref{sec: prelim} that are useful in later sections. We use Stein's method in \Cref{sec: coupling} and then bound the error terms in \Cref{sec: easy cov,,sec: hard cov,,sec: limit thms} to obtain Theorems \ref{thm: main theorem 1, B_n}, \ref{thm: main theorem 2, D_n} and \ref{thm: Wasserstein 1-distance}. In \Cref{sec: 2-distance}, we use \Cref{thm: Wasserstein 1-distance} to obtain a Wasserstein $2$-distance limit theorem and then use this to prove the main result in \Cref{sec: grand result}. The appendix contains some bounds on some moments of the descent and two-sided descent that are useful in later sections.

\subsection{Acknowledgements}

A special thanks to Jimmy He for being my mentor for this project. He originally came up with the idea of trying to prove such a result and provided critical references to help me get started (including his paper \cite{He_2022}). He also gave many helpful insights and tips, and the methods he used in his previous work \cite{He_2022} played an extremely important role in this paper.

Another thanks to the donors of MIT's Undergraduate Research Opportunities Program Fund for First-Year and Sophomore Students. They provided the funding necessary for the author to carry out this work, which was the direct result of a UROP project. Thanks to Alexei Borodin for being the faculty sponsor for this project.

\section{Background and Definitions} \label{sec: background}

\subsection{Finite Coxeter Groups} \label{subsec: Coxeter}

A $\textit{Coxeter group}$ is a group with presentation
$$
\left\langle r_1, r_2, \dots, r_n | (r_ir_j)^{e_{ij}} = 1 \ \forall i,j \right\rangle
$$
where $e_{ij}$ is $\infty$ or a positive integer greater than $1$ when $i \neq j$ (with $e_{ij} = \infty$ meaning no relation) and $e_{ii} = 1$ for all $i$. The pair $(W,S)$ of a Coxeter group with its generator set $S = \{r_1, r_2, \dots, r_n\}$ is called a $\textit{Coxeter system}$. Coxeter systems can be represented graphically where each generator is given a vertex and there is an edge between two generators $r_i, r_j$ if and only if $e_{ij} > 2$ (which is equivalent to $r_i, r_j$ not commuting). This is called a $\textit{Coxeter graph}$. If $e_{ij} \neq 3$, the edge is labeled with the value of $e_{ij}$ (otherwise, it is assumed to be $3$).

The symmetric group $S_{n+1}$ is an example of a finite Coxeter group and is denoted $A_n$. It's $n$ generators are the transpositions. There exist other infinite families of irreducible finite Coxeter groups, namely the signed permutations, or hyperoctahedral group, $B_n$ and the even signed permutations $D_n$. An element of $B_n$ is an automorphism of $[\pm n] = \{-n, -n+1, \dots, -1, 1, 2, \dots, n\}$ that commutes with negation (that is, $w(-i) = -w(i)$ for all $w \in B_n$). $D_n$ is the subset of $B_n$ that consists of all permutations that have send an even number of positive integers to negative integers. 

The groups $B_n$ and $D_n$ each have generator sets of size $n$ that can be written as $\{s_0, s_1, \dots, s_{n-1}\}$ where $s_i$ denotes the transposition $(i \ i+1) (-i \ -i-1)$. In $B_n$, $s_0$ represents the transposition $(-1 \ 1)$ that switches the values in the middle two indices. In $D_n$, $s_0$ represents the transposition $(-1 \ 2)$ that swaps the values at indices $-1,2$ and the values at indices $-2, 1$. In some of the results that follow, we will consider the arbitrary finite Coxeter group $W$ with arbitrary generator set $\{s_0, s_1, \dots, s_{n-1}\}$. The Coxeter graphs of $S_n = A_{n-1}, B_n, D_n$ are shown below (and are taken from \cite[Section 1.2]{BB05}).

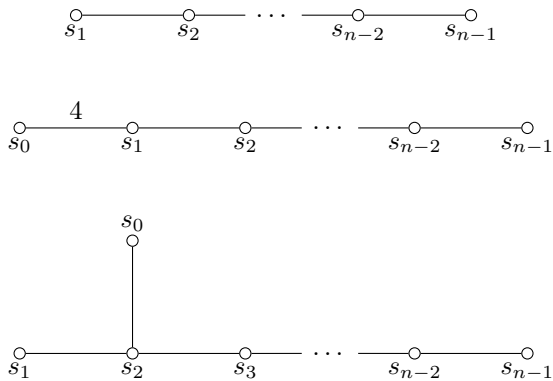
\begin{figure}[h]
    \centering
\begin{tikzpicture}[scale = 1.5]
    % First row of vertices
    \node[draw, circle, inner sep=1.5pt] (s1) at (1.5, 0) {};
    \node[below] at (s1) {$s_1$};
    
    \node[draw, circle, inner sep=1.5pt] (s2) at (2.5, 0) {};
    \node[below] at (s2) {$s_2$};
    
    \node at (3.25, 0) {$\dots$};
    
    \node[draw, circle, inner sep=1.5pt] (sn-2) at (4, 0) {};
    \node[below] at (sn-2) {$s_{n-2}$};
    
    \node[draw, circle, inner sep=1.5pt] (sn-1) at (5, 0) {};
    \node[below] at (sn-1) {$s_{n-1}$};
    
    \draw (s1) -- (s2);
    \draw (s2) -- (3, 0);  % edge leading to dots
    \draw (3.5, 0) -- (sn-2);  % edge from dots
    \draw (sn-2) -- (sn-1);
    
    % Second row of vertices (1 point below)
    \node[draw, circle, inner sep=1.5pt] (s0) at (1, -1) {};
    \node[below] at (s0) {$s_0$};
    
    \node[draw, circle, inner sep=1.5pt] (s1b) at (2, -1) {};
    \node[below] at (s1b) {$s_1$};
    
    \node[draw, circle, inner sep=1.5pt] (s2b) at (3, -1) {};
    \node[below] at (s2b) {$s_2$};
    
    \node at (3.75, -1) {$\dots$};
    
    \node[draw, circle, inner sep=1.5pt] (sn-2b) at (4.5, -1) {};
    \node[below] at (sn-2b) {$s_{n-2}$};
    
    \node[draw, circle, inner sep=1.5pt] (sn-1b) at (5.5, -1) {};
    \node[below] at (sn-1b) {$s_{n-1}$};
    
    \draw (s0) -- (s1b) node[midway, above] {4};
    \draw (s1b) -- (s2b);
    \draw (s2b) -- (3.5, -1);  % edge leading to dots
    \draw (4, -1) -- (sn-2b);  % edge from dots
    \draw (sn-2b) -- (sn-1b);

    % Third row of vertices (placed lower, with s_0 directly above s_2)
    \node[draw, circle, inner sep=1.5pt] (s1c) at (1, -3) {};
    \node[below] at (s1c) {$s_1$};
    
    \node[draw, circle, inner sep=1.5pt] (s2c) at (2, -3) {};
    \node[below] at (s2c) {$s_2$};
    
    \node[draw, circle, inner sep=1.5pt] (s3c) at (3, -3) {};
    \node[below] at (s3c) {$s_3$};
    
    \node at (3.75, -3) {$\dots$};
    
    \node[draw, circle, inner sep=1.5pt] (sn-3c) at (4.5, -3) {};
    \node[below] at (sn-3c) {$s_{n-2}$};
    
    \node[draw, circle, inner sep=1.5pt] (sn-2c) at (5.5, -3) {};
    \node[below] at (sn-2c) {$s_{n-1}$};
    
    % Vertex s0 directly above s2
    \node[draw, circle, inner sep=1.5pt] (s0c) at (2, -2) {};
    \node[above] at (s0c) {$s_0$};
    
    % Edges
    \draw (s0c) -- (s2c);
    \draw (s1c) -- (s2c);
    \draw (s2c) -- (s3c);
    \draw (s3c) -- (3.5, -3);  % edge leading to dots
    \draw (4, -3) -- (sn-3c);  % edge from dots
    \draw (sn-3c) -- (sn-2c);
\end{tikzpicture}
    \caption{Coxeter graphs of $S_n, B_n, D_n$, respectively}
    \label{Fig: Coxeter graphs}
\end{figure}

A Coxeter group is \emph{irreducible} if it is not the product of two nontrivial Coxeter groups. The following is a classification of the irreducible finite Coxeter groups. Rigorous descriptions of the groups referred to below are given in \cite{BB05}.

\begin{proposition}[{\cite[Appendix A1]{BB05}}]
    All finite irreducible Coxeter groups are contained in the families $\{A_n\}_{n \ge 1}$, $\{B_n\}_{n \ge 2}$, $\{D_n\}_{n \ge 4}$, $\{I_2(m)\}_{m \ge 3}$ ($I_2(m)$ being the dihedral group of order $2m$) save for a finite number of exceptions which are denoted $E_6, E_7, E_8, F_4, G_2, H_3, H_4$.
\end{proposition}

Any element $w$ of a Coxeter group can be written as a product of generators. The $\textit{length}$ of $w$, denoted $\ell(w)$, is the minimal number of generators among all such products. A \emph{reduced word} or \emph{reduced expression} for $w$ is a product of $\ell(w)$ generators that equals $w$. We say that the element $w \in W$ has a \emph{descent} at a generator $s_i$ if $\ell(ws_i) < \ell(w)$. From general Coxeter theory (see \cite[Proposition 1.4.2]{BB05}), $\ell(ws_i) = \ell(w) \pm 1$, so either $w$ has a descent at $s_i$ or $ws_i$ does. 

For a $w \in B_n$, a pair of indices $(i,j)$ is called a \emph{B-inversion} if $1 \le i<j \le n$ and $w(i) > w(j)$ or $1 \le -i \le j \le n$ and $w(i) > w(j)$. It is then true that $\ell(w)$ counts the number of B-inversions of $w$ (see \cite[Proposition 8.1.1]{BB05}). In other words, 
$$
\ell(w) = |\{(i,j): 1\le i < j \le n, w(i) > w(j)\}| + |\{(i,j):-n \le -j \le i \le -1, w(i) > w(j)\}|.
$$
Note that there are potentially up to $n^2$ such pairs $(i,j)$. Similarly, for $w \in D_n$,
$$
\ell(w) = |\{(i,j) \in [n]^2: i < j, w(i) > w(j)\}| + |\{(i,j):-i \in [n], j \in [n], -j < i, w(i) > w(j)\}|.
$$
There are potentially up to $n(n-1)$ such pairs $(i,j)$, and we call such a pair a D-inversion. We denote by $\des_i(w)$ the indicator function of the event that $w$ has a descent at $s_i$. The $\textit{(right) descent statistic}$, of $w$ is $\des(w) = \sum\limits_i \des_i(w)$. The $\textit{two-sided descent statistic}$ is given by $\des(w) + \des(w^{-1})$. 

\subsection{Parabolic Subgroups}

For an arbitrary Coxeter system $W$ with generator set $T$ and any subset $S$ of $T$, we denote the subgroup of $W$ generated by $S$ by $W_S$. This is called a $\textit{parabolic subgroup}$. We may also define the subset $W^S$ of $W$ that consists exactly of all of the elements of $W$ that do not have a right descent at any of the generators in $S$. We have the following result from \cite{BB05}.

\begin{lemma}[{\cite[Proposition 2.4.4]{BB05}}] \label{lem: parabolic decomposition}
For any Coxeter group $W$ and subset $S$ of its generators, we have:
\begin{enumerate}
    \item Every $w \in W$ has a unique decomposition $w = w^S \cdot w_S$ such that $w^S \in W^S$ and $w_S \in W_S$.
    \item Given this factorization, $\ell(w) = \ell(w^S) + \ell(w_S)$.
\end{enumerate}
\end{lemma}

\begin{corollary} \label{Cor: right mult parabolic}
Given $W$ and a subset $S$ of its generators, for any $w \in W$ and $w_0 \in W_S$, $w_S = (ww_0)_S$ and $(ww_0)^S = w^S$
\end{corollary}
\begin{proof}
We have that $ws_0 = w^Sw_Sw_0$. It is clear that $w^S \in W^S$ and $w_sw_0 \in W_S$, so uniqueness in Lemma \ref{lem: parabolic decomposition} implies the result.
\end{proof}

Now, we call $S$ a \textit{connected} subset of $T$ if its elements form a connected subgraph in the Coxeter graph of $W$. Thus, the connected subsets of the generators of $B_n$ are of the form $\{s_i| i \in I\}$ for a closed interval $I \subseteq [0,n-1]$. The connected subsets of the generators of $D_n$ are of the forms
\begin{itemize}
    \item $\{s_i| a \le i \le b\}$ for fixed $1 \le a \le b \le n-1$, 
    \item $\{s_0\} \cup \{s_i| 2 \le i \le b \le n\}$ for fixed $b \le n-1$, 
    \item or $\{s_0, s_1, s_2\}$. 
\end{itemize}
For a subset $S$ of the generators of $B_n$ and $D_n$, let $\overline{S}$ be the set of indices associated to $S$. That is, if $s_i \in S$ for $i \ge 1$, $i, i+1 \in \overline{S}$. If $s_0 \in S$, $-1, 1 \in \overline{S}$ in the case of $B_n$ and $-2, -1, 1, 2 \in \overline{S}$ in the case of $D_n$. Moreover, if in $S$, $s_k$ and $s_0$ are in the same connected component, then $-k, -k-1 \in \overline{S}$. It is not hard to check that in both groups $\overline{S}$ is either a set of positive consecutive integers or of the form $[\pm i]$ for all connected subsets $S$ of $T$. For a general $S$, we can partition $S$ into its \textit{connected components} given by the connected components of the subgraph corresponding to $S$.

Lemma \ref{lem: parabolic decomposition} can be interpreted in the cases $W = B_n, D_n$ as follows: for an element $w$ of $W$, $w_S$ gives the order of the values of $w$ within $\overline{S'}$ for each connected component $S'$ of $S$, while $w^S$ gives the assignment of values to the indices associated to the connected components. This is because switching values within the indices associated to a connected component corresponds to multiplying on the right by generators in $S$ and this simply changes $w_S$.

\subsection{The Mallows Distribution} \label{subsection: Mallows}

The \emph{Mallows measure} $\mu_q^W$ is a one-paramater family of probability measures on a finite Coxeter group $W$ indexed by $q \in (0, \infty)$. A random element $w \in W$ is \emph{Mallows distributed} with parameter $q$ if it is distributed according to $\mu_q^W$. We have that
$$
\mu_q^W(w_0) = \frac{q^{\ell(w_0)}}{Z_W(q)}
$$
for all $w_0 \in W$, where $Z_W(q)$ is a normalization constant. It is not hard to see that $Z_W$ is the evaluation of a generating function for the length of a random Mallows distributed element of $W$. By Theorem 2.3 and the bottom of page $6$ (where the degrees are explicitly given) in \cite{KS20}, we have the following, where $[m]_q = (q^m - 1)/(q-1)$ for any positive integer $m$.
$$
Z_{A_n} = [n]_q! = [n]_q[n-1]_q \dots [1]_q, 
$$
$$
Z_{B_n} = [2n]_q!! = [2n]_q[2n-2]_q \dots [2]_q, 
$$
$$
Z_{D_n} = [n]_q[2n-2]_q!! = [n]_q[2n-2]_q \dots [2]_q.
$$
When $q = 1$, the Mallows distribution reduces to the uniform distribution. If $W$ is the product of finite Coxeter group, the Mallows distribution becomes the product of independent Mallows distributions on each factor.

\section{Preliminary Results on Mallows-Distributed Coxeter Groups} \label{sec: prelim}

In this section we state and proof some computations and probabilisitic results about the Mallows measure on Coxeter Groups and that will be useful later on in bounding the error term in the size-bias coupling.

\subsection{Reversal Symmetry}

We first give the following result which allows us to simply work with the case $q \le 1$. For any signed permutation $w \in B_n$ or $w \in D_n$, let $-w$ denote the permutation obtained by reversing the values of $w$, i.e. setting $(-w)(i) = -w(i)$ for all indices $i$.

\begin{lemma}
\label{lem: q_invert}
For $q \in (0, \infty)$, $w \sim \mu_q^W$ is identically distributed to $-w'$ where $w' \sim \mu_{1/q}^W$ for $W = B_n, D_n$.
\end{lemma}
\begin{proof}
The probability that $w = w_0$ for some fixed $w_0 \in W$ is proportional to $q^{\ell(w_0)}$. The probability that $-w' = w_0$ is proportional to $(1/q)^{\ell(-w_0)}$. However, note that $(i,j)$ is an inversion in $w_0$ if and only if $w_0(i) > w_0(j)$, which is equivalent to $(-w_0)(i) < (-w_0)(j)$, i.e. $(i,j)$ not being an inversion in $-w_0$. Thus, $\ell(w_0) + \ell(-w_0) = n^2$ in the $B_n$ case while $\ell(w_0) + \ell(-w_0) = n(n-1)$ in the $D_n$ case. In any case, we have that $q^{\ell(w_0)}$ is proportional to $(1/q)^{\ell(-w_0)} = q^{-\ell(-w_0)}$. It follows that
$$
\Prob(w = w_0) = \Prob(-w' = w_0),
$$
so the conclusion follows.
\end{proof}

\begin{corollary} \label{cor: q_invert}
For $q \in (0, \infty)$ and $W = B_n, D_n$, the distribution of $\des(w) + \des(w^{-1})$ where $w \sim \mu_q^W$ is identical to the distribution of $2n - \des(w') - \des((w')^{-1})$ where $w' \sim \mu_{1/q}^W$.
\end{corollary}
\begin{proof}
We first show that $-(w')^{-1} = (-w')^{-1}$. To see this, simply note that $(-(w')^{-1})(i) = -(w')^{-1}(i) = (w')^{-1}(-i) = (-w')^{-1}(i)$ since the index that $-w'$ sends to $i$ and the index that $w'$ sends to $-i$ are both the negative of the index that $w'$ sends to $i$. Now, we also have that $w'$ has a descent at $i$ if and only if $-w'$ does not. Thus, $\des(-w') = n - \des(w')$, and, similarly, $\des((-w')^{-1}) = \des(-(w')^{-1}) = n - \des((w')^{-1})$. Thus, $\des(-w') + \des((-w')^{-1}) = 2n - \des(w') - \des((w')^{-1})$ and the result follows from Lemma \ref{lem: q_invert}.
\end{proof}

\subsection{Independence Results}

For each generator $s_i$ of $W$, let $T_i$ be the set of generators consisting of $s_i$ and its neighbors (i.e. the generators it does not commute with). We then have the following results.

\begin{lemma} \label{First_indep_result}
Let $W$ be a finite Coxeter system generated by $s_0, \dots, s_{n-1}$. Let $S \subseteq \{s_0, \dots, s_{n-1}\}$ and let $A$ be a subset of $W$ such that $Aw = A$ for all $w \in W_S$. If $w \in W$ is Mallows distributed, then $w_S$ is Mallows distributed in $W_S$ with the same parameter and $w_S$ and the event $A$ are independent.
\end{lemma}
\begin{proof}
For any $w_0 \in W_S$, note that $(ww_0)_S = w_Sw_0$ for all $w \in W$ by \Cref{Cor: right mult parabolic}. Then,
$$
\mathbb{P}(\{w_S = w_0\} \cap A) = \mathbb{P}((\{w_S = e\} \cap A)w_0)
$$
since $Aw_0 = A$. Now, we can decompose $w = w^Sw_S$ so that $\ell(w) = \ell(w^S) + \ell(w_S)$. Thus, if $w_S = e$, then $\ell(ww_0) = \ell(w^S) + \ell(w_0)$ since $(ww_0)_S = w_Sw_0$. It follows that
$$
\mathbb{P}((\{w_S=e\} \cap A)w_0) = q^{\ell(w_0)}\mathbb{P}(\{w_S=e\} \cap A)
$$
$$
\implies \mathbb{P}(\{w_S = w_0\} \cap A) = q^{\ell(w_0)}\mathbb{P}(\{w_S=e\} \cap A).
$$
Now, it is clear that $Ww = W$ for all $w \in W$, so we can take $A = W$ and obtain that $w_S$ is, in fact, Mallows distributed in $W_S$. Moreover, for fixed $A$, we have that $\mathbb{P}(\{w_S = w_0\} \cap A)$ is directly proportional to $q^{\ell(w_0)}$, which is directly proportional to $\mathbb{P}(w_S = w_0)$. Since $\sum_{w_0 \in W_S}\mathbb{P}(\{w_S = w_0\} \cap A) = \mathbb{P}(A)$, it follows that
$$
\mathbb{P}(\{w_S = w_0\} \cap A) = \mathbb{P}(w_S = w_0) \mathbb{P}(A)
$$
yielding independence of $w_S$ and $A$.
\end{proof}

\begin{lemma} \label{Lemma: Descents_invariant_commute}
If $s_i$ and $s_j$ commute (where $i \neq j$), then $\mathrm{des}_j(w) = \mathrm{des}_j(ws_i)$.
\end{lemma}
\begin{proof}
Suppose that $\text{des}_j(w) = 1$. Then by \cite[Corollary 1.4.6]{BB05}, $w = s^1s^2\dots s^ks_j$ where $k = \ell(w) - 1$ and $s^a$ are generators. We then have that 
$$
ws_i = s^1s^2\dots s^ks_js_i = s^1s^2\dots s^ks_is_j.
$$
Either this is a reduced word for $ws_i$, or it is two characters too long by \cite[Proposition 1.4.2]{BB05}. In the first case, it is clear then that $ws_i$ has a descent at $s_j$. In the second case, since $\ell(ws_i) < \ell(w)$, by the Strong Exchange Property, $ws_i = s^1\dots \widehat{s^m}\dots s^ks_j$ for some $1 \le m \le k$. Note that we cannot have $s_j$ be the deleted generator, or else we have
$$
ws_i = s^1s^2\dots s^ks_js_i = s^1s^2\dots s^k \implies s_js_i = 1 \implies s_i = s_j
$$
which is not true. Moreover $ws_i = s^1\dots \widehat{s^m}\dots s^ks_j$ is a reduced word. But, clearly $ws_is_j = s^1\dots \widehat{s^m}\dots s^k$ is of shorter length than $ws_i$, so $\text{des}_j(ws_i) = 1$, as desired. We have shown that $\text{des}_j(w) = 1$ implies $\text{des}_j(ws_i) = 1$. The other direction is obtained by replacing $w$ with $ws_i$.
\end{proof}

\begin{lemma} \label{Lemma: Difference is parabolic}
$\des(w) - \des(w_i^*)$ is determined by $w_{T_i}$ for all $w \in W$. Furthermore, $\des(w) - \des(w_i^*) = \sum_{s_j \in T_i} (\des_j(w) - \des_j(w_i^*))$.
\end{lemma}
\begin{proof}
First, we show that $w_{T_i}$ has a right descent at $s_j \in T_i$ if and only if $w$ does. If $w_{T_i}$ has a right descent at $s_j$, then, by \Cref{lem: parabolic decomposition} and \Cref{Cor: right mult parabolic}, $\ell(ws_j) = \ell((ws_j)^{T_i}) + \ell((ws_j)_{T_i}) = \ell(w^{T_i}) + \ell(w_{T_i}s_j) = \ell(w^{T_i}) + \ell(w_{T_i}) - 1 = \ell(w) - 1$ so $w$ has a right descent at $s_j$. If $w$ has a right descent at $s_j$, we have, since $w^{T_i}(w_{T_i}s_j)$ is the Proposition 2.4.4 factorization for $ws_j$, 
$$
\ell(w^{T_i}) + \ell(w_{T_i}s_j) = \ell(ws_j) < \ell(w) = \ell(w^{T_i}) + \ell(w_{T_i}) \implies \ell(w_{T_i}s_j) < \ell(w_{T_i}).
$$
It follows that $w_{T_i}$ has a right descent at $s_j$, as desired. Now, note that $ws_i = w^{T_i}w_{T_i}s_i$ is the unique factorization of $ws_i$, so we have that, by the same logic above, $ws_i$ has a right descent at $s_j \in T_i$ if and only if $w_{T_i}s_i$ does. It follows that both $\text{des}_j(w)$ and $\text{des}_j(ws_i)$ depend only on $w_{T_i}$. Lastly, for $s_j \not\in T_i$, $\text{des}_j(w) = \text{des}_j(ws_i)$ by Lemma \ref{Lemma: Descents_invariant_commute}. Thus,
$$
\des(w) - \des(w_i^*) = \sum_{s_j \in T_i} (\des_j(w) - \des_j(w_i^*))
$$
Moreover, it is clear that whether $w_i^*$ is equal to $w$ or $ws_i$ is also determined by $w_{T_i}$. It follows that $\des(w) - \des(w_i^*)$ is a function of $w_{T_i}$.
\end{proof}

\begin{lemma} \label{Lemma: general_parabolic_indep}
Let $S$ and $S'$ be two disjoint subsets of the generator set of a Coxeter group $W$ such that $s$ and $s'$ commute for any $s \in S$ and $s' \in S'$. If $w \sim \mu_q^W$, then $w_S$ and $w_{S'}$ are independent and Mallows distributed with parameter $q$ in $W_S$ and $W_{S'}$, respectively.
\end{lemma}
\begin{proof}

First, note that elements of $W_S$ commute with elements of $W_{S'}$. Consider the map from $\{w \in W\ | \ w_S = a\}$ to $\{w \in W \ | \ w_S = b\}$ given by $w \mapsto wa^{-1}b$. It is clear that this a bijection with the inverse $w \mapsto wb^{-1}a$. Moreover, this map preserves $w_{S'}$ since 
$$
wa^{-1}b = w^{S'}w_{S'}a^{-1}b = (w^{S'}a^{-1}b)w_{S'}
$$
and $w^{S'}a^{-1}b$ has no descent at any generators in $S'$ by Lemma \ref{Lemma: Descents_invariant_commute}. That is, $(wa^{-1}b)_{S'} = w_{S'}$. Hence, this map is also a bijection between the sets $\{w|w_S = a, w_{S'} = c\}$ and $\{w|w_S = b, w_{S'} = c\}$ for any $c \in W_{S'}$. We have that $(wa^{-1}b)^S = w^S$ so that
$$
\ell(wa^{-1}b) - \ell(w) = \ell(w^Sb) - \ell(w^Sa) = [\ell(w^S) + \ell(b)] - [\ell(w^S) + \ell(a)] = \ell(b) - \ell(a).
$$
where the second equality follows from Lemma \ref{lem: parabolic decomposition}. Thus,
$$
\mathbb{P}(w_S = b, w_{S'} = c) = q^{\ell(b) - \ell(a)}\mathbb{P}(w_S = a, w_{S'} = c)
$$
$$
\implies \mathbb{P}(w_S = b| w_{S'} = c) = q^{\ell(b) - \ell(a)}\mathbb{P}(w_S = a| w_{S'} = c)
$$
i.e. $w_S$ is a Mallows distributed conditioned on $w_{S'} = c$ for all $c \in W_{S'}$. Since the distribution is not affected by $c$, the result follows.
\end{proof}

\begin{lemma} \label{Lemma: parabolic_indep_conditioning}
Let $S$ and $S'$ be two disjoint subsets of the generator set of a Coxeter system with group $W = B_n$ or $D_n$ such that $s$ and $s'$ commute for any $s \in S$ and $s' \in S'$ and $\overline{S} \cap \overline{S'} = \emptyset$. Let $w \sim \mu_q^W$. Then $w_S$ and $w_{S'}$ are independent and Mallows distributed with parameter $q$, conditioning on any distribution of values to $\overline{S}$ and $\overline{S'}$ (and thus also without the conditioning). Moreover, $w_S$, $w_{S'}$, and $\{w(i)| i \in \overline{S'}\}$ are mutually independent.
\end{lemma}
\begin{proof}
First, note that $W_{S \cup S'} = W_S \times W_{S'}$ by the commutativity assumption, so $w_S$ and $w_{S'}$ are independent and Mallows distributed with parameter $q$ by Lemma \ref{Lemma: general_parabolic_indep}. Furthermore, Lemma \ref{First_indep_result} implies that $w_{S \cup S'} \sim \mu_q^{W_{S\cup S'}}$ as well. Let $A$ be the event that $w$ sends $\overline{S}$ and $\overline{S'}$ to two fixed disjoint subsets of $[\pm n]$ (that have been previously specified). Then $A$ is clearly invariant under right multiplication by elements in $W_{S \cup S'}$ (since it only permutes values among indices in $\overline{S}$ and $\overline{S'}$, never putting new values in this set of indices). Hence, Lemma \ref{First_indep_result} implies that $w_S$ and $w_{S'}$ are independent of $A$. 
\end{proof}

\begin{lemma} \label{inverse_indep}
Let $S$ and $S'$ be connected subsets of the generator set of a Coxeter group $W = B_n$ or $D_n$. Then conditional on
$$
\{w(i)|i \in \overline{S}\} \cap \overline{S'}, \ \{-w(i)|i \in \overline{S}\} \cap \overline{S'}
$$
each being empty or containing one element, $w_S$ and $(w^{-1})_{S'}$ are independent and Mallows distributed.
\end{lemma}
\begin{proof}
Let $A$ be the event that each of the two intersections above are at most singletons. That is,
$$
A = \{w : \ |\{w(i)|i \in \overline{S}\} \cap \overline{S'}|, \ |\{-w(i)|i \in \overline{S}\} \cap \overline{S'}| \ \le \ 1\}.
$$
We claim that $A$ and 
$$
A \cap \{w: (w^{-1})_{S'} = w_0\}
$$
are invariant events under right multiplication by elements of $W_S$. It is clear that $A$ is invariant since such multiplication simply permutes the values among indices of $\overline{S}$ among themselves. Such multiplication also keeps $(w^{-1})_{S'}$ invariant since the order of the values among indices in $\overline{S}$ has no effect on the order of values in $\overline{S'}$. By Lemma \ref{First_indep_result}, $w_S$ is independent of $A$. Thus, again by Lemma \ref{First_indep_result},
\begin{align*}
    \Prob(\{w_S = w_1\} \cap \{(w^{-1})_{S'} = w_0\} | A) &= \Prob(\{w_S = w_1\} \cap \{(w^{-1})_{S'} = w_0\} \cap A)\big/\Prob(A) \\
    &= \Prob(w_S = w_1)\Prob(\{(w^{-1})_{S'} = w_0\} \cap A)\big/\Prob(A) \\
    &= \Prob(w_S = w_1 | A)\Prob(\{(w^{-1})_{S'} = w_0\} | A)
\end{align*}
so $w_S$ and $(w^{-1})_{S'}$ are independent conditioned on $A$.
\end{proof}

\subsection{Probability Bounds}

We now upper bound the probability of observing a specific sequence of values at a some predetermined set of indices. We first consider this when $W = B_n$.

\begin{lemma}[Probability Bound for $B_n$] \label{lem: set_indices_B}
For some $C \subseteq [n]$, let $a_i \in [\pm n]$ be distinct for $i \in C$. Let $w'$ be the signed permutation with $w'(i) = a_i$ for all $i \in C$ and that is ordered elsewhere, i.e. $w(i)<w(j)$ for $i,j \not\in C \cup (-C)$. Then, for $w \sim \mu_q^{B_n}$,
$$
\mathbb{P}(w(i) = a_i \forall i \in C) \le \frac{q^{\ell(w')}[2n-2|C|]_q!!}{[2n]_q!!}
$$
\end{lemma}
\begin{proof}
Let $\overline{w}$ denote the signed permutation in $B_{n-|C|}$ given by the relative order of $w(i)$ for $i \not\in C \cup (-C)$. That is, let $V$ be the set of values among the indices in $[\pm n] \setminus (C \cup (-C))$ in $w$. Then, let $V^+$ be the positive values in $V$ and $V^-$ be the negative values in $V$. Let $i$ be an index of $\overline{w}$ and $i'$ be the corresponding index in $[\pm n] \setminus (C \cup (-C))$. We let $\overline{w}(i) = j$ for $j>0$ if and only if $w(i)$ equals the $j$th smallest element of $V^+$. Similarly, let $\overline{w}(i') = -j$ for $j>0$ if and only if $w(i')$ equals the $j$th largest element of $V^+$. It is not hard to see that $w(i_1') < w(i_2')$ if and only if $\overline{w}(i_1) < \overline{w}(i_2)$.

Now, $w \mapsto \overline{w}: \{w: w(i)=a_i\} \to B_{n-|C|}$ is a bijection. Let $\ell_1(w)$ denote the number of inversions among indices in $C^c$, i.e.
$$
\ell_1(w) = |\{(i,j) \in C^c \times C^c: i<j, w(i)>w(j)\}| + |\{(i,j) \in C^c \times C^c: i\le j, w(-i)>w(j)\}|.
$$
Notice that $\ell_1(w) = \ell(\overline{w})$. Let $\ell_2(w)$ denote the rest of the inversions, that is
$$
\ell_2(w) = |\{(i,j) \in [n] \times [n] \setminus C^c \times C^c: i<j, w(i)>w(j)\}| $$$$ \quad \quad \quad \quad + |\{(i,j) \in [n] \times [n] \setminus C^c \times C^c: i\le j, w(-i)>w(j)\}|.
$$
We have that $w'$ minimizes $\ell_2$ among $\{w: w(i)=a_i\}$ since switching any two values at positive indices in $C^c$ to put them in order can only decrease the number of inversions between $C$ and $C^c$. Similarly, switching $w(i)$ and $w(-i)$ for $i \not\in C$ can only decrease $\ell_2(w)$. By definition, $\ell_1(w') = 0$. That is, there does not exist $(i,j) \in C^c \times C^c$ such that $i<j$ and $w'(i)>w'(j)$ or $w'(-i)>w'(j)$, and all $w(i)$ are positive. We then have that
\begin{align*}
    \mathbb{P}(w(i) = a_i \forall i \in C) &= \sum_{w(i)=a_i} \frac{q^{\ell(w)}}{[2n]_q!!} \\
    &= \sum_{w(i)=a_i} \frac{q^{\ell_1(w) + \ell_2(w)}}{[2n]_q!!} \\
    &\le \sum_{w(i)=a_i} \frac{q^{\ell(\overline{w}) + \ell(w')}}{[2n]_q!!} \\
    &= \frac{q^{\ell(w')}[n-|C|]_q!!}{[n]_q!!}
\end{align*}
\end{proof}

Now, in order to show the result for $W = D_n$, we tweak the definition of $w'$ (where $w \sim \mu_q^{D_n}$). As before, fix some $C \subseteq [n]$, let $a_i \in [\pm n]$ for $i \in C$ such that no two $a_i$ are negatives of each other. Let $w' \in S_n^D$ be the signed permutation that has $w'(i) = a_i$ for all $i \in C$ and is "maximally" well-ordered elsewhere. More precisely, if the number of negative $a_i$ is even, then we may define $w'$ to be exactly such that $w(i)<w(j)$ for all $-n \le i<j \le n$ with $i,j \not\in C \cup (-C)$. If the number of negative $a_i$ is odd, we define $w'$ in the same way, except that if $k$ is the smallest integer in $[n] \setminus C$, we have $w(k)$ negative (i.e., the middle two values are switched).

\begin{lemma}[Probability Bound for $D_n$] \label{set_indices}
Let $w$ be Mallows distributed in $D_n$, fix $C \subseteq [n]$, and define $w'$ as above. We have,
$$
\mathbb{P}(w(i)=a_i \forall i \in C) \le \frac{q^{\ell(w')}[n-|C|]_q[2n-2|C|-2]_q!!}{[n]_q[2n-2]_q!!}
$$
\end{lemma}
\begin{proof}
Let $\overline{w}$ denote the signed permutation in $B_{n-|C|}$ given by the relative order of $w(i)$ for $i \not\in C \cup (-C)$. First, assume that an even number of $a_i$ are negative. We then have that $w \mapsto \overline{w}: \{w: w(i)=a_i\} \to D_{n-|C|}$ is a bijection. Let $\ell_1(w)$ denote the number of D-inversions among indices in $C^c = [n] \setminus C$, i.e.
$$
\ell_1(w) = |\{(i,j) \in C^c \times C^c: i<j, w(i)>w(j)\}| + |\{(i,j) \in C^c \times C^c: i< j, w(-i)>w(j)\}|.
$$
Note that $\ell_1(w) = \ell(\overline{w})$. Let $\ell_2(w)$ denote the rest of the D-inversions, i.e.
$$
\ell_2(w) = |\{(i,j) \in [n] \times [n] \setminus C^c \times C^c: i<j, w(i)>w(j)\}| $$$$ \quad \quad \quad \quad + |\{(i,j) \in [n] \times [n] \setminus C^c \times C^c: i< j, w(-i)>w(j)\}|.
$$
We have that $w'$ minimizes $\ell_2$ among $\{w: w(i)=a_i\}$ since switching any two values at positive indices in $C^c$ to put them in order can only decrease the number of inversions between $C$ and $C^c$. Similarly, switching $w(i)$ and $w(-i)$ for $i \not\in C$ can only decrease $\ell_2(w)$. By definition, $\ell_1(w') = 0$. That is, there does not exist $(i,j) \in C^c \times C^c$ such that $i<j$ and $w'(i)>w'(j)$ or $w'(-i)>w'(j)$, and all $w'(i)$ are positive. We then have that
\begin{align*}
    \mathbb{P}(w(i) = a_i \forall i \in C) &= \sum_{w(i)=a_i} \frac{q^{\ell(w)}}{[2n-2]_q!![n]_q} \\
    &= \sum_{w(i)=a_i} \frac{q^{\ell_1(w) + \ell_2(w)}}{[2n-2]_q!![n]_q} \\
    &\le \sum_{w(i)=a_i} \frac{q^{\ell(\overline{w}) + \ell(w')}}{[2n-2]_q!![n]_q} \\
    &= \frac{q^{\ell(w')}[n-|C|]_q[2n-2|C|-2]_q!!}{[n]_q[2n-2]_q!!}.
\end{align*}
It remains to deal with the case when an odd number of $a_i$ are positive. If we define $\overline{w}$ the same we as above (that is, completely well-ordered outside of $C$), then $w \mapsto \overline{w}: \{w: w(i)=a_i\} \to B_{n-{C}} \setminus D_{n-{C}}$ is a bijection. However, there is a clear bijection between $B_{n-{C}} \setminus D_{n-{C}}$ and $D_{n-{C}}$ that involves switching the values $1$ and $-1$. Thus, we may let $\overline{w} \in D_{n-C}$ via this bijection. Moreover, this bijection preserves the number of D-inversions in $\overline{w}$. In particular, 
$$
\sum_{w(i)=a_i} q^{\ell(\overline{w})} = [n-|C|][2n-2|C|-2]_q!!.
$$
We can then define $\ell_1(w)$ and $\ell_2(w)$ as before and again observe that $\ell_1(w)=\ell(\overline{w})$, $\ell_1(w')=0$. We also claim that $w'$ minimizes $\ell_2$ among $\{w:w(i)=a_i \forall i \in C\}$. To see this, let $w''$ denote the result of switching the values of $w'$ at $k$ and $-k$, where $k$ is the smallest integer in $[n]\setminus C$. Then, $w''$ is not in $D_n$, but it minimizes $\ell_2$ over all signed permutations. In general, this switch provides a bijection between $\{w \in D_n: w(i)=a_i\}$ and $\{w \in B_n \setminus D_n: w(i)=a_i\}$ and subtracts a certain number of D-inversions to $\ell_2$ that is equal to
\begin{equation} \label{number_of_new_D_inversions}
    |\{|i| \in C: -k < i < k, |w(i)| < |w(k)|\}|.
\end{equation}
However, $w'$ minimizes the magnitude of $w(k)$, so (\ref{number_of_new_D_inversions}) is minimized for $w'$. Since $w''$ has minimum $\ell_2$ over all signed permutations, and the bijection $w'' \mapsto w'$ adds a minimal number of D-inversions, we have that $\ell_2(w')$ is minimal for all type D signed permutations. We can then perform the same computation as above:
\begin{align*}
    \mathbb{P}(w(i) = a_i \forall i \in C) &= \sum_{w(i)=a_i} \frac{q^{\ell(w)}}{[2n-2]_q!![n]} \\
    &= \sum_{w(i)=a_i} \frac{q^{\ell_1(w) + \ell_2(w)}}{[2n-2]_q!![n]} \\
    &\le \sum_{w(i)=a_i} \frac{q^{\ell(\overline{w}) + \ell(w')}}{[2n-2]_q!![n]} \\
    &= \frac{q^{\ell(w')}[n-|C|][2n-2|C|-2]_q!!}{[n][2n-2]_q!!}.
\end{align*}
\end{proof}

\section{Size-bias coupling for two-sided descents}
\label{sec: coupling}
In this section, size-bias couplings and their relation to Stein's method are reviewed. Then a size-bias coupling for the two-sided descent statistic is constructed, based on a coupling due to Goldstein \cite{G05} which was also used by Conger and Viswanath \cite{CV07} to study descents in multiset permutations.

\subsection{Size-bias Coupling and Stein's Method}
Let $X$ be a non-negative discrete random variable with positive mean. Say that $X^*$ has the \emph{size-bias distribution} with respect to $X$ if
\begin{equation*}
    \Prob(X^*=x)=\frac{x\Prob(X=x)}{\E(X)}.
\end{equation*}
A \emph{size-bias coupling} is a pair $(X,X^*)$ of random variables, defined on the same probability space such that $X^*$ has the size-bias distribution of $X$.

The following version of Stein's method using size-bias coupling is the main tool used to prove Theorems \ref{thm: main theorem 1, B_n} and \ref{thm: main theorem 2, D_n}.

\begin{theorem}[{\cite[Theorem 1.1]{GR96}}]
\label{thm: size-bias steins}
Let $X$ be a non-negative random variable, with $\E(X)=\mu$ and $\Var(X)=\sigma^2$ and let $(X,X^*)$ be a size-bias coupling. Let $Z$ denote a standard normal random variable. Then for all piecewise continuously differentiable functions $h:\mathbf{R}\to \mathbf{R}$,
\begin{equation*}
    \left|\E h\left(\frac{X-\mu}{\sigma}\right)-\E h(Z)\right|\leq 2\|h\|_\infty \frac{\mu}{\sigma^2}\sqrt{\Var \E(X-X^*|X)}+\|h'\|_\infty \frac{\mu}{\sigma^3}\E(X-X^*)^2.
\end{equation*}
\end{theorem}

The next result is an alternative version of Stein's method that will be used to prove Theorem \ref{thm: Wasserstein 1-distance}.

\begin{theorem}[{\cite[Theorem 3.20]{R11b}}]
\label{thm: size-bias steins 1-distance}
Let $X, X^s, Z$ be as in Theorem \ref{thm: size-bias steins}. Then,
\begin{equation*}
    d_W\left( \frac{X-\mu}{\sigma}, Z \right) \le \sqrt{\frac{2}{\pi}} \frac{\mu}{\sigma^2}\sqrt{\Var \E(X-X^*|X)} + \frac{\mu}{\sigma^3}\E(X-X^*)^2.
\end{equation*}
\end{theorem}

This theorem becomes a central limit theorem if both error terms can be controlled. For the application to two-sided descents, the relevant error term will be the first one, because $|X-X^*|$ will be bounded, and $\mu,\sigma^2$ are both of order $n$.

The following construction of a size-bias coupling for sums of random variables is also crucial.

\begin{lemma}[{\cite[Lemma 2.1]{BRS89}}]
\label{lem: size-bias characterization}
Let $X=\sum X_i$ be a sum of non-negative random variables. Let $I$ be a random index with the distribution $\Prob(I=i)=\E X_i/\sum \E X_j$. Let $X^*=\sum X_i'$ where conditional on $I=i$, $X_i'$ has the size-bias distribution of $X_i$ and
\begin{equation*}
    \Prob((X_1',\dotsc,X_n')\in A|I=i,X_i'=x)=\Prob((X_1,\dotsc,X_n)\in A|X_i=x).
\end{equation*}
Then $X^*$ has the size-bias distribution of $X$.
\end{lemma}

\subsection{Construction}

Let $w \in W$ be an arbitrary element of an arbitrary Coxeter system $W$ with generator set $\{s_0, s_1, \dots, s_{n-1}\}$. For each $i$, define
$$
w_i^* = w_i^{+*} = \begin{cases} w & \text{if } \des_i(w) = 1 \\ ws_i & \text{if } \des_i(w) = 0 \end{cases}
$$
and $w_i^{-*} = ((w^{-1})^*_i)^{-1}$. That is, obtaining $w_i^*$ from $w$ ensures a right descent at $s_i$ while taking $w_i^{-*}$ ensures a left descent at $s_i$. Now, let $w^*$ be the random variable (taking values in $W$) that is obtained by taking uniformly at random a generator $s_i$ and uniformly at random a sign $\pm$ and then letting $w^* = w_i^{\pm *}$. We claim that this gives a size-bias coupling for two-sided descents.

\begin{proposition} \label{prop: size-bias coupling}
Let $W$ be a finite Coxeter group and $w \in W$ be Mallows distributed. Then the random variable $\des(w^*)+\des((w^*)^{-1})$ has the size-bias distribution of $\des(w)+\des(w^{-1})$.
\end{proposition}
\begin{proof}
The proof proceeds by showing that it coincides with the construction given by Proposition \ref{lem: size-bias characterization}.

We have,
\begin{equation*}
    \des(w)+\des(w^{-1})=\sum _{i=0}^{n-1}\des_i(w)+\des_i(w^{-1}).
\end{equation*}
Note that each $\des_i(w)$ and $\des_i(w^{-1})$ is identically distributed (although they are not independent), with $\mathbb{P}(\des_i(w)=1)=q/(1+q)$. The size-bias distribution of $\des_i(w)$ is the constant $1$. (It is not hard to show that this is the case for any Bernoulli random variable.) Then $\des(w^*)+\des((w^*)^{-1})$ can be described as picking a summand uniformly at random, and replacing it with its size-bias distribution.

Thus, it remains to check that conditional on picking some index $i$, and either $w$ or $w^{-1}$, which by symmetry can be taken to be $w$, the distribution of all other summands $\des_j(w_i^*)$, $j\neq i$, and $\des_j((w_i^*)^{-1})$ is the same as that of $\des_j(w)$, $j\neq i$ and $\des_j(w^{-1})$ conditioned to have $\des_i(w)=1$ (note that the conditioning gives $w^*=w_i^*$).

In fact, it can be shown that the distribution of $w_i^*$ is equal to that of $w$ conditioned to have $\des_i(w)=1$. To see this, note that
\begin{equation*}
\begin{split}
    \Prob(w_i^*=w_0)&=\Prob(w=w_0|\des_i(w)=1)\Prob(\des_i(w)=1)
    \\&\qquad+\Prob(w_i^*=w_0|\des_i(w)=0)\Prob(\des_i(w)=0)
\end{split}
\end{equation*}
and so it suffices to show that the distribution of $w_i^*$ given that $\des_i(w)=0$ and the distribution of $w$ given that $\des_i(w)=1$ are the same. Note that, conditioning on the event that $\des_i(w)=0$, $w_i^*=ws_i$. But the map $w\mapsto ws_i$ is a bijection from $\{w|\des_i(w)=0\}$ to $\{w|\des_i(w)=1\}$ and $\Prob(ws_i)=q\Prob(w)$ if $\des_i(w)=0$. Thus, the relative probabilities are unchanged and so the distributions are the same.
\end{proof}

\subsection{Decomposition of the Variance Term}

The idea is to now use this coupling to apply Theorem \ref{thm: size-bias steins} to obtain a quantitative bound. The main focus will be on the first error term, which will be called the \emph{variance term}.

First, note that 
\begin{equation*}
\begin{split}
&\Var \E(\des(w)+\des(w^{-1})-\des(w^*)-\des((w^*)^{-1})|\des(w^{-1})+\des(w))
\\\leq &\Var(\E(\des(w)+\des(w^{-1})-\des(w^*)-\des((w^*)^{-1})|w)).
\end{split}
\end{equation*}
Writing
\begin{equation*}
    E(\des(w^*)|w)=\frac{1}{2(n-1)}\sum _{i,\pm}\des(w_{i}^{\pm *})
\end{equation*}
and similarly for $(w^*)^{-1}$, the variance can be written as
\begin{equation*}
    \Var\left(\frac{1}{2n}\left(\Sigma_1+\Sigma_2+\Sigma_3+\Sigma_4\right) \right)  = \frac{1}{4n^2} \Var\left(\Sigma_1+\Sigma_2+\Sigma_3+\Sigma_4\right)
\end{equation*}
where
\begin{align*}
    \Sigma_1&=\sum _i \des(w)-\des(w_i^*)
    \\\Sigma_2&=\sum _i \des(w)-\des(w_{i}^{-*})
    \\\Sigma_3&=\sum _i \des(w^{-1})-\des((w_i^*)^{-1})
    \\\Sigma_4&=\sum_i \des(w^{-1})-\des((w_{i}^{-*})^{-1}).
\end{align*}

The variance can be split into 16 types of covariance terms. These can be reduced by symmetry and the fact that $w$ and $w^{-1}$ have the same distribution to 6 types of terms. For example,
\begin{equation*}
\begin{split}
    &\Cov(\des(w)-\des(w_i^*),\des(w)-\des(w_j^*))
    \\=&\Cov(\des(w^{-1})-\des((w_{i}^{-*})^{-1}),\des(w^{-1})-\des((w_{j}^{-*})^{-1})).
\end{split}
\end{equation*}
The types of terms are summarized in Table \ref{tab:term types} along with their multiplicities.

\begin{table}
\caption{The types of terms and multiplicities in the variance bound}
\begin{center}
\begin{tabular}{c|c|c}
     Type & Multiplicity & Term \\\hline
     1&2& $\sum \Cov(\des(w)-\des(w_i^*),\des(w)-\des(w_j^*))$\\
     2&4& $\sum \Cov(\des(w)-\des(w_i^*),\des(w^{-1})-\des((w_j^*)^{-1}))$\\
     3&4& $\sum \Cov(\des(w)-\des(w_i^*),\des(w)-\des(w_{j}^{-*}))$\\
     4&2& $\sum \Cov(\des(w)-\des(w_i^*),\des(w^{-1})-\des((w_{j}^{-*})^{-1}))$\\
     5&2& $\sum \Cov(\des(w)-\des(w_{i}^{-*}),\des(w)-\des(w_{j}^{-*}))$\\
     6&2& $\sum \Cov(\des(w)-\des(w_{j}^{-*}),\des(w^{-1})-\des((w_i^*)^{-1}))$
\end{tabular}
\end{center}
\label{tab:term types}
\end{table}

\section{Covariance Bounds: Types 1,2,3, and 4} \label{sec: easy cov}

\subsection{Results on the Covariance Terms}

\begin{lemma} \label{Lemma: Uniform Bound Difference}
For all $i$ and if $W = B_n$ or $D_n$, $|\des(w) - \des(w_i^*)| \le 3$.
\end{lemma}
\begin{proof}
From Lemma \ref{Lemma: Difference is parabolic},
$$
\des(w) - \des(w_i^*) = \sum_{s_j \in T_i} (\des_j(w) - \des_j(w_i^*))
$$
and the result follows if $|T_i| \le 3$, which is the case except when $s_i = s_2$ and $W = D_n$. However, in this case, we can only have that $\des(w) - \des(w_2^*)$ is of magnitude $4$ only if $w_2^* = ws_2$ and each $\des_j(w) - \des_j(w_2^*) = -1$. But, this means that $w(3) > w(2)$ and $w(2) > w(4) > w(3)$ (a descent is induced at $s_3$). This is impossible, so the result follows.
\end{proof}

\begin{lemma} \label{Lemma: Uniform Bound Inverse}
For any $W$, we have $|\des(w)-\des(w_i^{-*})| \le 1$.
\end{lemma}
\begin{proof}
It suffices to show that multiplying on the left by $s_i$ changes at most one descent. That is, suppose that $\text{des}_i(w^{-1}) = 0$. Then, if $w$ has a reduced expression $w = s^1s^2 \dots s^k$, $s_iw$ has reduced expression 
$$
s_iw = s_is^1s^2 \dots s^k.
$$
If $w$ has a right descent at $s_j$, then $\ell(s_iws_j) \le \ell(ws_j) + 1 = \ell(w) = \ell(s_iw) - 1$ by \cite[Proposition 1.4.2 (iii)]{BB05}, so $s_iw$ also has a right descent at $s_j$. If $w$ does not have a right descent at $s_j$, then we claim
$$
s_iws_j = s_is^1s^2 \dots s^ks_j
$$
is a reduced word for all but at most one $j$. Suppose that it is not. Then, we can cancel two generators from the above word by the Deletion 
Property (see \cite[Proposition 1.4.7]{BB05}). This must include $s_i$ or else $s^1s^2 \dots s^ks_j$ is not a reduced word for $ws_j$ and $\ell(ws_j) < \ell(w)$, a contradiction. Similarly, it must also include $s_j$ or else $s_iw = s_is^1s^2 \dots s^k$ is not a reduced word, but $w$ does not have a left descent at $s_i$. It follows that 
$$
s_is^1s^2 \dots s^ks_j = s^1s^2 \dots s^k \iff s_is^1s^2 \dots s^k = s^1s^2 \dots s^ks_j.
$$
There is at most one $s_j$ that solves the above equation, namely $s_j = w^{-1}s_iw$.
\end{proof}

\subsection{Sum Bounds}

Now, by Lemma \ref{Lemma: general_parabolic_indep} and the above, we have that
$$
\Cov(\des(w) - \des(w_i^*), \des(w) - \des(w_j^*)) = 0
$$
if all elements from $T_i$ and $T_j$ commute. It is then sufficient that $i$ and $j$ are a distance of greater than $3$ apart in the Coxeter graph of $W$. We claim the following:

\begin{lemma} [Type 1 Bound] \label{Lemma: Type 1 Bound}
In both cases $W = B_n, D_n$,
$$
|\sum_{i,j \ge 0} \Cov(\des(w) - \des(w_i^*), \des(w) - \des(w_j^*))| \le 63(n-1).
$$
\end{lemma}
\begin{proof}
For $W=B_n$, we have that $\Cov(\des(w) - \des(w_i^*), \des(w) - \des(w_j^*)) = 0$ if $|i-j| > 3$ since the Coxeter graph is just a straight line. Thus, a covariance is nonzero only if $|i-j| \le 3$. It is not hard to see that there are $n$ pairs $(i,j)$ such that $|i-j| = 0$ and $2(n-k)$ pairs $(i,j)$ such that $|i-j| = k > 0$. So there are
$$
n + 2(n-1) + 2(n-2) + 2(n-3) = 7n - 12 \le 7(n-1)
$$
such pairs. Moreover, for any such pair,
$$
|\Cov(\des(w) - \des(w_i^*), \des(w) - \des(w_j^*))| \le 9
$$
since each term inside the covariance is bounded absolutely by $3$ by Lemma \ref{Lemma: Uniform Bound Difference}. It follows that the total sum of covariances is bounded above by $63(n-1)$.

In the case that $W=D_n$, we have that $\Cov(\des(w) - \des(w_i^*), \des(w) - \des(w_j^*)) = 0$ if $|i-j| > 3$ for $i,j > 0$. Moreover, the only generators within a distance of three from $s_0$ are $s_1, s_2, s_3, s_4$. Thus, the number of pairs $(i,j)$ such that the covariance may not be $0$ is at most
$$
(n-1) + 2(n-2) + 2(n-3) + 2(n-4) + 9 = 7n-10 \le 7(n-1).
$$
Moreover, Lemma \ref{Lemma: Uniform Bound Difference} implies that for any such pair
$$
|\Cov(\des(w) - \des(w_i^*), \des(w) - \des(w_j^*))| \le 9.
$$
It follows that the total sum of covariances is bounded above by $63(n-1)$.
\end{proof}

\begin{lemma}[Type 2 Bound] \label{Lemma: Type 2 Bound}
In both cases $W = B_n, D_n$,
$$
\sum_{i,j} \Cov(\des(w)-\des(w_i^*),\des(w^{-1})-\des((w_j^*)^{-1})) \le 63(n-1)
$$
\end{lemma}
\begin{proof}
Consider the case $W = B_n$. For $j > 0$, note that $\des_i(w^{-1}) - \des_i((w_j^*)^{-1})$ is equal to $-1$ if and only if $w(j) < w(j+1)$ and $|w(j) - w(j+1)| = 1$ or $(w(j), w(j+1)) = \{-1,1\}$, and is $0$ otherwise. This is because reverse sorting at position $j$ only changes $w$ if it is well-ordered at $j$, and swapping $j, j+1$ in $w^{-1}$ only affects descents if $j, j+1$ are adjacent in $w^{-1}$ (i.e. if $\overline{\{s_k\}} = \{j, j+1\}$ for some $k$). Thus, $\des(w^{-1}) - \des((w_j^*)^{-1})$ depends solely on $w_{\{s_j\}}$ and the set $\{w(j), w(j+1)\}$. Then by Lemma \ref{Lemma: parabolic_indep_conditioning}, since $\des(w)-\des(w_i^*)$ depends only on $w_{T_i}$, if $T_i \cap T_j = \emptyset$ and $\overline{T_i} \cap \overline{\{s_j\}} = \emptyset$ then $\des(w)-\des(w_i^*)$ and $\des(w^{-1}) - \des((w_j^*)^{-1})$ are independent. This occurs if $|i-j| \ge 3$ and $i,j \ge 1$. The number of pairs $(i,j)$ for which this is not the case is
$$
n-1 + 2(n-2) + 2(n-3) + (2(n-1) + 1) = 7n - 12 \le 7(n-1).
$$
Hence, 
$$
\sum_{i,j} \Cov(\des(w)-\des(w_i^*),\des(w^{-1})-\des((w_j^*)^{-1})) \le 63(n-1)
$$
since each term is bounded by $9$. The $D_n$ case is essentially identical.
\end{proof}

We use the following probability result:

\begin{lemma}
\label{lem: cond ind bound}
Let $X$ and $Y$ be random variables with $|X|\leq C_1$ and $|Y|\leq C_2$. Let $A$ be some event such that conditional on $A$, $X$ and $Y$ are uncorrelated. Then
\begin{equation*}
\begin{split}
    |\Cov(X,Y)|\leq& 4C_1C_2\Prob(A^c).
\end{split}
\end{equation*}
\end{lemma}

\begin{lemma}[Type 3 Bound] \label{Lemma: Type 3 Bound}
In both cases $W = B_n, D_n$,
$$
\sum_{i,j} \Cov(\des(w)-\des(w_i^*),\des(w)-\des(w_{j}^{-*})) \le 306(n-1) + 9.
$$
\end{lemma}
\begin{proof}
Consider a term $\Cov(\des(w)-\des(w_i^*),\des(w)-\des(w_{j}^{-*}))$ for $i,j > 0$. If we condition on the event
$$
A = \{w(\overline{T_i}) \cap \overline{\{s_j\}} = w(\overline{T_i}) \cap (-\overline{\{s_j\}}) = \emptyset\} = \{w | w(a) \neq \pm j, \pm(j+1) \ \forall i-1 \le a \le i+2\}
$$
then this covariance term is $0$ by Lemma \ref{inverse_indep} since $\des(w)-\des(w_i^*)$ is a function of $w_{T_i}$ and $\des(w)-\des(w_{-j}^*)$ is a function of $(w^{-1})_{\{s_j\}}$. Thus, by Lemma \ref{lem: cond ind bound}
$$
|\Cov(\des(w)-\des(w_i^*),\des(w)-\des(w_{j}^{-*}))| \le 36\Prob(A^c) \le 36\sum_{k = i-1}^{i+2} \sum_{l = j}^{j+1} \Prob(w(k) = l).
$$
Thus,
\begin{align*}
    \sum_{i,j \ge 1} \Cov(\des(w)-\des(w_i^*),\des(w)-\des(w_{j}^{-*})) &\le \sum_{i,j \ge 1} 36\sum_{k = i-1}^{i+2} \sum_{l = j}^{j+1} \Prob(w(k) = \pm l) \\
    &\le 288 \sum_{i,j} \Prob(w(i) = \pm j) \\
    &= 288(n-1).
\end{align*}
Now, there are $2n-1$ other terms in the sum, so the total sum is bounded above by
$$
288(n-1) + 9(2n-1) = 306(n-1) + 9.
$$
\end{proof}

\begin{lemma}[Type 4 Bound] \label{Lemma: Type 4 Bound}
In both cases $W = B_n, D_n$,
$$
\sum_{i,j} \Cov(\des(w)-\des(w_i^*),\des(w^{-1})-\des((w_{-j}^*)^{-1})) \le 594(n-1) + 9.
$$
\end{lemma}
\begin{proof}
First suppose that $i,j > 0$. Then $\des(w)-\des(w_i^*)$ is a function of $w_{T_i}$ and $\des(w^{-1})-\des((w_{-j}^*)^{-1})$ is a function of $(w^{-1})_{T_j}$. Thus, if we condition on the event
$$
A = \{w| \{|w(i-1)|, |w(i)|, |w(i+1)|, |w(i+2)|\} \cap \{j-1, j, j+1, j+2\} = \emptyset\}
$$
the covariance term becomes $0$ by Lemma \ref{inverse_indep}. Hence, by Lemma \ref{lem: cond ind bound}, we have
$$
\Cov(\des(w)-\des(w_i^*),\des(w^{-1})-\des((w_{-j}^*)^{-1})) \le 36\Prob(A^c) \le 36\sum_{k = i-1}^{i+2} \sum_{l = j-1}^{j+2} \Prob(w(k) = \pm l).
$$
Thus,
\begin{align*}
    \sum_{i,j \ge 1} \Cov(\des(w)-\des(w_i^*),\des(w)-\des(w_{j}^{-*})) &\le \sum_{i,j \ge 1} 36\sum_{k = i-1}^{i+2} \sum_{l = j-1}^{j+2} \Prob(w(k) = \pm l) \\
    &\le 576 \sum_{i,j} \Prob(w(i) = \pm j) \\
    &= 576(n-1).
\end{align*}
Now, there are $2n-1$ other terms in the sum, so the total sum is bounded above by
$$
576(n-1) + 9(2n-1) = 594(n-1) + 9.
$$
\end{proof}

\section{Covariance Bounds: Types 5,6} \label{sec: hard cov}

We wish to evaluate
$$
\sum_{i,j} \mathrm{Cov}(\mathrm{des}(w)-\mathrm{des}(w_{i}^{-*}), \mathrm{des}(w)-\mathrm{des}(w_{j}^{-*})).
$$

\subsection{$B_n$ Case}

Note that we have that $\mathrm{des}(w)-\mathrm{des}(w_{i}^{-*})$ is $0$ or $-1$ and similarly for $\mathrm{des}(w)-\mathrm{des}(w_{j}^{-*})$, so that, for positive $i,j$,
\begin{align*}
    \mathrm{Cov}(\mathrm{des}(w)-\mathrm{des}(w_{i}^{-*}), \mathrm{des}(w)-\mathrm{des}(w_{j}^{-*})) &= \mathbb{E}((\mathrm{des}(w)-\mathrm{des}(w_{i}^{-*}))(\mathrm{des}(w)-\mathrm{des}(w_{j}^{-*}))) \\
    &\le \mathbb{P}\left( \substack{w(i+1)-w(i) = 1 \\ w(j+1)-w(j) = 1} \right)
\end{align*}
since $\mathrm{des}(w)-\mathrm{des}(w_{i}^{-*}) = -1$ if and only if $w(i+1)-w(i) = 1$.

\subsubsection{$q \approx 1$}

\begin{lemma}
Let $|i-j|>1$, for $i,j \in [n]$, and $n \ge 4$. Then,
$$
\mathbb{P}\left( \substack{w(i+1)-w(i) = 1 \\ w(j+1)-w(j) = 1} \right) \le \frac{22(1-q)^2}{(1-q^{2n-6})^2}.
$$
\end{lemma}
\begin{proof}
We have
\begin{align*}
    \mathbb{P}\left( \substack{w(i+1)-w(i) = 1 \\ w(j+1)-w(j) = 1} \right) &= \sum_{k,l \in [\pm n] \setminus \{-1, n\} } \mathbb{P}\left( \substack{w(i)=k, w(i+1)=k+1 \\ w(j)=l, w(j+1)=l+1} \right) \\
    &\le \sum_{k,l \in [n]} \frac{q^{\max(2|i-k|+2|j-l|-4, 0)}(1-q)^4}{(1-q^{2n})(1-q^{2n-2})(1-q^{2n-4})(1-q^{2n-6})} \\
    & \quad + \sum_{k,l \in [-n]} \frac{q^{\max(2|i-k|+2|j-l|-10, 0)}(1-q)^4}{(1-q^{2n})(1-q^{2n-2})(1-q^{2n-4})(1-q^{2n-6})} \\
    & \quad + \sum_{k \in [-n], l \in [n]} \frac{q^{\max(2|i-k|+2|j-l|-7, 0)}(1-q)^4}{(1-q^{2n})(1-q^{2n-2})(1-q^{2n-4})(1-q^{2n-6})} \\
    & \quad + \sum_{k \in [n], l \in [-n]} \frac{q^{\max(2|i-k|+2|j-l|-7, 0)}(1-q)^4}{(1-q^{2n})(1-q^{2n-2})(1-q^{2n-4})(1-q^{2n-6})}
\end{align*}
by Lemma 2.5. Setting $w(i)=k,w(i+1)=k+1$ for positive $k$ induces $2|i-k|$ inversions by the pigeonhole principle. Doing this for negative $k$ induces only $2|i-k|-3$ inversions. The $-4$ is from possible double counting of inversions between $i,i+1$ and $j,j+1$ (or their negatives). Note that the number of double counts is at most $4$ since there is at most one double count between $\pm i$ and $\pm j$. In order for there to be more than one count, (assuming $i>j$) we must have $w(i)<0$ and $|w(j)|<|w(i)|$, but the $2|j-l|-3$ inversions from assigning $j,j+1$ come purely from descent with values that are of magnitude less than $l$ or values with index between $j,-j$. 

We now wish to evaluate 
$$
\sum_{k,l \in [n]} q^{\max(2|i-k|+2|j-l|-4, 0)}.
$$
We can then sum of over the following regimes:
\begin{enumerate}
    \item $k>i, l>j$ or $k>i, l<j$ or $k<i,l>j$ or $k<i,l<j$: Each of these cases yields a factorable sum of the form $(1+q^2+ \dots)(1+q^2+ \dots) \le (1+q^2+ \dots + q^{2(n-2)})^2 = \frac{(1-q^{2n-2})^2}{(1-q^2)^2}$.
    \item $k=i, l>j+1$ or $k=i, l<j-1$ or $k<i-1, l=j$ or $k>i+1, l=j$: Each of these cases yields a sum of the form $1+q^2+ \dots \le 1+q^2+ \dots + q^{2(n-2)} = \frac{1-q^{2n-2}}{1-q^2}$.
    \item $(k,l) = (i,j), (i,j-1),(i-1,j), (i+1,j), (i,j+1)$: Each case yields $1$.
\end{enumerate}
It follows that the total sum is bounded above by 
$$
4\frac{(1-q^{2n-2})^2}{(1-q^2)^2} + 4\frac{1-q^{2n-2}}{1-q^2} + 5 \le 13\frac{(1-q^{2n-2})^2}{(1-q^2)^2}.
$$

We now wish to evaluate 
$$
\sum_{k,l \in [-n]} q^{\max(2|i-k|+2|j-l|-10, 0)}.
$$
In this case we actually have $2|i-k|+2|j-l|-10 = 2i+2j-2-2(k+2)-2(l+2)$, where $k+2,l+2$ range from $0$ to $2-n$, so that the sum equals
$$
q^{2i+2j-2} (1+q^2+ \cdots + q^{2(n-2)})^2 = q^{2i+2j-2} \frac{(1-q^{2n-2})^2}{(1-q)^2}.
$$
We now wish to evaluate
$$
\sum_{k \in [-n], l \in [n]} q^{\max(2|i-k|+2|j-l|-7, 0)}.
$$
We have $2|i-k|+2|j-l|-7 = 2i-3 + 2|j-l| - 2(k+2)$. We can sum over the following regimes:
\begin{itemize}
    \item $j>l$ or $j<l$: In either case, we have a sum of the form 
    $$
    \sum_{k \in [-n]} q^{2i-1}q^{-2(k+2)}(1+q^2 + \cdots + q^{2(n-2)}) \le q^{2i-1}(1+q^2 + \cdots + q^{2(n-2)})^2 = \frac{q^{2i-1}(1-q^{2n-2})^2}{(1-q^2)^2}.
    $$
    \item $j=l, k<-2$: This yields the sum $q^{2i-1}(1+q^2+\cdots + q^{2(n-3)}) = \frac{q^{2i-1}(1-q^{2n-4})}{1-q^2}$.
    \item $j=l,k=-2$: This yields a maximum term of $1$.
\end{itemize}
The sum totals to 
$$
2\frac{q^{2i-1}(1-q^{2n-2})^2}{(1-q^2)^2} + \frac{q^{2i-1}(1-q^{2n-4})}{1-q^2} + 1 \le 4\frac{q^{2i-1}(1-q^{2n-2})^2}{(1-q^2)^2}.
$$
By similar reasoning, we have
$$
\sum_{k \in [n], l \in [-n]} q^{\max(2|i-k|+2|j-l|-7, 0)} \le 4\frac{q^{2j-1}(1-q^{2n-2})^2}{(1-q^2)^2}.
$$
It follows that the sum of the power $q$ terms in the all of the four sums at the start is bounded above by 
$$
22\frac{(1-q^{2n-2})^2}{(1-q^2)^2} \le 22\frac{(1-q^{2n-2})^2}{(1-q)^2}.
$$
We can then multiply by $$\frac{(1-q)^4}{(1-q^{2n})(1-q^{2n-2})(1-q^{2n-4})(1-q^{2n-6})} \le \frac{(1-q)^4}{(1-q^{2n-2})^2(1-q^{2n-6})^2}$$ to get the result.
\end{proof}

\begin{lemma} \label{Lemma: Type 5 Bound, B_n, q large}
    Fix $q \ge 1-(n-1)^{-1/2}$ and let $n \ge 4$. Then
    $$
    \sum_{0\le i,j \le n-1} \mathrm{Cov}(\mathrm{des}(w)-\mathrm{des}(w_{-i}^*), \mathrm{des}(w)-\mathrm{des}(w_{-j}^*)) \le 43(n-1).
    $$
\end{lemma}
\begin{proof}
We have, since $\frac{1-q}{1-q^{2n-6}}$ is a decreasing function of $q$ in $(0,1)$, 
$$
\frac{1-q}{1-q^{2n-6}} \le \frac{(n-1)^{-1/2}}{1-(1-(n-1)^{-1/2})^{2n-6}} \le \frac{5}{4}(n-1)^{-1/2}
$$
where we note that $1-(1-(n-1)^{-1/2})^{2n-6} \ge \frac{4}{5}$ for all $n \ge 4$. We now have,
\begin{align*}
    \sum_{0\le i,j \le n-1} \mathrm{Cov}(&\mathrm{des}(w)-\mathrm{des}(w_{-i}^*), \mathrm{des}(w)-\mathrm{des}(w_{-j}^*)) \\
    &\le \sum_{\substack{i,j \in [n] \\ |i-j|>1}}\mathbb{P}\left( \substack{w(i+1)-w(i) = 1 \\ w(j+1)-w(j) = 1} \right) + \sum_{\substack{i,j \in [n] \\ |i-j|\le 1}}\mathbb{P}\left( \substack{w(i+1)-w(i) = 1 \\ w(j+1)-w(j) = 1} \right) \\ 
    &\quad + 2\sum_{1 \le i \le n-1} \mathbb{P}\left( \substack{w(i+1)-w(i) = 1 \\ w(1)=1, \ w(-1)=-1} \right) + \mathbb{P}(w(1)=1, \ w(-1)=-1) \\
    &\le (n-1)^2\frac{22(1-q)^2}{(1-q^{2n-6})^2} + (3(n-1) - 2) + 2(n-1) + 1 \\
    &\le 43(n-1).
\end{align*}
\end{proof}

\subsubsection{$q << 1$}

We first control the diagonal:
\begin{lemma} \label{Lemma: B_n near diagonal}
Fix some $0\le m \le n-1$ and $q \le 1-(n-1)^{-1/2}$ with $n \ge 4$. Then 
$$
\sum_{\substack{i,j \in [n-1] \\ |i-j|<m}}\mathbb{P}\left( \substack{w(i+1)-w(i) = 1 \\ w(j+1)-w(j) = 1} \right) \le 55m(n-1)(1-q)^2 + 3(n-1).
$$
\end{lemma}
\begin{proof}
We have that $\frac{1}{1-q^{2n-6}}$ is an increasing function of $q$, so
$$
\frac{1}{1-q^{2n-6}} \le \frac{1}{1-(1-(n-1)^{-1/2})^{2n-6}} \le \frac{5}{4}.
$$
We then have that
\begin{align*}
    \sum_{\substack{i,j \in [n-1] \\ |i-j|<m}}\mathbb{P}\left( \substack{w(i+1)-w(i) = 1 \\ w(j+1)-w(j) = 1} \right) &= \sum_{\substack{i,j \in [n-1] \\ 1<|i-j|<m}}\mathbb{P}\left( \substack{w(i+1)-w(i) = 1 \\ w(j+1)-w(j) = 1} \right) + \sum_{\substack{i,j \in [n-1] \\ |i-j|<2}}\mathbb{P}\left( \substack{w(i+1)-w(i) = 1 \\ w(j+1)-w(j) = 1} \right) \\
    &\le 2m(n-1) \frac{22(1-q)^2}{(1-q^{2n-6})^2} + 3(n-1) \\
    &\le 55m(n-1)(1-q)^2 + 3(n-1)
\end{align*}
by Lemma 3.1.
\end{proof}

We will use the following result below (it is elementary to prove):
\begin{lemma} \label{lem: assisting ineq}
For any positive integer $k$,
    $$
    \sum_{r=0}^k rq^r = \frac{q}{(1-q)^2} - \frac{(k+1)q^{k+1}}{1-q} - \frac{q^{k+2}}{(1-q)^2}
    $$
\end{lemma}

\begin{lemma} \label{Lemma: B_n off-diagonal}
Fix some $2\le m \le n-1$ and $q \le 1-(n-1)^{-1/2}$ with $n \ge 6$. Let $A_i$ denote the event that $w(i+1)-w(i) = 1$. Then
$$
\sum_{|i-j|\ge m} \mathrm{Cov}(I_{A_i}, I_{A_j}) \le 44(n-1)q^{m-1} + \frac{(n-1)q^{m-1}(2m+2)}{(1-q)^2} + \frac{8(n-1)q^{m-1}}{(1-q)^3}
$$
\end{lemma}
\begin{proof}
We have
$$
\mathrm{Cov}(I_{A_i}, I_{A_j}) = \sum_{k,l \in [\pm n] \setminus \{-1, n\} } \mathbb{P}\left( \substack{w(i)=k, w(i+1)=k+1 \\ w(j)=l, w(j+1)=l+1} \right) - \mathbb{P}(w(i)=k, w(i+1)=k+1)\mathbb{P}(w(j)=l, w(j+1)=l+1).
$$
Assuming $i< j$.
$$
\mathbb{P}\left( \substack{w(i)=k, w(i+1)=k+1 \\ w(j)=l, w(j+1)=l+1} \right) = \sum_{C} \frac{q^{\ell(C)}}{[n]_q!!} \sum_{w_0 \in S_{j-i-2}} q^{\ell(w_0)} = \sum_{C} q^{\ell(C)} \frac{[j-i-2]_q!}{[n]_q!!}
$$
where $C$ ranges over all assignments to indices in $[n]$ that are less than $i$ or greater than $j+1$ (and obviously the negatives of these indices are assigned the appropriate values) and each $C$ specifies the sign of if each value that is assigned in the middle block. $\ell(C)$ denotes the number of inversions induced by the assignment $C$. That is, $\ell(C)=\ell(w_C)$ for the signed permutation $w_C$ that satisfies $C$ and is completely well-ordered on the block from $i+2$ to $j-1$. The claim follows from noting that $\ell(w)=\ell(C)+\ell(w_0)$ where $w_0$ corresponds to the way $w$ is ordered on $[i+2,j-1]$. This is because
\begin{align*}
    \ell(w) &= |\{ (i',j'): i'\in [-j',j'-1], j' \in [n], w(i')>w(j') \}| \\
    &= |\{ (i',j') \in [\pm ([n]\setminus [i+2,j-1])]^2: i'\in [-j',j'-1], j' \in [n], w(i')>w(j') \}| \\
    & \quad + |\{ (i',j') \in \pm [i+2,j-1] \times [j,n]: w(i')>w(j') \}| \\
    & \quad + |\{ (i',j') \in [\pm (i+1)] \times [i+2,j-1]: w(i')>w(j') \}| \\
    & \quad + |\{ (i',j') \in [i+2,j-1]^2: i'\le j', w(-i')>w(j') \}| \\
    & \quad + |\{ (i',j') \in [i+2,j-1]^2: i'<j', w(i')>w(j') \}| \\
    &= |\{ (i',j') \in [\pm ([n]\setminus [i+2,j-1])]^2: i'\in [-j',j'-1], j' \in [n], w_C(i')>w_C(j') \}| \\
    & \quad + |\{ (i',j') \in \pm [i+2,j-1] \times [j,n]: w_C(i')>w_C(j') \}| \\
    & \quad + |\{ (i',j') \in [\pm (i+1)] \times [i+2,j-1]: w_C(i')>w_C(j') \}| \\
    & \quad + |\{ (i',j') \in [i+2,j-1]^2: i'\le j', w_C(-i')>w_C(j') \}| \\
    & \quad + |\{ (i',j') \in [i-j-2]^2: i'<j', w(i')>w(j') \}| \\
    &= \ell(C) + \ell(w_0)
\end{align*}
where we have $|\{ (i',j') \in [i+2,j-1]^2: i'\le j', w(-i')>w(j') \}| = |\{ (i',j') \in [i+2,j-1]^2: i'\le j', w_C(-i')>w_C(j') \}|$ since flipping two adjacent values in $[i+2,j-1]$ does not change the number of such types of inversions.

We also have that
$$
\mathbb{P}(w(i)=k, w(i+1)=k+1) = \sum_{C'} \frac{q^{\ell(C')}}{[n]_q!!} \sum_{w_1 \in S_{n-i-1}} q^{\ell(w_1)} = \sum_{C'}q^{\ell(C')} \frac{[n-i-1]_q!}{[n]_q!!} 
$$
where $C'$ ranges over assignments of exact values to $[i-1]$ and signs to the other values on the right half of the permutation. Moreover, if $C''$ ranges over assignments over $[j+2,n]$, with no stipulations on the signs of other values in the right half,
$$
\mathbb{P}(w(j)=l, w(j+1)=l+1) = \sum_{C''} q^{\ell(C'')}\frac{[j-1]_q!!}{[n]_q!!}
$$
since the arrangement of values among the indices in $[\pm (j-1)]$ has no effect on the number of inversions involving indices in $[j+2,n]$. We then have
$$
\mathbb{P} \left( \substack{w(i)=k \\ w(i+1)=k+1} \right) \mathbb{P} \left( \substack{w(j)=l \\ w(j+1)=l+1} \right) = \sum_{C',C''} q^{\ell(C')+\ell(C'')}\frac{[j-1]_q!![n-i-1]_q!}{([n]_q!!)^2}.
$$
Now, consider the case that $C'$ and $C''$ can be combined to form an assignment $C$ to $[\pm n] \setminus \pm [i+2,j-1]$, i.e. that the values assigned by $C'$, $C''$ are nonintersecting subsets of $[\pm n] \setminus \pm \{k,k+1,l,l+1\}$ and $C''$ follows the sign assignments by $C'$. Suppose that, for any $w$ satisfying this assignment, $|w(i')|<w(j)$, i.e. that the numbers assigned by $C'$ are all less than those assigned by $C''$. This requires that $C''$ only assigns positive values to positive indices.

We have,
\begin{align*}
    \ell(C')+\ell(C'') &= |\{ (i',j') \in [\pm(i+1)]\times [i+1]: -j'\le i'<j', w_{C'}(i')>w_{C'}(j') \}| \\
    & \quad + |\{ (i',j') \in [\pm(i+1)]\times [i+2,n]: w_{C'}(i')>w_{C'}(j') \}| \\
    & \quad + |\{ (i',j') \in [-n,-i-2]\times [i+2,n]: -i'\le j', w_{C'}(i')>w_{C'}(j') \}| \\
    & \quad + |\{ (i',j') \in \pm [j,n] \times [j,n]: -j'\le i'<j', w_{C''}(i')>w_{C''}(j') \}| \\
    & \quad + |\{ (i',j') \in [j-1]\times [j,n]: w_{C''}(i')>w_{C''}(j') \}| \\
    &= |\{ (i',j') \in [\pm(i+1)]\times [i+1]: -j'\le i'<j', w_{C}(i')>w_{C}(j') \}| \\
    & \quad + |\{ (i',j') \in [\pm(i+1)]\times [i+2,n]: w_{C}(i')>w_{C}(j') \}| \\
    & \quad + |\{ (i',j') \in [-n,-i-2]\times [i+2,n]:-i'\le j', w_{C}(i')>w_{C}(j') \}| \\
    & \quad + |\{ (i',j') \in [j,n] \times [j,n]: -j'\le i'<j', w_{C}(i')>w_{C}(j') \}| \\
    & \quad + |\{ (i',j') \in [i+2,j-1]\times [j,n]: w_{C}(i')>w_{C}(j') \}| \\
    &= |\{ (i',j') \in [\pm n] \times [n]: -j'\le i'<j', w_{C}(i')<w_{C}(j') \}| \\
    &= \ell(C)
\end{align*}
since $w_C$ is well ordered on $[i+2,j-1]$.

We can then express
\begin{align*}
    \mathbb{P} \left( \substack{w(i)=k \\ w(i+1)=k+1} \right) \mathbb{P} \left( \substack{w(j)=l \\ w(j+1)=l+1} \right) &= \sum_{C}q^{\ell(C)} \frac{[j-1]_q!![n-i-1]_q!}{([n]_q!!)^2} \\
    & \quad + \sum_{C}q^{\ell'(C)} \frac{[j-1]_q!![n-i-1]_q!}{([n]_q!!)^2} \\
    &\quad + \sum_{C',C''} q^{\ell(C')+\ell(C'')}\frac{[j-1]_q!![n-i-1]_q!}{([n]_q!!)^2}
\end{align*}
where the last sum is over $C'$ and $C''$ that intersect (i.e. $C''$ assigns one of $k,k+1$, $C'$ assigned one of $l,l+1$, or $C',C''$ assigned the same value). The second sum is over $C$ where we do not have $|w(i')|<w(j')$ for all $0<i'<i+2<j-1<j'$. We can then write
\begin{align*}
    \mathbb{P}\left( \substack{w(i)=k, w(i+1)=k+1 \\ w(j)=l, w(j+1)=l+1} \right) - &\mathbb{P} \left( \substack{w(i)=k \\ w(i+1)=k+1} \right) \mathbb{P} \left( \substack{w(j)=l \\ w(j+1)=l+1} \right) \\
    &\le \sum_C q^{\ell(C)} \frac{[j-i-2]_q![n]_q!! - [n-i-1]_q![j-1]_q!!}{([n]_q!!)^2} \\
    &\quad + \sum_{C} \frac{[j-i-2]_q![n]_q!!q^{\ell(C)} - [n-i-1]_q![j-1]_q!!q^{\ell'(C)}}{([n]_q!!)^2}.
\end{align*}

The first sum can be written
$$
\sum_C q^{\ell(C)} \frac{[j-i-2]_q!}{[n]_q!!} \left( 1 - \frac{[n-i-1]_q![j-1]_q!!}{[j-i-2]_q![n]_q!!} \right).
$$
Now, note that
\begin{align*}
    1 - \frac{[n-i-1]_q![j-1]_q!!}{[j-i-2]_q![n]_q!!} &= 1 - \frac{(1-q^{j-i-1})\dots (1-q^{n-i-1})}{(1-q^{2j})\dots (1-q^{2n})} \\
    &\le 1 - (1-q^{j-i-1})\dots (1-q^{n-i-1}) \\
    &\le q^{j-i-1} + \dots + q^{n-i-1} \\
    &\le q^{j-i-1} \frac{1-q^n}{1-q}
\end{align*}
and
$$
\sum_C q^{\ell(C)} \frac{[j-i-2]_q!}{[n]_q!!} \le \mathbb{P}\left( \substack{w(i)=k, w(i+1)=k+1 \\ w(j)=l, w(j+1)=l+1} \right)
$$
so the sum is bounded above by
$$
q^{j-i-1} \frac{1-q^n}{1-q} \mathbb{P}\left( \substack{w(i)=k, w(i+1)=k+1 \\ w(j)=l, w(j+1)=l+1} \right).
$$
Summing over $k,l$, to get the total contribution to the covariance term, gives an upper bound of
$$
q^{|i-j|-1} \frac{1-q^n}{1-q} \cdot \frac{22(1-q)^2}{(1-q^{2n-6})^2} \le q^{|i-j|-1} \frac{22(1-q)}{1-q^{2n-6}}
$$
by Lemma 3.1. Now, summing over all $i,j$ with $|i-j|\ge m$, we get an upper bound of
\begin{equation}
    q^{-1}\frac{22(1-q)}{1-q^{2n-6}} \sum_{|i-j|\ge m} q^{|i-j|} \le q^{-1}\frac{22(1-q)}{1-q^{2n-6}} 2(n-1)q^m \frac{1-q^{n-1}}{1-q} \le 44(n-1)q^{m-1}.
\end{equation}

For the second sum, we can throw away the negative part. The remaining terms give the probability that there exists some $a\le i+1$, $b \ge j$ such that $|w(a)|>w(b)$. Denote this event by $X$. Summing over $k$ and $l$ gives
$$
\sum_{k,l} \mathbb{P}\left( \left[ \substack{w(i)=k, w(i+1)=k+1 \\ w(j)=l, w(j+1)=l+1} \right] \cap X \right) = \mathbb{P} \left( \bigcup_{k,l} \left[ \substack{w(i)=k, w(i+1)=k+1 \\ w(j)=l, w(j+1)=l+1} \right] \cap X \right) \le \mathbb{P}[X].
$$
However, we also have that
\begin{align}
    \mathbb{P}[X] &\le \sum_{\substack{a\le i+1 \\ b\ge j}} \sum_{x>y,0} \mathbb{P}\left( w(a)=x, w(b)=y \right) + \mathbb{P}(w(a)=-x, w(b)=y) \nonumber \\
    &\le 2\sum_{\substack{a\le i+1 \\ b\ge j}} \sum_{x>y,0} \mathbb{P}\left( w(a)=x, w(b)=y \right) \nonumber \\
    &\le 2\sum_{\substack{a\le i+1 \\ b\ge j}} \sum_{x>y>0} \frac{q^{\max(|a-x|+|b-y|-1, 0)} (1-q)^2}{(1-q^{2n})(1-q^{2n-2})} \nonumber \\
    & \quad + 2\sum_{\substack{a\le i+1 \\ b\ge j}} \sum_{x>0>y} \frac{q^{\max(|a-x|+|b-y|-2, 0)} (1-q)^2}{(1-q^{2n})(1-q^{2n-2})} \label{prob_X_bound}
\end{align}
 The above sum can be rewritten by summing over $y' = y+1$ in the second sum:
$$
2\sum_{\substack{a\le i+1 \\ b\ge j}} \sum_{x=1}^n \sum_{y=-n+1}^{x-1} \frac{q^{\max(|a-x|+|b-y|-1, 0)} (1-q)^2}{(1-q^{2n})(1-q^{2n-2})}.
$$
We then break the inner double sum up into the following regimes: 
\begin{enumerate}
    \item $y<x\le a<b$
    \item $b<y<x$
    \item $b \ge y$ and $a < x$.
\end{enumerate}
We omit the factor of
$$
\frac{2(1-q)^2}{(1-q^{2n})(1-q^{2n-2})}.
$$
For regime $1$, we have the sum
\begin{align*}
    \sum_{x=1}^{a} \sum_{y=-n+1}^{x-1} q^{a-x+b-y-1} &\le q^{a+b-1}\sum_{x=1}^{a} q^{-x}\sum_{y=-n+1}^{a-1} q^{-y} \\
    &= q^{a+b-1} \left(q^{-a} \frac{1-q^{a}}{1-q}\right) \left( q^{1-a} \frac{1-q^{n+a-1}}{1-q} \right) \\
    &= \frac{q^{b-a} (1-q^{a})(1-q^{n+a-1})}{(1-q)^2}
\end{align*}
For regime $2$, we have the sum
\begin{align*}
    \sum_{x>y>b} q^{x-a+y-b-1} &\le q^{1-a-b} \sum_{x=b+1}^n q^x \sum_{y=b+1}^n q^y \\
    &= q^{-1-a-b} \left( q^{b+1}\frac{1-q^{n-b}}{1-q} \right)\left( q^{b+1}\frac{1-q^{n-b}}{1-q} \right) \\
    &= \frac{q^{b-a+1}(1-q^{n-b})^2}{(1-q)^2}
\end{align*}
For regime $3$, we have the sum
\begin{align*}
    \sum_{\substack{b \ge y, a < x, x>y}} q^{x-a+b-y-1} &= q^{b-a-1} \sum_{\substack{y \le b, x > a \\ x-y = k}} q^k \\
    &\le q^{b-a-1} \sum_{k = 1}^{2n-1} (b-a+k)q^k \\
    &\le q^{b-a} \sum_{k=0}^{2n-2} (b-a+1+k)q^k \\
    &= (b-a)q^{b-a} \frac{1-q^{2n-1}}{1-q} + \frac{q^{b-a}}{(1-q)^2}.
\end{align*}
Putting this all together, we get an upper bound of 
$$
(b-a)q^{b-a} \frac{1-q^{2n-1}}{1-q} + 3\frac{q^{b-a}}{(1-q)^2}
$$
so that 
$$
\mathbb{P}[X] \le \frac{2(1-q)^2}{(1-q^{2n})(1-q^{2n-2})} \sum_{\substack{a\le i+1 \\ b\ge j}} \left[ (b-a)q^{b-a} \frac{1-q^{2n-1}}{1-q} + 3\frac{q^{b-a}}{(1-q)^2} \right].
$$
Now, note that, by using \Cref{lem: assisting ineq},
\begin{align*}
    \sum_{\substack{a\le i+1 \\ b\ge j}} (b-a)q^{b-a} &= \sum_{a=1}^{i+1} \sum_{b=j}^n (b-a)q^{b-a} \\
    &= \sum_{a=1}^{i+1} \left[ \frac{(j-a)q^{j-a}}{1-q} + \frac{q^{j-a+1}}{(1-q)^2} - \frac{(n-a+1)q^{n-a+1}}{(1-q)} - \frac{q^{n-a+2}}{(1-q)^2}\right] \\
    &\le \sum_{a=1}^{i+1} \left[ \frac{(j-a)q^{j-a}}{1-q} + \frac{q^{j-a+1}}{(1-q)^2} \right] \\
    &= \frac{1}{1-q} \left( \frac{(j-i-1)q^{j-i-1}}{1-q} + \frac{q^{j-i}}{(1-q)^2} - \frac{jq^j}{1-q} - \frac{q^{j+1}}{(1-q)^2} \right) + \frac{q^{j-i} - q^{j+1}}{(1-q)^3} \\
    &\le \frac{(j-i-1)q^{j-i-1}}{(1-q)^2} - \frac{(j-i-1)q^{j}}{(1-q)^2} + \frac{2q^{j-i}-2q^{j+1}}{(1-q)^3} \\
    &= \frac{(j-i-1)q^{j-i-1}(1-q^{i+1})}{(1-q)^2} + \frac{2q^{j-i}(1-q^{i+1})}{(1-q)^3}
\end{align*}
We also have that
\begin{align*}
    \sum_{a=1}^{i+1} \sum_{b=j}^n q^{b-a} &= \sum_{a=1}^{i+1} \frac{q^{j-a}- q^{n-a+1}}{1-q} \\
    &= \frac{1}{1-q} \left( \frac{q^{j-i-1}-q^{j}}{1-q} - \frac{q^{n-i}-q^{n+1}}{1-q} \right) \\
    &\le \frac{q^{j-i-1}(1-q^{i+1})(1-q^{n-j+1})}{(1-q)^2}.
\end{align*}
It follows that
\begin{align*}
    \mathbb{P}[X] &\le \frac{2(1-q)^2}{(1-q^{2n})(1-q^{2n-2})} \Biggr[ \frac{(j-i-1)q^{j-i-1}(1-q^{i+1})(1-q^{2n-1})}{(1-q)^3} \\
    & \quad + \frac{2q^{j-i}(1-q^{i+1})(1-q^{2n-1})}{(1-q)^3} + \frac{3q^{j-i-1}(1-q^{i+1})(1-q^{n-j+1})}{(1-q)^4}\Biggr] \\
    &= \frac{2(j-i-1)q^{j-i-1}}{1-q} + \frac{4q^{j-i}}{1-q} + \frac{6q^{j-i-1}}{(1-q)^2} \\
    &= q^{j-i-1} \left( \frac{2(j-i-1)}{1-q} + \frac{4q}{1-q} + \frac{6}{(1-q)^2}\right).
\end{align*}
Now, we note that 
$$
\sum_{|i-j|\ge m} (|j-i|-1) q^{|i-j|-1} \le (n-1)\sum_{c=m-1}^\infty cq^c = \frac{(n-1)(m-1)q^{m-1}}{1-q} + \frac{(n-1)q^m}{(1-q)^2}
$$
and
$$
\sum_{|i-j|\ge m} q^{|i-j|-1} \le (n-1) \sum_{c=m-1} q^c = \frac{(n-1)q^{m-1}}{1-q}.
$$
Hence,
\begin{align*}
    \sum_{|i-j|\ge m} \mathbb{P}[X] &\le \frac{2(n-1)}{1-q} \left( \frac{(m-1)q^{m-1}}{1-q} + \frac{q^m}{(1-q)^2} \right) + \left(\frac{4q}{1-q} + \frac{6}{(1-q)^2}\right) \frac{(n-1)q^{m-1}}{1-q} \\
    &= \frac{(n-1)q^{m-1}(2m-2+4q)}{(1-q)^2} + \frac{(n-1)q^{m-1}(2q+6)}{(1-q)^3} \\
    &\le \frac{(n-1)q^{m-1}(2m+2)}{(1-q)^2} + \frac{8(n-1)q^{m-1}}{(1-q)^3}.
\end{align*}
We can then combine this with the result in (1) to find that the total covariance sum is bounded above by 
$$
44(n-1)q^{m-1} + \frac{(n-1)q^{m-1}(2m+2)}{(1-q)^2} + \frac{8(n-1)q^{m-1}}{(1-q)^3}.
$$
\end{proof}

\begin{lemma} \label{Lemma: Type 5 Bound, B_n, q small}
Fix some $2 \le m \le n-1$ and $q \le 1-(n-1)^{-1/2}$ and let $n \ge 10$. Then
$$
\sum_{i,j \ge 0} \mathrm{Cov}(\mathrm{des}(w)-\mathrm{des}(w_{i}^{-*}), \mathrm{des}(w)-\mathrm{des}(w_{j}^{-*}))
$$
$$
\le 1 + (n-1)\left(5 + 55m(1-q)^2 + 44q^{m-1} + \frac{q^{m-1}(2m+2)}{(1-q)^2} + \frac{8q^{m-1}}{(1-q)^3} \right)
$$
\end{lemma}
\begin{proof}
This follows from combining Lemmas \ref{Lemma: B_n off-diagonal} and \ref{Lemma: B_n near diagonal} and the $2(n-1) + 1$ covariance terms (each bounded by $1$ by Lemma \ref{Lemma: Uniform Bound Inverse}) from when $i$ or $j$ is $0$. 
\end{proof}

\subsubsection{Unconditional Bounds}

\begin{lemma}[Type 5 Bound] \label{Lemma: Type 5 Bound, B_n}
    The type $5$ terms are bounded by $173(n-1) + 1$.
\end{lemma}
\begin{proof}
In the case that $q \ge 1 - (n-1)^{-1/2}$, the result is true by Lemma \ref{Lemma: Type 5 Bound, B_n, q large}. Thus, assume $q \le 1 - (n-1)^{-1/2}$. We then want to bound the quantity given in Lemma \ref{Lemma: Type 5 Bound, B_n, q small}. Take
$$
m = 4\frac{\log(1-q)}{\log q} + 1.
$$
We have that
$$
m \le 4\frac{\log((n-1)^{-1/2})}{\log(1-(n-1)^{-1/2})} + 1 \le 3(n-1)^{1/2}\log((n-1)^{1/2}) + 1 \le n-1
$$
for $n \ge 100$. If $n \le 100$, the result holds by trivially bounding each term in the sum, so may assume $n\ge 50$ and $2\le m \le n-1$. Also, 
$$
q^m = q(1-q)^4
$$
so we use the bound in Lemma \ref{Lemma: Type 5 Bound, B_n, q small} to obtain that the sum is bounded above by
$$
1 + (n-1)(5 + 57m(1-q)^2 + 44 + 2(1-q)^2 + 8(1-q)) \le 1 + (n-1)(59 + 57m(1-q)^2) \le 1 + 173(n-1)
$$
since $m(1-q)^2 \le 2$ for all $0 < q < 1$.
\end{proof}

\begin{lemma}[Type 6 Bound] \label{Lemma: Type 6 Bound, B_n}
The type 6 terms are bounded by $1 + 173(n-1)$.
\end{lemma}
\begin{proof}
Let $A_i$ denote the event that $w(i+1) - w(i) = 1$ and $B_i$ denote the event that $w^{-1}(i+1) - w^{-1}(i)$. We have that
\begin{align*}
    \sum_{i,j \ge 1} \Cov(\des(w)-&\des(w_{j}^{-*}),\des(w^{-1})-\des((w_i^*)^{-1})) \\
    &= \sum_{i,j \ge 1} \Cov(I_{A_j}, I_{B_i}) \\
    &= \sum_{i,j \ge 1} \mathbb{P}\left( \substack{w(j+1) - w(j) = 1 \\ w^{-1}(i+1) - w^{-1}(i) = 1} \right) - \Prob(w(j+1) - w(j) = 1)\Prob(w^{-1}(i+1) - w^{-1}(i) = 1) \\
    &= \sum_{i,j \ge 1} \sum_{k,l} \Prob \left( \substack{w(j) = l, w(j+1) = l+1 \\ w(k) = i, w(k+1) = i+1} \right) - \Prob\left( \substack{w(j) = l \\ w(j+1) = l+1} \right)\Prob\left( \substack{w(k) = i \\ w(k+1) = i+1} \right)
\end{align*}
where $k$ and $l$ are taken over all values in $[-n, -2] \cup [1, n-1]$. Now, note that taking $w(k) = i$ when $k$ is negative is equivalent to having $w(|k|) = -i$. Thus, we can restrict $k$ to positive values and let $i$ range over negative values in the above sum to obtain
$$
\sum_{k,j \ge 1} \sum_{i,l} \Prob \left( \substack{w(j) = l, w(j+1) = l+1 \\ w(k) = i, w(k+1) = i+1} \right) - \Prob\left( \substack{w(j) = l \\ w(j+1) = l+1} \right)\Prob\left( \substack{w(k) = i \\ w(k+1) = i+1} \right)
$$
which is equal to $\sum_{i,j \ge 1} \Cov(\des(w)-\des(w_{j}^{-*}),\des(w)-\des(w_i^{-*}))$. Thus, we can obtain analogous results for Lemmas \ref{Lemma: Type 5 Bound, B_n, q large} and \ref{Lemma: Type 5 Bound, B_n, q small} since the terms involving $i$ or $j$ being $0$ are all trivially bounded.
\end{proof}

\subsection{$D_n$ Case}

This section uses similar techniques as the last section for the calculations of bounds for the Types 5 and 6 sums. We wish to bound
$$
\sum_{i,j} \mathrm{Cov}(\mathrm{des}(w)-\mathrm{des}(w_{i}^{-*}), \mathrm{des}(w)-\mathrm{des}(w_{j}^{-*})).
$$
For $i>0$, let $A_i$ denote the event that $w^{-1}(i+1)-w^{-1}(i)=1$ or $(w^{-1}(i+1), w^{-1}(i))= (2, -1)$ or $(1, -2)$. We then have that $\mathrm{des}(w)-\mathrm{des}(w_{i}^{-*}) = -\mathbbm{1}_{A_i}$. Thus, for $i,j>0$,
$$
\mathrm{Cov}(\mathrm{des}(w)-\mathrm{des}(w_{i}^{-*}), \mathrm{des}(w)-\mathrm{des}(w_{j}^{-*})) = \mathbb{P}(A_i \cap A_j) - \mathbb{P}(A_i)\mathbb{P}(A_j).
$$
We can write
\begin{align}
\begin{split} \label{intersection_term}
    \mathbb{P}(A_i \cap A_j) &= \mathbb{P}\left( \substack{w(i+1)-w(i) = 1 \\ w(j+1)-w(j) = 1} \right) + \mathbb{P}\left( \substack{w(i+1)-w(i) = 1 \\ (w(j+1), w(j)) = (2, -1)} \right) + \mathbb{P} \left( \substack{(w(i+1), w(i)) = (2, -1) \\ w(j+1)-w(j) = 1} \right) \\
    & \quad + \mathbb{P}\left( \substack{w(i+1)-w(i) = 1 \\ (w(j+1), w(j)) = (1, -2)} \right) + \mathbb{P} \left( \substack{(w(i+1), w(i)) = (1, -2) \\ w(j+1)-w(j) = 1} \right)
\end{split}
\end{align}
and
\begin{align*}
    \mathbb{P}(A_i)\mathbb{P}(A_j) &= \left[ \mathbb{P}(w(i+1)-w(i) = 1) + \mathbb{P}\left( \substack{w(i+1) = 2 \\ w(i) = -1} \right) + \mathbb{P}\left( \substack{w(i+1) = 1 \\ w(i) = -2} \right) \right] \\ 
    & \quad \cdot \left[ \mathbb{P}(w(j+1)-w(j) = 1) + \mathbb{P}\left( \substack{w(j+1) = 2 \\ w(j) = -1} \right) + \mathbb{P}\left( \substack{w(j+1) = 1 \\ w(j) = -2} \right) \right].
\end{align*}

\subsubsection{$q \approx 1$}

\begin{lemma} \label{q_approx_1_bound}
Let $|i-j|>1$, for $i,j \in [n]$, and $n$ big. Then,
$$
\mathbb{P}\left( \substack{w(i+1)-w(i) = 1 \\ w(j+1)-w(j) = 1} \right) \le 32 \frac{(1-q)^2}{(1-q^{2n-4})(1-q^{2n-6})} \le 32 \frac{(1-q)^2}{(1-q^{2n-6})^2}.
$$
\end{lemma}
\begin{proof}
We have,
\begin{align*}
    \mathbb{P}\left( \substack{w(i+1)-w(i) = 1 \\ w(j+1)-w(j) = 1} \right) &= \sum_{k,l \in [\pm n] \setminus \{-1, n\}} \mathbb{P}\left( \substack{w(i)=k, w(i+1)=k+1 \\ w(j)=l, w(j+1)=l+1} \right) \\
    &\le \sum_{k,l \in [n]} \frac{q^{\max(2|i-k|+2|j-l|-4, 0)}(1-q)^4}{(1-q^{2n-2})(1-q^{2n-4})(1-q^{2n-6})(1+q^{n-4})(1-q^n)} \\
    &\quad + \sum_{k,l \in [-n]} \frac{q^{\max(2|i-k|+2|j-l|-14, 0)}(1-q)^4}{(1-q^{2n-2})(1-q^{2n-4})(1-q^{2n-6})(1+q^{n-4})(1-q^n)} \\
    &\quad + \sum_{k \in [-n], l \in [n]} \frac{q^{\max(2|i-k|+2|j-l|-9, 0)}(1-q)^4}{(1-q^{2n-2})(1-q^{2n-4})(1-q^{2n-6})(1+q^{n-4})(1-q^n)} \\
    &\quad + \sum_{k \in [n], l \in [-n]} \frac{q^{\max(2|i-k|+2|j-l|-9, 0)}(1-q)^4}{(1-q^{2n-2})(1-q^{2n-4})(1-q^{2n-6})(1+q^{n-4})(1-q^n)}
\end{align*}
by Lemma \ref{set_indices}. This is from observing that $2|i-k|$ inversions are induced from setting $w(i) = k, w(i+1) = k+1$ for positive $k$, while $2|i-k|-5$ are induced for negative $k$ (this is from the first $i-1$ numbers on the right and the first $2-k$ numbers on the left). There is then an extra $4$ for double counting. We now have
$$
\sum_{k,l \in [n]} q^{\max(2|i-k|+2|j-l|-4, 0)}.
$$
This was computed in the hyperoctahedral case, with an upper bound of 
$$
4\frac{(1-q^{2n-2})^2}{(1-q^2)^2} + 4\frac{1-q^{2n-2}}{1-q^2} + 5 \le 13\frac{(1-q^{2n-2})^2}{(1-q^2)^2}.
$$
We now wish to evaluate 
$$
\sum_{k,l \in [-n]} q^{\max(2|i-k|+2|j-l|-10, 0)}.
$$
In this case we actually have $2|i-k|+2|j-l|-14 = 2i+2j-6-2(k+2)-2(l+2)$, where $k+2,l+2$ range from $0$ to $2-n$, so that the sum equals (since $2i+2j-6$ is nonnegative)
$$
q^{2i+2j-6} (1+q^2+ \cdots + q^{2(n-2)})^2 = q^{2i+2j-6} \frac{(1-q^{2n-2})^2}{(1-q)^2}.
$$

We now wish to evaluate
$$
\sum_{k \in [-n], l \in [n]} q^{\max(2|i-k|+2|j-l|-9, 0)}.
$$
We have $2|i-k|+2|j-l|-9 = 2i-5 + 2|j-l| - 2(k+2)$. We can sum over the following regimes:
\begin{itemize}
    \item $j>l+1$ or $j<l-1$: In either case, we have a sum of the form 
    \begin{align*}
        \sum_{k \in [-n]} q^{2i-1}q^{-2(k+2)}(1+q^2 + \cdots + q^{2(n-3)}) &\le q^{2i-1}(1+q^2 + \cdots + q^{2(n-2)})(1+q^2 + \cdots + q^{2(n-3)}) \\
        &= \frac{q^{2i-1}(1-q^{2n-2})(1-q^{2n-4})}{(1-q^2)^2}.
    \end{align*}
    \item $j=l+1$ or $l-1$, $k<-2$: This yields the sum $2q^{2i-1}(1+q^2+\cdots + q^{2(n-3)}) = \frac{2q^{2i-1}(1-q^{2n-4})}{1-q^2}$ (note the factor of $2$).
    \item $j=l, k<-3$: This yields the sum $q^{2i-1}(1+q^2+\cdots + q^{2(n-4)}) = \frac{q^{2i-1}(1-q^{2n-6})}{1-q^2}$.
    \item $(j,k) = (l,-2), (l,-3), (l-1, -2), (l+1,-2)$: This yields a maximum term of $1$ for each of the $4$ pairs.
\end{itemize}
The sum totals to
$$
2\frac{q^{2i-1}(1-q^{2n-2})(1-q^{2n-4})}{(1-q^2)^2} + \frac{2q^{2i-1}(1-q^{2n-4})}{1-q^2} + \frac{q^{2i-1}(1-q^{2n-6})}{1-q^2} + 4
$$
$$
\le 9\frac{q^{2i-1}(1-q^{2n-2})(1-q^{2n-4})}{(1-q^2)^2}.
$$
By similar reasoning, we have
$$
\sum_{k \in [n], l \in [-n]} q^{\max(2|i-k|+2|j-l|-9, 0)} \le 9\frac{q^{2j-1}(1-q^{2n-2})(1-q^{2n-4})}{(1-q^2)^2}.
$$
The sum of these four different sums is bounded above by
$$
13\frac{(1-q^{2n-2})^2}{(1-q^2)^2} + q^{2i+2j-6} \frac{(1-q^{2n-2})^2}{(1-q)^2} $$$$+ 9\frac{q^{2i-1}(1-q^{2n-2})(1-q^{2n-4})}{(1-q^2)^2} + 9\frac{q^{2j-1}(1-q^{2n-2})(1-q^{2n-4})}{(1-q^2)^2}
$$
$$
\le 32\frac{(1-q^{2n-2})^2}{(1-q)^2} 
$$
It follows that
\begin{align*}
    \mathbb{P}\left( \substack{w(i+1)-w(i) = 1 \\ w(j+1)-w(j) = 1} \right) &\le 32\frac{(1-q^{2n-2})^2}{(1-q)^2} \cdot \frac{(1-q)^4}{(1-q^{2n-2})(1-q^{2n-4})(1-q^{2n-6})(1+q^{n-4})(1-q^n)} \\
    &\le 32 \frac{(1-q)^2}{(1-q^{2n-4})(1-q^{2n-6})} \cdot \frac{(1+q^{n-1})(1-q^{n-1})}{(1+q^{n-4})(1-q^n)} \\
    &\le 32 \frac{(1-q)^2}{(1-q^{2n-4})(1-q^{2n-6})}.
\end{align*}
\end{proof}

\begin{lemma}
Fix $q \ge 1-(n-1)^{-1/2}$ and let $n \ge 10$. Then
$$
\sum_{i,j \ge 1} \mathrm{Cov}(\mathrm{des}(w)-\mathrm{des}(w_{i}^{-*}), \mathrm{des}(w)-\mathrm{des}(w_{j}^{-*})) \le 57(n-1).
$$
\end{lemma}
\begin{proof}
We have, since $\frac{1-q}{1-q^{2n-6}}$ is a decreasing function of $q$ in $(0,1)$, 
$$
\frac{1-q}{1-q^{2n-6}} \le \frac{(n-1)^{-1/2}}{1-(1-(n-1)^{-1/2})^{2n-6}} \le \frac{5}{4}(n-1)^{-1/2}
$$
where we note $1-(1-(n-1)^{-1/2})^{2n-6} \ge \frac{4}{5}$. We then observe from Lemma \ref{q_approx_1_bound} that
$$
\mathbb{P}\left( \substack{w(i+1)-w(i) = 1 \\ w(j+1)-w(j) = 1} \right) \le 32 \frac{(1-q)^2}{(1-q^{2n-4})(1-q^{2n-6})} \le 32 \frac{(1-q)^2}{(1-q^{2n-6})^2} \le 50(n-1)^{-1}.
$$
Then, we can sum over this to get
$$
\sum_{|i-j| > 1} \mathbb{P}\left( \substack{w(i+1)-w(i) = 1 \\ w(j+1)-w(j) = 1} \right) \le 50(n-1)
$$
since there are less than $(n-1)^2$ terms. For $|i-j| \le 1$, the probabilities are bounded above by $1$, so 
$$
\sum_{i,j \ge 1}\mathbb{P}\left( \substack{w(i+1)-w(i) = 1 \\ w(j+1)-w(j) = 1} \right) \le \sum_{|i-j| > 1} \mathbb{P}\left( \substack{w(i+1)-w(i) = 1 \\ w(j+1)-w(j) = 1} \right) + \sum_{|i-j| \le 1, i,j \ge 1} 1 \le 53(n-1). 
$$
Now, looking at the terms in (\ref{intersection_term}), we sum over the second term to obtain
$$
\sum_{i,j \ge 1} \mathbb{P}\left( \substack{w(i+1)-w(i) = 1 \\ (w(j+1), w(j)) = (2, -1)} \right) = \sum_{i \ge 1} \mathbb{P}\left( \bigcup_{j \ge 1} \left\{ \substack{w(i+1)-w(i) = 1 \\ (w(j+1), w(j)) = (2, -1)} \right\} \right) \le \sum_{i \ge 1} 1 = n-1
$$
since the events $\{(w(j+1), w(j)) = (2, -1)\}$ are pairwise disjoint. We can sum over the other terms in (\ref{intersection_term}) in a similar manner to obtain
$$
\sum_{i,j \ge 1} \mathbb{P}(A_i \cap A_j) \le 53(n-1) + (n-1) + (n-1) + (n-1) + (n-1) = 57(n-1).
$$
Since 
$$
\mathrm{Cov}(\mathrm{des}(w)-\mathrm{des}(w_{i}^{-*}), \mathrm{des}(w)-\mathrm{des}(w_{j}^{-*})) \le \mathbb{P}(A_i \cap A_j)
$$
we have the result.
\end{proof}

\begin{lemma} \label{Lemma: Type 5 Bound, D_n, q large}
Fix $q \ge 1-(n-1)^{-1/2}$ and let $n \ge 10$. Then
$$
\sum_{i,j} \mathrm{Cov}(\mathrm{des}(w)-\mathrm{des}(w_{i}^{-*}), \mathrm{des}(w)-\mathrm{des}(w_{j}^{-*})) \le 59(n-1)+1.
$$
\end{lemma}
\begin{proof}
We have,
\begin{align*}
    \sum_{i,j} \mathrm{Cov}(\mathrm{des}(w)-\mathrm{des}(w_{i}^{-*}), &\mathrm{des}(w)-\mathrm{des}(w_{j}^{-*})) \\
    &\le \sum_{i,j \ge 1} \mathrm{Cov}(\mathrm{des}(w)-\mathrm{des}(w_{i}^{-*}), \mathrm{des}(w)-\mathrm{des}(w_{j}^{-*})) \\
    & \quad + \sum_{j \ge 1} \mathrm{Cov}(\mathrm{des}(w)-\mathrm{des}(w_{0}^{-*}), \mathrm{des}(w)-\mathrm{des}(w_{j}^{-*})) \\
    & \quad + \sum_{i \ge 1} \mathrm{Cov}(\mathrm{des}(w)-\mathrm{des}(w_{i}^{-*}), \mathrm{des}(w)-\mathrm{des}(w_{0}^{-*})) \\
    & \quad + \mathrm{Cov}(\mathrm{des}(w)-\mathrm{des}(w_{0}^{-*}), \mathrm{des}(w)-\mathrm{des}(w_{0}^{-*})).
\end{align*}
Now, note that $\mathrm{des}(w)-\mathrm{des}(w_{i}^{-*})$ is an indicator random variable of the event $A_i$ for $i \ge 1$. We claim that a similar thing is true for $\mathrm{des}(w)-\mathrm{des}(w_{0}^{-*})$. First, note that obtaining $w_{0}^{-*}$ from $w$ involves switching the values $2, -1$ and the values $1, -2$ in the permutation if $-1$ comes before $2$. This changes whether $w$ has a descent at $s_i$ for $i \ge 1$ only if $\{w(i+1), w(i), w(-i), w(-i-1)\} = \{-2, -1, 1, 2\}$, since otherwise the order of $w(i+1), w(i)$ is unchanged ($\mathrm{des}_i(w)$ only depends on this order). Similarly, $\mathrm{des}_0(w)$ depends only on the order of the comparison between $w(-2), w(1)$. Thus, we have that going from $w$ to $w_0^{-*}$ can only add two descents if they are at $s_0$ and $s_1$ and $\{w(2),w(1), w(-1), w(-2)\} = \{2,1,-1,-2\}$. Observe that there are only two ways that the values at indices $-1,-2,1,2$ can be switched by multiplying on the left by $s_0$ (since $2$ and $-1$ cannot be opposite):
\begin{align*}
    -b, -a, a, b \to & a, b, -b, -a \\
    & -a, -b, b, a
\end{align*}
In any case, it is not hard to observe that the order of $w(1), w(2)$ remains the same or the order of $w(-1), w(2)$ remains the same. Hence, $\mathrm{des}(w)-\mathrm{des}(w_{0}^{-*})$ is at most $1$, and is $0$ otherwise. It follows that any of the covariance terms above are covariances of indicator random variables, so they are bounded above by $1$. We then have, using the previous lemma,
$$
\sum_{i,j} \mathrm{Cov}(\mathrm{des}(w)-\mathrm{des}(w_{i}^{-*}), \mathrm{des}(w)-\mathrm{des}(w_{j}^{-*})) \le 57(n-1) + (n-1)+(n-1)+1 = 59(n-1)+1.
$$
\end{proof}

\subsubsection{$q << 1$}

We first control the diagonal:
\begin{lemma} \label{Lemma: D_n near diagonal}
Fix some $0\le m \le n-1$ and $q \le 1-(n-1)^{-1/2}$ with $n \ge 4$. Then 
$$
\sum_{\substack{i,j \in [n-1] \\ |i-j|<m}}\mathbb{P}\left( \substack{w(i+1)-w(i) = 1 \\ w(j+1)-w(j) = 1} \right) \le 100m(n-1)(1-q)^2 + 3(n-1)
$$
\end{lemma}
\begin{proof}
We have that $\frac{1}{1-q^{2n-6}}$ is an increasing function of $q$, so
$$
\frac{1}{1-q^{2n-6}} \le \frac{1}{1-(1-(n-1)^{-1/2})^{2n-6}} \le \frac{5}{4}.
$$
We then have that
\begin{align*}
    \sum_{\substack{i,j \in [n-1] \\ |i-j|<m}}\mathbb{P}\left( \substack{w(i+1)-w(i) = 1 \\ w(j+1)-w(j) = 1} \right) &= \sum_{\substack{i,j \in [n-1] \\ 1<|i-j|<m}}\mathbb{P}\left( \substack{w(i+1)-w(i) = 1 \\ w(j+1)-w(j) = 1} \right) + \sum_{\substack{i,j \in [n-1] \\ |i-j|<2}}\mathbb{P}\left( \substack{w(i+1)-w(i) = 1 \\ w(j+1)-w(j) = 1} \right) \\
    &\le 2m(n-1) \frac{32(1-q)^2}{(1-q^{2n-6})^2} + 3(n-1) \\
    &\le 100m(n-1)(1-q)^2 + 3(n-1).
\end{align*}
\end{proof}

\begin{lemma} \label{Lemma: D_n off-diagonal}
Fix some $2\le m \le n-1$ and $q \le 1-(n-1)^{-1/2}$ with $n \ge 6$. Then,
$$
\sum_{|i-j|\ge m} \mathbb{P} \left( \substack{w(i+1)-w(i) = 1 \\ w(j+1)-w(j) = 1} \right) - \mathbb{P}(w(i+1)-w(i) = 1) \mathbb{P} (w(j+1)-w(j) = 1)
$$
$$
\le 64(n-1)q^{m-1} + \frac{(n-1)q^{m-2}(2m+2)}{(1-q)^2} + \frac{8(n-1)q^{m-2}}{(1-q)^3}.
$$
\end{lemma}
\begin{proof}
We have
$$
\mathbb{P} \left( \substack{w(i+1)-w(i) = 1 \\ w(j+1)-w(j) = 1} \right) - \mathbb{P}(w(i+1)-w(i) = 1) \mathbb{P} (w(j+1)-w(j) = 1)
$$
$$
= \sum_{k,l \in [\pm n] \setminus \{-1, n\} } \mathbb{P}\left( \substack{w(i)=k, w(i+1)=k+1 \\ w(j)=l, w(j+1)=l+1} \right) - \mathbb{P}(w(i)=k, w(i+1)=k+1)\mathbb{P}(w(j)=l, w(j+1)=l+1).
$$
We denote
$$
Q_{ij}(k,l) = \mathbb{P}\left( \substack{w(i)=k, w(i+1)=k+1 \\ w(j)=l, w(j+1)=l+1} \right) - \mathbb{P}(w(i)=k, w(i+1)=k+1)\mathbb{P}(w(j)=l, w(j+1)=l+1).
$$
Assuming $i< j$,
$$
\mathbb{P}\left( \substack{w(i)=k, w(i+1)=k+1 \\ w(j)=l, w(j+1)=l+1} \right) = \sum_{C} \frac{q^{\ell(C)}}{[2n-2]_q!![n]} \sum_{w_0 \in S_{j-i-2}} q^{\ell(w_0)} = \sum_{C} q^{\ell(C)} \frac{[j-i-2]_q!}{[2n-2]_q!![n]}.
$$
where $C$ ranges over all assignments to indices in $[n]$ that are less than $i$ or greater than $j+1$ (and obviously the negatives of these indices are assigned the appropriate values) and each $C$ specifies the sign of each value that is assigned in this block. In other words, each $C$ specifies the set of values on the positive side of the permutation, such that an even number of these values are negative. Let $\ell(C)$ denote the number of inversions induced by the assignment $C$. That is, $\ell(C)=\ell(w_C)$ for the signed permutation $w_C$ that satisfies $C$ and is completely well-ordered on the block from $i+2$ to $j-1$. The above claim follows from noting that $\ell(w)=\ell(C)+\ell(w_0)$ where $w_0$ corresponds to the way $w$ is ordered on $[i+2,j-1]$. This is because $w_C$ clearly minimizes the total number of $D$-inversions among permutations that follow the assignment $C$, and any additional inversions arise solely from the block from $i+2$ to $j-1$. Moreover, it is clear that adding one inversion to $w_0$ by a transposition in this block also adds exactly one inversion to the whole signed permutation. We also have that
$$
\mathbb{P}(w(i)=k, w(i+1)=k+1) = \sum_{C'} \frac{q^{\ell(C')}}{[2n-2]_q!![n]} \sum_{w_1 \in S_{n-i-1}} q^{\ell(w_1)} = \sum_{C'}q^{\ell(C')} \frac{[n-i-1]_q!}{[2n-2]_q!![n]} 
$$
where $C'$ ranges over assignments of exact values to $[i-1]$ and signs to the other values on the right half of the permutation. Moreover, if $C''$ ranges over assignments over $[j+2,n]$, with no stipulations on the signs of other values in the right half,
$$
\mathbb{P}(w(j)=l, w(j+1)=l+1) = \sum_{C''} \frac{q^{\ell(C'')}}{[2n-2]_q!![n]} \sum_{w''} q^{\ell(w'')} = \sum_{C''} q^{\ell(C'')} \frac{F(C'')}{[2n-2]_q!![n]}
$$
where $w''$ ranges over $D_{j-1}$ if $C''$ assigns an even number of negative integers to $[j+2, n]$ and $w''$ ranges over $B_{j-1} \setminus D_{j-1}$ if $C''$ assigns an odd number of negative integers $[j+2, n]$. ($w''$ corresponds to how $w$ is ordered on the first $j-1$ terms.) It is then easy to see that 
$$
F(C'') = \begin{cases} [2j-4]_q!![j-1] & \text{ if } C'' \text{ assigns an even number of negatives to } [j-1] \\ [2j-2]_q!! - [2j-4]_q!![j-1] & \text{ if } C'' \text{ assigns an odd number of negatives to } [j-1] \end{cases}
$$
We then have
$$
\mathbb{P} \left( \substack{w(i)=k \\ w(i+1)=k+1} \right) \mathbb{P} \left( \substack{w(j)=l \\ w(j+1)=l+1} \right) = \sum_{C',C''} q^{\ell(C')+\ell(C'')}\frac{F(C'')[n-i-1]_q!}{([2n-2]_q!![n])^2}.
$$
Now, consider the case that $C'$ and $C''$ can be combined to form an assignment $C$ to $[\pm n] \setminus \pm [i+2,j-1]$, i.e. that the values assigned by $C'$, $C''$ are nonintersecting subsets of $[\pm n] \setminus \pm \{k,k+1,l,l+1\}$ and $C''$ follows the sign assignments by $C'$. Suppose that, for any $w$ satisfying this assignment, $|w(i')|<w(j')$ for any $i' \in [1,i+1], j' \in [j, n]$, i.e. that the numbers assigned by $C'$ are all less than those assigned by $C''$. This requires that $C''$ only assigns positive values to positive indices. We claim that $\ell(C') + \ell(C'') = \ell(C)$. Consider an arbitrary inversion $(i',j')$, $i' < j'$, in $w_C$. There is then a bijection between such $(i',j')$ with $j-1 \le j' \le n$ and inversions in $C''$, where we simply take the inversion between the same two values. Similarly, there is a bijection between $(i', j')$ such that $i' \in [-i-1, i+1]$ or $(i',j') \in [1-j, -i-2] \times [i+2, j-1]$ and inversions in $C''$, where we again take the inversion between the same two values. This covers all inversions in $w_C$, since there are no inversions in the block $[i+2, j-1]$ itself.

Now, note that, for $C'$ and $C''$ discussed above, we have $F(C'') = [2j-4]_q!![j-1]$ since $C''$ assigns zero negative values to positive indices. We then can then sum over just these such $C', C''$ to obtain 
$$
\mathbb{P} \left( \substack{w(i)=k \\ w(i+1)=k+1} \right) \mathbb{P} \left( \substack{w(j)=l \\ w(j+1)=l+1} \right) \ge \sum_{C} q^{\ell(C)} \frac{[2j-4]_q!![j-1][n-i-1]_q!}{([2n-2]_q!![n])^2}
$$
where we are summing over $C$ such that $|w(i')|<w(j')$ for all $i' \in [1,i+1], j' \in [j, n]$. Thus,
\begin{align*}
    \mathbb{P}\left( \substack{w(i)=k, w(i+1)=k+1 \\ w(j)=l, w(j+1)=l+1} \right) &- \mathbb{P} \left( \substack{w(i)=k \\ w(i+1)=k+1} \right) \mathbb{P} \left( \substack{w(j)=l \\ w(j+1)=l+1} \right) \\
    &\le \sum_{C \in \Gamma} q^{\ell(C)} \frac{[j-i-2]_q!}{[2n-2]_q!![n]} \left(1 - \frac{[2j-4]_q!![j-1][n-i-1]_q!}{[2n-2]_q!![n][j-i-2]_q!} \right) \\
    & \quad + \sum_{C \not\in \Gamma} q^{\ell(C)} \frac{[j-i-2]_q!}{[2n-2]_q!![n]}.
\end{align*}
where $\Gamma$ is the set of $C$ such that $|w(i')|<w(j')$ for all $i' \in [1,i+1], j' \in [j, n]$. For the first sum, we compute:
\begin{align*}
    1 - \frac{[j-1]}{[n]} \cdot \frac{[2j-4]_q!![n-i-1]_q!}{[2n-2]_q!![j-i-2]_q!} &= 1 - \frac{1-q^{j-1}}{1-q^n} \cdot \frac{(1-q^{j-i-1}) \dots (1-q^{n-i-1})}{(1-q^{2j-2}) \dots (1-q^{2n-2})} \\
    &\le 1 - \cdot \frac{(1-q^{j-i-1}) \dots (1-q^{n-i-1})}{(1-q^{2j-2}) \dots (1-q^{2n-2})} \\
    &\le 1 - (1-q^{j-i-1})\dots (1-q^{n-i-1}) \\
    &\le q^{j-i-1} + \dots + q^{n-i-1} \\
    &= q^{j-i-1} \frac{1-q^{n-j+1}}{1-q} \\
    &\le q^{j-i-1} \frac{1-q^n}{1-q}.
\end{align*}
Hence, the first sum is bounded above by
$$
q^{j-i-1} \frac{1-q^n}{1-q}\sum_{C \in \Gamma} q^{\ell(C)} \frac{[j-i-2]_q!}{[2n-2]_q!![n]} \le q^{j-i-1} \frac{1-q^n}{1-q} \mathbb{P}\left( \substack{w(i)=k, w(i+1)=k+1 \\ w(j)=l, w(j+1)=l+1} \right)
$$
Summing over $k,l$, to get the total contribution to the covariance term, gives an upper bound of
$$
q^{|i-j|-1} \frac{1-q^n}{1-q} \cdot \frac{32(1-q)^2}{(1-q^{2n-6})^2} \le q^{|i-j|-1} \frac{32(1-q)}{1-q^{2n-6}}
$$
by Lemma 3.1. Now, summing over all $i,j$ with $|i-j|\ge m$, we get an upper bound of
\begin{equation} \label{first_sum}
    q^{-1}\frac{32(1-q)}{1-q^{2n-6}} \sum_{|i-j|\ge m} q^{|i-j|} \le q^{-1}\frac{32(1-q)}{1-q^{2n-6}} 2(n-1)q^m \frac{1-q^{n-1}}{1-q} \le 64(n-1)q^{m-1}.
\end{equation}
Now, we wish to bound the second sum. It is equal to the probability that there exists some $a\le i+1$, $b \ge j$ such that $|w(a)|>w(b)$ (and that $w(i)=k$, etc.). Denote this event by $X$. We have, summing the second sum over $k,l$,
$$
\sum_{k,l} \mathbb{P}\left( \left[ \substack{w(i)=k, w(i+1)=k+1 \\ w(j)=l, w(j+1)=l+1} \right] \cap X \right) = \mathbb{P} \left( \bigcup_{k,l} \left[ \substack{w(i)=k, w(i+1)=k+1 \\ w(j)=l, w(j+1)=l+1} \right] \cap X \right) \le \mathbb{P}[X].
$$
However, we also have that
\begin{align*}
    \mathbb{P}[X] &\le \sum_{\substack{a\le i+1 \\ b\ge j}} \sum_{x>y,0} \mathbb{P}\left( w(a)=x, w(b)=y \right) + \mathbb{P}(w(a)=-x, w(b)=y) \\
    &\le 2\sum_{\substack{a\le i+1 \\ b\ge j}} \sum_{x>y,0} \mathbb{P}\left( w(a)=x, w(b)=y \right) \\
    &\le 2\sum_{\substack{a\le i+1 \\ b\ge j}} \sum_{x>y>0} \frac{q^{\max(|a-x|+|b-y|-1, 0)} (1-q^{n-2})(1-q)^2}{(1-q^n)(1-q^{2n-2})(1-q^{2n-4})} \\
    & \quad + 2\sum_{\substack{a\le i+1 \\ b\ge j}} \sum_{x>0>y} \frac{q^{\max(|a-x|+|b-y|-3, 0)} (1-q^{n-2})(1-q)^2}{(1-q^n)(1-q^{2n-2})(1-q^{2n-4})}
\end{align*}
since switching $x$ and $-x$ at positions $a$ and $-a$ can only increase the total number of inversions. The exponents in the last line follows from noting that we induce $|a-x|$ inversions by setting $w(a) = x$ and $|b-y|$ inversions by setting $w(b)=y$ if $y$ is positive, or $|b-y|-2$ inversions if $y$ is negative. The above sum can be rewritten 
\begin{align*}
    \mathbb{P}[X] &\le 2q^{-1}\sum_{\substack{a\le i+1 \\ b\ge j}} \sum_{x>y>0} \frac{q^{\max(|a-x|+|b-y|-1, 0)} (1-q^{n-2})(1-q)^2}{(1-q^n)(1-q^{2n-2})(1-q^{2n-4})} \\
    & \quad + 2q^{-1}\sum_{\substack{a\le i+1 \\ b\ge j}} \sum_{x>0>y} \frac{q^{\max(|a-x|+|b-y|-2, 0)} (1-q^{n-2})(1-q)^2}{(1-q^n)(1-q^{2n-2})(1-q^{2n-4})} \\
    &\le 2q^{-1}\sum_{\substack{a\le i+1 \\ b\ge j}} \sum_{x>y>0} \frac{q^{\max(|a-x|+|b-y|-1, 0)} (1-q)^2}{(1-q^{2n})(1-q^{2n-2})} + 2q^{-1}\sum_{\substack{a\le i+1 \\ b\ge j}} \sum_{x>0>y} \frac{q^{\max(|a-x|+|b-y|-2, 0)} (1-q)^2}{(1-q^{2n})(1-q^{2n-2})}
\end{align*}
This is just $q^{-1}$ times the sum \eqref{prob_X_bound} we bounded in the hyperoctahedral case. Thus, we have 
$$
\sum_{|i-j|\ge m} \mathbb{P}[X] \le q^{-1}\frac{(n-1)q^{m-1}(2m+2)}{(1-q)^2} + q^{-1}\frac{8(n-1)q^{m-1}}{(1-q)^3} = \frac{(n-1)q^{m-2}(2m+2)}{(1-q)^2} + \frac{8(n-1)q^{m-2}}{(1-q)^3}
$$
We can then combine this with the result in (\ref{first_sum}) to find that the total sum in the lemma statement is bounded above by 
$$
64(n-1)q^{m-1} + \frac{(n-1)q^{m-2}(2m+2)}{(1-q)^2} + \frac{8(n-1)q^{m-2}}{(1-q)^3}
$$
\end{proof}

\begin{lemma} 
Fix some $2 \le m \le n-1$ and $q \le 1-(n-1)^{-1/2}$ and let $n \ge 10$. Then
$$
\sum_{i,j \ge 1} \mathrm{Cov}(\mathrm{des}(w)-\mathrm{des}(w_{i}^{-*}), \mathrm{des}(w)-\mathrm{des}(w_{j}^{-*}))
$$
$$
\le (n-1)\left(7 + 100m(1-q)^2 + 64q^{m-1} + \frac{q^{m-2}(2m+2)}{(1-q)^2} + \frac{8q^{m-2}}{(1-q)^3}\right).
$$
\end{lemma}
\begin{proof}
From (\ref{intersection_term}), it is not hard to see that
\begin{align*}
    \sum_{i,j \ge 1} \mathrm{Cov}(&\mathrm{des}(w)-\mathrm{des}(w_{i}^{-*}), \mathrm{des}(w)-\mathrm{des}(w_{j}^{-*})) \\
    &= \sum_{i,j \ge 1} \mathbb{P}(A_i \cap A_j) - \mathbb{P}(A_i)\mathbb{P}(A_j) \\
    &\le \sum_{i,j \ge 1} \mathbb{P} \left( \substack{w(i+1)-w(i) = 1 \\ w(j+1)-w(j) = 1} \right) - \mathbb{P}(w(i+1)-w(i) = 1) \mathbb{P} (w(j+1)-w(j) = 1) \\
    & \quad + \sum_{i,j \ge 1} \mathbb{P}\left( \substack{w(i+1)-w(i) = 1 \\ (w(j+1), w(j)) = (2, -1)} \right) + \sum_{i,j \ge 1} \mathbb{P} \left( \substack{(w(i+1), w(i)) = (2, -1) \\ w(j+1)-w(j) = 1} \right) \\
    & \quad + \sum_{i,j \ge 1} \mathbb{P}\left( \substack{w(i+1)-w(i) = 1 \\ (w(j+1), w(j)) = (1, -2)} \right) + \sum_{i,j \ge 1} \mathbb{P} \left( \substack{(w(i+1), w(i)) = (1, -2) \\ w(j+1)-w(j) = 1} \right)
\end{align*}
where we have simply dropped sum terms from the sum over $\mathbb{P}(A_i)\mathbb{P}(A_j)$. Now, from the previous two lemmas,
\begin{align*}
    \sum_{i,j \ge 1} \mathbb{P} \left( \substack{w(i+1)-w(i) = 1 \\ w(j+1)-w(j) = 1} \right) &- \mathbb{P}(w(i+1)-w(i) = 1) \mathbb{P} (w(j+1)-w(j) = 1) \\
    &\le \sum_{|i-j| \le m} \mathbb{P} \left( \substack{w(i+1)-w(i) = 1 \\ w(j+1)-w(j) = 1} \right) - \mathbb{P}(w(i+1)-w(i) = 1) \mathbb{P} (w(j+1)-w(j) = 1) \\
    & \quad + \sum_{|i-j| > m} \mathbb{P} \left( \substack{w(i+1)-w(i) = 1 \\ w(j+1)-w(j) = 1} \right) - \mathbb{P}(w(i+1)-w(i) = 1) \mathbb{P} (w(j+1)-w(j) = 1) \\
    &\le 100m(n-1)(1-q)^2 + 3(n-1) \\
    & \quad + 64(n-1)q^{m-1} + \frac{(n-1)q^{m-2}(2m+2)}{(1-q)^2} + \frac{8(n-1)q^{m-2}}{(1-q)^3} \\
    &= (n-1)\left(3 + 100m(1-q)^2 + 64q^{m-1} + \frac{q^{m-2}(2m+2)}{(1-q)^2} + \frac{8q^{m-2}}{(1-q)^3}\right)
\end{align*} 
In each of the other four sums, fixing $i$ results in a sum of probabilities of pairwise disjoint events, so they are each bounded above by $n-1$ since there are $n-1$ choices for $i$. Hence, there is an extra contribution of $4(n-1)$.
\end{proof}

\begin{lemma} \label{Lemma: Type 5 Bound, D_n, q small}
Fix some $2 \le m \le n-1$ and $q \le 1-(n-1)^{-1/2}$ and let $n \ge 10$. Then
$$
\sum_{i,j \ge 0} \mathrm{Cov}(\mathrm{des}(w)-\mathrm{des}(w_{i}^{-*}), \mathrm{des}(w)-\mathrm{des}(w_{j}^{-*}))
$$
$$
\le 1+(n-1)\left(9 + 100m(1-q)^2 + 64q^{m-1} + \frac{q^{m-2}(2m+2)}{(1-q)^2} + \frac{8q^{m-2}}{(1-q)^3}\right).
$$
\end{lemma}
\begin{proof}
As in Lemma \ref{Lemma: Type 5 Bound, D_n, q large}, the extra terms total to $2(n-1) + 1$.
\end{proof}

\subsubsection{Unconditional Bounds}

\begin{lemma} [Type 5 Bound] \label{Lemma: Type 5 Bound, D_n}
The Type 5 sum is bounded by $1 + 287(n-1)$
\end{lemma}
\begin{proof}
In the case that $q \ge 1 - (n-1)^{-1/2}$, the result is true by Lemma \ref{Lemma: Type 5 Bound, D_n, q large}. Thus, assume $q \le 1 - (n-1)^{-1/2}$. We then want to bound the quantity given in Lemma \ref{Lemma: Type 5 Bound, D_n, q small}. Take
$$
m = 4\frac{\log(1-q)}{\log q} + 2.
$$
We then have that
$$
m \le 4\frac{\log((n-1)^{-1/2})}{\log(1 -(n-1)^{-1/2})} +2 \le 4(n-1)^{1/2}\log((n-1)^{1/2}) \le n-1
$$
when $n \ge 101$. The result is clearly true when $n \le 100$, since we can trivially bound each term in the sum by $1$, so assume that $2 \le m \le n$. Notice that $q^m = q^2(1-q)^4$. We then obtain that the sum is bounded by
$$
1 + (n-1)(9 + 102m(1-q)^2 + 64q(1-q)^4 + 10(1-q)) \le 1 + 287(n-1)
$$
since $m(1-q)^2 \le 2$ for $0 < q < 1$.
\end{proof}

\begin{lemma} \label{Lemma: Type 6 Bound, D_n}
The Type 6 sum is also bounded by $1 + 287(n-1)$.
\end{lemma}
\begin{proof}
As in Lemma \ref{Lemma: Type 6 Bound, B_n}, we can observe that
$$
\sum_{i,j \ge 1} \mathbb{P}\left( \substack{w(j+1) - w(j) = 1 \\ w^{-1}(i+1) - w^{-1}(i) = 1} \right) - \Prob(w(j+1) - w(j) = 1)\Prob(w^{-1}(i+1) - w^{-1}(i) = 1) $$$$= \sum_{i,j \ge 1} \sum_{k,l} \Prob \left( \substack{w(j) = l, w(j+1) = l+1 \\ w(k) = i, w(k+1) = i+1} \right) - \Prob\left( \substack{w(j) = l \\ w(j+1) = l+1} \right)\Prob\left( \substack{w(k) = i \\ w(k+1) = i+1} \right)
$$
$$
\sum_{k,j \ge 1} \sum_{i,l} \Prob \left( \substack{w(j) = l, w(j+1) = l+1 \\ w(k) = i, w(k+1) = i+1} \right) - \Prob\left( \substack{w(j) = l \\ w(j+1) = l+1} \right)\Prob\left( \substack{w(k) = i \\ w(k+1) = i+1} \right)
$$
so that this portion of the sum is the same as in the Type 5 case. The rest of the covariance sum can be similarly bounded by finding probabilities that sum to at most $1$ and trivially bounding terms where $i$ or $j$ are $0$ trivially by $1$.
\end{proof}

\section{Proofs of Theorems \ref{thm: main theorem 1, B_n}, \ref{thm: main theorem 2, D_n}, and \ref{thm: Wasserstein 1-distance}} \label{sec: limit thms}

\subsection{$B_n$ Case}

We wish to prove Theorem \ref{thm: main theorem 1, B_n} and the first part of Theorem \ref{thm: Wasserstein 1-distance}. Using \cref{thm: size-bias steins,,thm: size-bias steins 1-distance}, it suffices to bound the quantities
\begin{equation} \label{full_var_quant}
    \frac{\mu}{\sigma^2}\sqrt{\Var\E(\des(w)+\des(w^{-1})-\des(w^*)-\des((w^*)^{-1})|w)}
\end{equation}
and
\begin{equation} \label{full_exp_quant}
    \frac{\mu}{\sigma^3}\E(\des(w)+\des(w^{-1})-\des(w^*)-\des((w^*)^{-1}))^2.
\end{equation}
From Lemma \ref{Moments: Type B}, we have that
$$
\mu = \frac{2qn}{1+q}
$$
and
\begin{equation} \label{variance_bound_B}
\sigma^2 \ge \frac{2nq(1-q+q^2)}{(1+q)^2(1+q+q^2)} \ge \frac{nq}{6}.
\end{equation}
It follows that 
\begin{equation} \label{ratio_bound_B}
\frac{\mu}{\sigma^2} \le \frac{(1+q)(1+q+q^2)}{1-q+q^2} \le 6
\end{equation}
and
$$
\frac{\mu}{\sigma^3} \le 6\sqrt{6}q^{-\frac{1}{2}}n^{-\frac{1}{2}}.
$$
By Lemmas \ref{Lemma: Uniform Bound Difference} and \ref{Lemma: Uniform Bound Inverse}, $|\des(w)+\des(w^{-1})-\des(w^*)-\des((w^*)^{-1}| \le 4$. Thus, \eqref{full_exp_quant} is bounded by $96\sqrt{6}q^{-\frac{1}{2}}n^{-\frac{1}{2}} \le 236q^{-\frac{1}{2}}n^{-\frac{1}{2}}$.

To obtain a bound on \eqref{full_var_quant} we can combine the results in Lemmas \ref{Lemma: Type 1 Bound}, \ref{Lemma: Type 2 Bound}, \ref{Lemma: Type 3 Bound}, \ref{Lemma: Type 4 Bound}, \ref{Lemma: Type 5 Bound, B_n}, \ref{Lemma: Type 6 Bound, B_n}. This results in an upper bound on the sum of covariance terms given by
$$
2 \cdot 63(n-1) + 4 \cdot 63(n-1) + 4 (306(n-1) + 9) + 2(594(n-1) + 9) + 4(173(n-1) + 1)
$$
$$
= 3482n-3424 < 3600n.
$$
Thus, due to the factor of $(2n)^{-2}$, the variance is bounded by $900n^{-1}$. Hence, \eqref{full_var_quant} is bounded above by $360n^{-1/2}$. \cref{thm: size-bias steins,,thm: size-bias steins 1-distance} then give the desired result.

\subsection{$D_n$ Case}

We again wish to bound the quantities
\begin{equation} \label{full_var_term_D}
    \frac{\mu}{\sigma^2}\sqrt{\Var\E(\des(w)+\des(w^{-1})-\des(w^*)-\des((w^*)^{-1})|w)}
\end{equation}
and
\begin{equation} \label{full_exp_term_D}
    \frac{\mu}{\sigma^3}\E(\des(w)+\des(w^{-1})-\des(w^*)-\des((w^*)^{-1}))^2.
\end{equation}
Observe that \ref{Moments: Type D} gives
\[
\mu = \frac{2qn}{1+q}
\]
again and
\begin{equation} \label{variance_bound_D}
\sigma^2 \ge \frac{2nq(1-q+q^2)}{(1+q)^2(1+q+q^2)} - \frac{10q^2}{(1+q)^2} \ge \frac{(n-15)q}{6} \ge \frac{nq}{12}
\end{equation}
when $W = D_n$ and $n \ge 30$. For $n \ge 30$, we also have
\begin{equation} \label{ratio_bound_D}
\frac{\mu}{\sigma^2} \le \frac{n(1+q)(1+q+q^2)}{n(1-q+q^2) - 5q(1+q+q^2)} \le \frac{6n}{n-15} \le 12
\end{equation}
and
$$
\frac{\mu}{\sigma^3} \le 24\sqrt{3}n^{-1/2}q^{-1/2}.
$$
By Lemmas \ref{Lemma: Uniform Bound Difference} and \ref{Lemma: Uniform Bound Inverse}, $|\des(w)+\des(w^{-1})-\des(w^*)-\des((w^*)^{-1}| \le 4$. Thus, \eqref{full_exp_term_D} is bounded by $384\sqrt{3}q^{-1/2}n^{-1/2} \le 666q^{-1/2}n^{-1/2}$.

The term in \eqref{full_var_term_D} can be bounded using Lemmas \ref{Lemma: Type 1 Bound}, \ref{Lemma: Type 2 Bound}, \ref{Lemma: Type 3 Bound}, \ref{Lemma: Type 4 Bound}, \ref{Lemma: Type 5 Bound, D_n}, \ref{Lemma: Type 6 Bound, D_n}. We get an upper bound on the sums of covariance terms given by
$$
2 \cdot 63(n-1) + 4 \cdot 63(n-1) + 4 (306(n-1) + 9) + 2(594(n-1) + 9) + 4(1 + 287(n-1))
$$
$$
\le 3938n-3880 < 4000n.
$$
The variance is then bounded by $1000n^{-1}$ so that \eqref{full_var_term_D} is bounded by $768n^{-1/2}$.

\section{Wasserstein $2$-Distance Bounds} \label{sec: 2-distance}

In this section, we prove \Cref{thm: Wasserstein 2-distance} by using the Wasserstein $1$-distance bounds in \Cref{thm: Wasserstein 1-distance}. We recall the definition of the Wasserstein $2$-distance.
\begin{definition}
Let $A$ and $B$ be real-valued random variables. Then the \emph{Wasserstein $2$-distance} between $A$ and $B$ is given by
\[
d_2(A,B) = \inf_{(A',B') \in \Gamma(A,B)} \left(\E[(A'-B')^2]\right)^{\frac{1}{2}}
\]
where $\Gamma(A,B)$ denotes the set of couplings of $(A,B)$.
\end{definition}
The Wasserstein $2$-distance bound given by \Cref{thm: Wasserstein 2-distance} is necessary in the proof of our main result \Cref{CLT: General Coxeter Groups} since we need it to show convergence in distribution (which is not gauranteed by the Wasserstein $1$-distance bound given in \Cref{thm: Wasserstein 1-distance}).

In the language of \cite{CGJ18}, the size-bias coupling of $w$ (where $w$ is an element of any Coxeter group) is $(4, 1)$-bounded for the upper and lower tails. This follows from observing that
$$
\des(w^*) + \des((w^*)^{-1}) \le \des(w) + \des(w^{-1}) + 4
$$
by Lemmas \ref{Lemma: Uniform Bound Difference} and \ref{Lemma: Uniform Bound Inverse}. \cite[Theorem 3.3]{CGJ18} (or, alternatively, \cite[Corollary 1.1]{AB15}) then implies the following.

\begin{proposition} \label{prop: tail_bounds}
Let $w$ be Mallows distributed in either $B_n$ or $D_n$ and let $\mu = \frac{2nq}{1+q}$ be the mean of $\des(w) + \des(w^{-1})$. Then
$$
\Prob(\des(w) + \des(w^{-1}) - \mu \ge x) \le \exp \left( -\frac{x^2}{8(x/3 + \mu)} \right)
$$
for $x \ge 0$ and
$$
\Prob(\des(w) + \des(w^{-1}) - \mu \le -x) \le \exp \left( - \frac{x^2}{8\mu}\right)
$$
for $0 \le x < \mu$.
\end{proposition}
\begin{remark}
In fact, we have that both inequalities are true for all $x \ge 0$, since if $x > \mu$, then 
$$
\Prob(\des(w) + \des(w^{-1}) - \mu \le -x) = \Prob(\des(w) + \des(w^{-1}) \le \mu-x) \le \Prob(\des(w) + \des(w^{-1}) < 0) = 0.
$$
(When $x = \mu$, the inequality must also hold since the left hand side is monotonically nonincreasing in $x$ and the right hand side is continuous.)
\end{remark}
\begin{remark}
The proposition also holds for when $w$ is sampled from $A_n$, since it is weaker than Proposition 4.4 in \cite{He_2022}.
\end{remark}

Now, we have that, for some coupling $(X, N)$ of $(\des(w) + \des(w^{-1}) - \mu)/\sigma$ and the standard normal, 
\begin{equation} \label{good_coupling}
    d_W\left(\frac{\des(w) + \des(w^{-1}) - \mu}{\sigma}, N\right) \ge \E \left| X - N \right| - 0.00001q^{-1/2}
\end{equation}
by the definition of the Wasserstein $1$-distance. Thus, for any $y > \frac{\mu}{\sigma}$,
\begin{align*}
    \E \left| X - N \right|^2 &\le \E \left( \left| X - N \right|^2 \big| -y \le X,N \le y \right)\Prob(-y \le X,N \le y) \\
    &\quad + \E \left( \left| X - N \right|^2 \big| \max(|X|, |N|) > y \right) \Prob(\max(|X|, |N|) > y)
\end{align*}
Moreover, noting that $|X-N|^2 = X^2 + N^2 - 2XN \le 2X^2 + 2N^2$ yields
\begin{align*}
    \E \left| X - N \right|^2 &\le \E \left( 2y\left| X - N \right| \big| -y \le X,N \le y \right)\Prob(-y \le X,N \le y) \\
    &\quad + 2\E \left( X^2 \big| |X| > y \right) \Prob(|X| > y) + 2\E(X^2 \big| |N| > y \ge |X|)\Prob(|N| > y \ge |X|)\\
    &\quad + 2\E \left( N^2 \big| |N| > y \right) \Prob(|N| > y) + 2\E(N^2 \big| |X| > y \ge |N|)\Prob(|X| > y \ge |N|) \\
    &\le 2y \E |X-N| + 2\E \left( X^2 \big| |X| > y \right) \Prob(|X| > y) + 2\E \left( N^2 \big| |N| > y \right) \Prob(|N| > y) \\
    &\quad + 2y^2\left[\Prob(|N| > y \ge |X|) + \Prob(|X| > y \ge |N|)\right].
\end{align*}
We now bound the above terms. 

\begin{lemma} \label{Lemma: x tail, x^2}
For any $y > 0$,
$$
\E \left( X^2 \big| |X| > y \right) \Prob(|X| > y) $$$$\le \frac{32\mu}{3\sigma^2} \exp\left( -\frac{3\sigma^2 y^2}{32\mu} \right) + \frac{32\mu}{3\sigma^2} \exp\left( -\frac{3\mu}{32} \right) + \frac{2048}{9\sigma^2}\exp\left( -\frac{3\mu}{32} \right) + \frac{8\mu}{\sigma^2} \exp\left( -\frac{\sigma^2 y^2}{8\mu} \right)
$$
\end{lemma}
\begin{proof}
First, let $B$ denote the event that $|X| \ge y$. We have, using \Cref{prop: tail_bounds},
\begin{align}
    \E(X^2 | B)\Prob (B) &= \int_{0}^\infty \Prob(\{X^2 > z\} \cap B) dz \nonumber\\
    &= \int_{y^2}^\infty \Prob(X^2 > z) dz \nonumber \\
    &\le \int_{y^2}^\infty \exp\left( -\frac{\sigma^2 z}{8(\sigma\sqrt{z}/3 + \mu)} \right)dz + \int_{y^2}^\infty \exp\left( -\frac{\sigma^2 z}{8\mu} \right) dz \label{two_integrals}
\end{align}
since $X^2 > z$ is equivalent to $|\des(w) + \des(w^{-1}) - \mu| > \sigma \sqrt{z}$. The first integral is bounded above by
$$
\int_{y^2}^{\mu^2/\sigma^2} \exp\left( -\frac{\sigma^2 z}{8(\mu/3 + \mu)} \right)dz + \int_{\mu^2/\sigma^2}^\infty \exp\left( -\frac{\sigma^2 z}{8(\sigma\sqrt{z}/3 + \sigma\sqrt{z})} \right)dz
$$
$$
= \int_{y^2}^{\mu^2/\sigma^2} \exp\left( -\frac{3\sigma^2 z}{32\mu} \right)dz + \int_{\mu^2/\sigma^2}^\infty \exp\left( -\frac{3\sigma \sqrt{z}}{32} \right)dz
$$
The above sum of integrals can be computed as
\[
\frac{32\mu}{3\sigma^2} \exp\left( -\frac{3\sigma^2 y^2}{32\mu} \right) - \frac{32\mu}{3\sigma^2} \exp\left( -\frac{3\mu}{32} \right) + \frac{2048}{9\sigma^2} \left( \frac{3\mu}{32}\exp\left(-\frac{3\mu}{32}\right) + \exp\left( -\frac{3\mu}{32} \right) \right)
\]
The second term in (\ref{two_integrals}) is easily computed as
$$
\frac{8\mu}{\sigma^2} \exp\left( -\frac{\sigma^2 y^2}{8\mu} \right)
$$
yielding the desired bound.
\end{proof}

\begin{lemma} \label{Lemma: n tail, n^2}
For any $y>0$,
$$
\E(N^2 | |N| > y)\Prob(|N|>y) \le\sqrt{\frac{2}{\pi}}ye^{-\frac{y^2}{2}} + \frac{2}{y}\sqrt{\frac{2}{\pi}}e^{-\frac{y^2}{2}}
$$
\end{lemma}
\begin{proof}
We have,
$$
\E(N^2 | |N| > y)\Prob(|N|>y) = 2\E(N^2 | N>y)\Prob(N>y) = 2\int_y^\infty t^2\frac{1}{\sqrt{2\pi}}e^{-\frac{t^2}{2}} dt
$$
since $N$ is symmetrically distributed. Evaluating the above integral with integration by parts yields the result.
\end{proof}

\begin{lemma} \label{Lemma: trivial_tails}
For $y \ge 0$, 
$$
2y^2\left[\Prob(|N| > y \ge |X|) + \Prob(|X| > y \ge |N|)\right]
$$
$$
\le 2y^2 \left[ \exp\left( -\frac{3y^2\sigma^2}{32\mu} \right) + \exp\left( -\frac{3y\sigma}{32} \right) + \exp\left( -\frac{y^2\sigma^2}{8\mu} \right) \right] + 2\sqrt{\frac{2}{\pi}}y e^{-\frac{y^2}{2}}
$$
\end{lemma}
\begin{proof}
Proposition \ref{prop: tail_bounds} implies
$$
\Prob(|X| > y \ge |N|) \le \Prob(|X| > y) \le \exp\left( -\frac{y^2\sigma^2}{8(y\sigma/3+\mu)} \right) + \exp\left( -\frac{y^2\sigma^2}{8\mu} \right).
$$
If $\mu \ge y\sigma$, then the $y\sigma$ in the denominator of the first term above may be replaced by $\mu$. The opposite is true when $y\sigma \ge \mu$. It follows that the above is upper bounded by
$$
\exp\left( -\frac{3y^2\sigma^2}{32\mu} \right) + \exp\left( -\frac{3y\sigma}{32} \right) + \exp\left( -\frac{y^2\sigma^2}{8\mu} \right).
$$
In addition, as in the proof of Lemma \ref{Lemma: x tail, x^2},
$$
\Prob(|X| > y \ge |N|) \le \Prob(N > y) = \int_y^\infty \frac{1}{\sqrt{2\pi}}e^{-\frac{t^2}{2}} dt \le \frac{1}{\sqrt{2\pi}} \int_y^\infty e^{-\frac{yt}{2}} dt = \frac{1}{y} \sqrt{\frac{2}{\pi}} e^{-\frac{y^2}{2}}.
$$
Combining these bounds yields the result.
\end{proof}

Assume that $q \le 1$. Note that $\mu \ge qn$. We can combine this fact with the results of lemmas \ref{Lemma: x tail, x^2}, \ref{Lemma: n tail, n^2}, and \ref{Lemma: trivial_tails} and the bounds given in (\ref{ratio_bound_B}), (\ref{ratio_bound_D}), (\ref{variance_bound_B}), and (\ref{variance_bound_D}) to find that, in the $B_n$ case,
\begin{align*}
    \E((X-N)^2) &\le 2y\E |X-N| + (2y^2 + 128)\exp\left( -\frac{y^2}{64} \right) + 2y^2\exp \left( -\frac{\sqrt{6}n^{1/2}q^{1/2}y}{64} \right) \\
    &\quad + \left(2y^2 + 48\right)\exp \left( -\frac{y^2}{48} \right) + \left(128 + \frac{8128}{3nq}\right)\exp\left( -\frac{3nq}{32} \right) + \left(4y + \frac{4}{y} \right)\exp\left({-\frac{y^2}{2}}\right).
\end{align*}
Furthermore, the above is an upper bound for $d_2\left( \frac{\des(w) + \des(w^{-1}) - \mu}{\sigma}, N\right)$ by the definition of the Wasserstein $2$-distance (since $(X,N)$ is a coupling). Then, assume that $nq \ge 32/3$ so that $n^{1/2}q^{1/2} \ge \frac{4\sqrt{6}}{3}$. Set $y = 8\log(nq) > 8$ and let $n \ge 4$. We then obtain that, using \eqref{good_coupling} and \Cref{thm: Wasserstein 1-distance},
\begin{align*}
    \E((X-N)^2) &\le 16\log(nq) \left(180+236q^{-1/2} + 0.00001q^{-1/2}\right)n^{-\frac{1}{2}} + 256(\log(nq))^2 \exp\left( -(\log(nq))^2 \right) \\
    &\quad + 128(\log (nq))^2 \exp\left( -\log (nq) \right) + 112(\log(nq))^2\exp\left( -\frac{128(\log(nq))^2}{9} \right) + 384 \frac{32}{3nq}\\
    &\quad + (32\log(nq) + 1)\exp(-32 (\log (nq))^2) \\
    &\le 10000q^{-1/2}n^{-1/2}\log(nq).
\end{align*}
where we have used the fact that, for any $x \ge e$, $\log x \le \sqrt{x}$ (and we let $x = nq$).

This proves Theorem $\ref{thm: Wasserstein 2-distance}$ in the $B_n$ case when $q \le 1$. Instances when $q \ge 1$ follow from Corollary \ref{cor: q_invert}. The $D_n$ case is proved similarly (and we know that Corollary \ref{Cor: D_n simple_var_bound} and (\ref{ratio_bound_D}) still hold since $n \ge nk \ge 30$). To see the $A_n = S_{n+1}$ case, note that, by \cite[Proposition 3.11]{He_2022}, when $q \le 1$,
$$
\sigma^2 \ge \frac{2nq(1-q+q^2)}{(1+q)^2(1+q+q^2)} \ge \frac{nq}{6}
$$
and 
$$
\frac{\mu}{\sigma^2} \le \frac{(1+q)(1+q+q^2}{1-q+q^2} \le 6
$$
so the proof is identical as in the $B_n$ case. (\cite[Proposition 2.7]{He_2022} allows us to extend the result to when $q \ge 1$.)

\section{Central Limit Theorem for General Coxeter Groups} \label{sec: grand result}

In this section, we finally prove our main result, \Cref{CLT: General Coxeter Groups}. F\'eray proves a similar result in \cite{ferayCLT} that involves $w_n$ being uniformly distributed in $W_n$. Equivalently, he proves the case in which the set of permissible parameters is $\{1\}$. Our proof of Theorem \ref{CLT: General Coxeter Groups} closely mirrors his methods.

The following are \cite[Lemmas 3 \& 4]{ferayCLT} and also \cite[Lemmas 1 \& 3]{Mal72}.

\begin{lemma} \label{Lemma: 2-distance convergence}
Let $X_n$ and $X$ be square integrable random variables. Then $d_2(X_n, X) \to 0$ if and only if $X_n \to X$ in distribution and $\E(X_n^2) \to \E(X^2)$.
\end{lemma}
\begin{lemma} \label{2-distance_split_sum}
Let $k>0$ be an integer and $Z$ be a standard normal random variable. If $X_1, \dots, X_k$ are independent random variables and $a_j$ for $1 \le j \le k$ are real such that $\sum_{j = 1}^k a_j^2 = 1$, then
$$
d_2\left( \sum_{j=1}^k a_jX_j, Z \right) \le \sum_{j=1}^k a_j^2 d_2(X_j, Z).
$$
\end{lemma}

We now begin the proof of \Cref{CLT: General Coxeter Groups}. Note that if $\frac{t(w_n) - \E(t(w_n))}{\sqrt{\Var(t(w_n))}}$ tends to a standard normal in distribution then we must have that $\Var(t(w_n))$ approaches infinity, since $t(w_n)$ is integer valued. \cite[Proposition 6.15]{KS20} has more details. Thus, it remains to show the "if" direction. Assume $\Var(t(w_n))$ approaches infinity. Let 
$$w_n = \prod\limits_{j = 1}^{k_n} w_{n,j}
$$
where $w_{n,j} \in W_n^j$. We then have that the $w_{n,j}$ are independent and Mallows distributed in $W_n^j$. Let the parameters of these distributions be $q_{n,j}$ and let the rank of $W_n^j$ be $r_{n,j}$. Moreover, let $t_n = t(w_n)$ and $t_{n,j} = t(w_{n,j})$. By \cite[Lemma 2.4]{BR22},
\begin{equation} \label{general_descent_decomp}
t_n = \sum_{j = 1}^{k_n} t_{n,j}.
\end{equation}
Let $s_{n,j}^2 = \Var(t(w_{n,j}))$ and $s_n^2 = \Var(t(w_n)) = \sum_{j = 1}^{k_n} s_{n,j}^2$. We also define the following normalized random variables:
$$
\widetilde{t_{n,j}} = \frac{t_{n,j} - \E(t_{n,j})}{\sqrt{\Var(t_{n,j})}}, \quad \widetilde{t_n} = \frac{t_{n} - \E(t_{n})}{\sqrt{\Var(t_{n})}}.
$$
(\ref{general_descent_decomp}) can then be rewritten as
$$
\widetilde{t_n} = \sum_{j=1}^{k_n} \frac{s_{n,j}}{s_n}\widetilde{t_{n,j}}.
$$

Recall that all irreducible finite Coxeter groups are one the following: $A_p$ for $p \ge 1$, $B_p$ for $p \ge 2$, $D_p$ for $p \ge 4$, $I_2(m)$ for $m \ge 3$, or one of the exceptional types, $H_3, H_4, E_6, E_7, E_8$. We define $a_p^q$ to be $t(w)$ where $w$ is a random element of $A_p$ that is Mallows distributed with parameter $q$, and similarly for $b_p^q$ and $d_p^q$. We also define the normalized random variables
$$
\widetilde{a_p^q} = \frac{a_p^q - \E(a_p^q)}{\sqrt{\Var(a_p^q)}}, \quad \widetilde{b_p^q} = \frac{b_p^q - \E(b_p^q)}{\sqrt{\Var(b_p^q)}}, \quad \widetilde{d_p^q} = \frac{d_p^q - \E(d_p^q)}{\sqrt{\Var(d_p^q)}}.
$$
By Lemma \ref{Lemma: 2-distance convergence}, theorem $\ref{thm: Wasserstein 2-distance}$ ensures that $\widetilde{a_p^q}$ approaches $Z$ in distribution as long as $p\min(q, 1/q) \to \infty$.

Fix $\epsilon > 0$. By Theorem \ref{thm: Wasserstein 2-distance}, there exists some uniform constant bound $C > 100$ (that depends on $\epsilon$) such that
\begin{equation} \label{large_groups_close}
d_2(\widetilde{a_p^q}, N), \ d_2(\widetilde{b_p^q}, N), \ d_2(\widetilde{d_p^q}, N) \le \epsilon
\end{equation}
for all integers $p$ and positive reals $q$ such that $p\min(q, 1/q) \ge C$. Call any $W_{n}^j$ \textit{large} if it is one of $a_p^q, b_p^q, d_p^q$ for which $p\min(q, 1/q) \ge C$. Otherwise, we call $W_n^j$ \textit{small}. We may then reorder the indices $j$ such that there exists some $l_n$ where $W_{n,j}$ is large if and only if $j \le l_n$. We now write
$$
s_{n,+}^2 = \sum_{j = 1}^{l_n} s_{n,j}^2, \quad s_{n, -}^2 = \sum_{j = l_n+1}^{k_n} s_{n,j}^2
$$
and
$$
\widetilde{t_{n,+}} = \sum_{j = 1}^{l_n} \frac{s_{n,j}}{s_{n,+}} \widetilde{t_{n,j}}, \quad \widetilde{t_{n,-}} = \sum_{j = l_n+1}^{k_n} \frac{s_{n,j}}{s_{n,-}} \widetilde{t_{n,j}}
$$
so that
$$
\widetilde{t_n} = \frac{s_{n,+}}{s_n} \widetilde{t_{n,+}} + \frac{s_{n,-}}{s_n} \widetilde{t_{n,-}}.
$$
We now prove convergence to the standard normal using estimates on the characteristic function. Fix some $\zeta \in \mathbb{R}$.

\subsection{Large Component Bound}

From (\ref{large_groups_close}), we have that $d_2(\widetilde{t_{n,j}}, Z) \le \epsilon$ for $j \le l_n$. By Lemma \ref{2-distance_split_sum}, we then have,
$$
d_2(\widetilde{t_{n,+}}, Z) = d_2\left( \sum_{j = 1}^{l_n} \frac{s_{n,j}}{s_{n,+}}\widetilde{t_{n,j}}, Z \right) \le \sum_{j = 1}^{l_n} \frac{s_{n,j}^2}{s_{n,+}^2} d_w(\widetilde{t_{n,j}}, Z) \le \epsilon \sum_{j = 1}^{l_n}\frac{s_{n,j}^2}{s_{n,+}^2} = \epsilon.
$$
Recall that $x \mapsto \exp(-x^2/2)$ is the characteristic function of $Z$. Since $u \mapsto \exp(iu)$ as a function on $\mathbb{R}$ is $1$-Lipschitz, we have that,
\begin{equation} \label{large component bound}
\begin{split}
    \left| \E\left[ \exp\left(i\zeta\frac{s_{n,+}}{s_n} \widetilde{t_{n,+}}\right) \right] - \exp\left( -\frac{\zeta^2s_{n,+}^2}{2s_n^2} \right)\right| &\le \E\left[\left| \exp\left(i\zeta\frac{s_{n,+}}{s_n} \widetilde{t_{n,+}}\right) - \exp\left(i\zeta\frac{s_{n,+}}{s_n} Z\right) \right|\right] \\
    &\le \frac{s_{n,+}}{s_n}|\zeta| \E\left| \widetilde{t_{n,+}}-Z \right| \le |\zeta| \| \widetilde{t_{n,+}}-Z\|_2 \le |\zeta|\epsilon
\end{split}
\end{equation}
where $\|\cdot\|_2$ denotes the $L_2$ norm. The second to last inequality follows from the Cauchy-Schwarz inequality. The last inequality follows from choosing a coupling of $(\widetilde{t_{n,+}}, Z)$ where the $L_2$ norm above is equal to $d_2(\widetilde{t_{n,+}}, Z)$. 

\subsection{Small Component Bound}

We first show that there exists some $K$ (only dependent on $\epsilon$) such that $s_{n,j} \le K$ and 
\begin{equation*}
\E\left(|t_{n,j} - \E(t_{n,j})|^3 \right) \le Ks_{n,j}^2
\end{equation*}
for all $j > l_n$. For all such $j$, $W_n^j$ is small. First suppose it is $I_2(m)$ for some $m$, an exceptional type, or is one of $a_p^q, b_p^q, d_p^q$ for $p \le 30$. Then, since $I_2(m)$ always has rank $2$ and the largest rank among the exceptional types is $8$, the rank of $W_n^j$ is at most $30$. It follows that $t_{n,j} \le 60$ so $s_{n,j} \le 60$ and 
$$
\E\left(|t_{n,j} - \E(t_{n,j})|^3 \right) \le 60\E\left(|t_{n,j} - \E(t_{n,j})|^2 \right) = 60s_{n,j}^2.
$$
Now, suppose that $W_n^j$ is one of $a_p^q, b_p^q, d_p^q$ for $p \ge 30$. We then have that $p\min(q, 1/q) \le C$ so that, by Corollaries \ref{Cor: A_n simple_var_bound}, \ref{Cor: B_n simple_var_bound}, and \ref{Cor: D_n simple_var_bound}, $s_{n,j}^2 \le 8C$. Assume initially that $q \le 1$. Lemma \ref{Lemma: descent_cubed_bound} then implies that, using Jensen's inequality,
\begin{align*}
    \E\left(|t_{n,j} - \E(t_{n,j})|^3 \right) &\le 9\E\left( \des(w)^3 + \des(w^{-1})^3 + \E(t_{n,j})^3 \right) \\
    &= 18\E \left( \des(w)^3 \right) + 9\E(t_{n,j})^3 \\
    &\le 18(p^3q^3 + 24p^2q^2 + 16pq) + 9p^3q^3 \\
    &\le (27C^2 + 432C + 288)pq.
\end{align*}
Corollaries \ref{Cor: A_n simple_var_bound}, \ref{Cor: B_n simple_var_bound}, and \ref{Cor: D_n simple_var_bound} imply that $pq \le 12\sigma_{n,j}^2$ so we have that
$$
\E\left(|t_{n,j} - \E(t_{n,j})|^3 \right) \le 12(27C^2 + 432C + 288)\sigma_{n,j}^2.
$$
Hence, we can take $K = \max(60, \sqrt{8C}, 12(27C^2 + 432C + 288))$ and the result is proven.

Now, since the $s_{n,j}$ for $j > l_n$ are uniformly bounded with no dependence on $n$, and $\lim_{n \to \infty} s_n = +\infty$ by assumption, for $n$ sufficiently large we have $\frac{\zeta^2}{s_n^2}s_{n,j}^2 \le 1$ for all $l_n < j \le k_n$. We may also assume that $|\zeta| \le s_n$. Using \cite[(27.11) \& (27.15)]{Bil95} then yields
\begin{equation*}
\begin{split}
    \left| \E\left[ \exp\left(i\frac{\zeta}{s_n}(t_{n,j} - \E(t_{n,j}))\right) \right] - \exp\left( -\frac{\zeta^2s_{n,j}^2}{2s_n^2} \right) \right| &\le \frac{|\zeta|^3}{s_n^3}\E\left( |t_{n,j} - \E(t_{n,j})|^3 \right) + \frac{|\zeta|^4}{s_n^4}s_{n,j}^4 \\
    &\le \frac{|\zeta|^3}{s_n^3}Ks_{n,j}^2 + \frac{|\zeta|^3}{s_n^3}K^2s_{n,j}^2 \le 2K^2|\zeta|^3\frac{s_{n,j}^2}{s_n^3}.
\end{split}
\end{equation*}
We now use the following from \cite{Bil95}:
\begin{proposition}[{\cite[(27.5)]{Bil95}}] \label{prop: products_to_sum}
For any collections of complex numbers $(a_i)_{1 \le i \le t}$, $(b_i)_{1\le i \le t}$ of absolute value at most $1$,
$$
\left| \prod_{i=1}^t a_i - \prod_{i=1}^t b_i \right| \le \sum_{i=1}^t |a_i-b_i|.
$$
\end{proposition}
We have, since the $t_{n,j}$ are independent,
\begin{align*}
    \left| \E\left[ \exp\left(i\zeta\frac{s_{n,-}}{s_n}\widetilde{t_{n,-}}\right) \right] - \exp\left( -\frac{\zeta^2s_{n,-}^2}{2s_n^2} \right) \right| &= \left| \prod_{j=l_n+1}^{k_n} \E\left[ \exp\left(i\frac{\zeta}{s_n}(t_{n,j} - \E(t_{n,j}))\right) \right] - \prod_{j=l_n+1}^{k_n} \exp\left( -\frac{\zeta^2s_{n,j}^2}{2s_n^2} \right) \right| \\
    &\le \sum_{j=l_n+1}^{k_n} 2K^2|\zeta|^3\frac{s_{n,j}^2}{s_n^3} = 2K^2|\zeta|^3\frac{s_{n,-}^2}{s_n^3} \le \frac{2K^2|\zeta|^3}{s_n}.
\end{align*}
The last expression above tends to $0$ as $n \to \infty$, so for $n$ sufficiently large, 
\begin{equation} \label{small component bound}
\left| \E\left[ \exp\left(i\zeta\frac{s_{n,-}}{s_n}\widetilde{t_{n,-}}\right) \right] - \exp\left( -\frac{\zeta^2s_{n,-}^2}{2s_n^2} \right) \right| \le \epsilon.
\end{equation}

\subsection{Proof of Convergence}

Note that, by independence,
$$
\E\left[ \exp(i\zeta\widetilde{t_n}) \right] = \E\left[ \exp\left(i\zeta\frac{s_{n,+}}{s_n}\widetilde{t_{n,+}}\right) \right] \E\left[ \exp\left(i\zeta\frac{s_{n,-}}{s_n}\widetilde{t_{n,-}}\right) \right]
$$
and
$$
\exp\left( -\frac{\zeta^2}{2} \right) = \exp\left( -\frac{\zeta^2 s_{n,-}^2}{2s_n^2} \right) \exp\left( -\frac{\zeta^2 s_{n,+}^2}{2s_n^2} \right).
$$
Thus, by Proposition \ref{prop: products_to_sum} (the $t=2$ case, specifically) and equations (\ref{large component bound}) and (\ref{small component bound}),
$$
\left| \E\left[ \exp(i\zeta\widetilde{t_n}) \right] - \exp\left( -\frac{\zeta^2}{2} \right) \right| \le (1 + |\gamma|)\epsilon
$$
for sufficiently large $n$ (where the threshold value depends solely on $\epsilon$ and $\zeta$). Thus, the characteristic function of $\widetilde{t_n}$ converges pointwise to $\zeta \mapsto \exp\left( -\frac{\zeta^2}{2} \right)$, which is the characteristic function of $N$. Hence, by L\'evy's continuity theorem, $\widetilde{t_n} = \frac{t(w_n) - \E(t(w_n))}{\sqrt{\Var(t(w_n))}}$ converges in distribution to $N$.

\appendix

\section{Moment Computations}

\begin{lemma} \label{Lem: Coxeter Expectation}
Let $w \in W$ be Mallows distributed with parameter $q \le 1$ where $W$ is any finite Coxeter group and let the number of generators of $W$ be $n$. Then,
\[
\mathbb{E}(\mathrm{des}(w) + \mathrm{des}(w^{-1})) = \frac{2qn}{1+q}
\]
\end{lemma}
\begin{proof}
The result follows from noting that $\mathbb{E}(\mathrm{des}_i(w)) = \frac{q}{1+q}$ since there is an obvious bijection $w \mapsto ws_i$ that maps signed permutations without a descent at $s_i$ to those without a descent at $s_i$. It is then clear that this bijection adds $\pm 1$ to the length of $w$ (depending on if there is a descent or not).
\end{proof}

The following is a corollary of \cite[Proposition 3.11]{He_2022}.

\begin{lemma} \label{Moments: Type A}
Let $w \in A_n$ be Mallows distributed with parameter $q \le 1$ and $n \ge 2$. Then,
$$
\frac{2nq(1-q+q^2)}{(1+q)^2(1+q+q^2)} \le \mathrm{Var}(\mathrm{des}(w) + \mathrm{des}(w^{-1})) \le \frac{2nq(1-q)(1+q)^2}{1-q^n} + \frac{2(n+2)q(1-q+q^2)}{(1+q)^2(1+q+q^2)}.
$$
\end{lemma}
\begin{proof}
For the lower bound on the variance, we have that, from \cite[Proposition 3.11]{He_2022},
$$
\mathrm{Var}(\mathrm{des}(w) + \mathrm{des}(w^{-1})) \ge \frac{2(n+1)q(1-q+q^2)}{(1+q)^2(1+q+q^2)} - \frac{2q(1-3q+q^2)}{(1+q)^2(1+q+q^2)} \ge \frac{2nq(1-q+q^2)}{(1+q)^2(1+q+q^2)}.
$$
To show the upper bound, \cite[Proposition 3.11]{He_2022} states that
$$
\Cov(\des(w), \des(w^{-1})) \le \frac{nq(1-q)(1+q)^2}{1-q^n}.
$$
In the proof (and in \cite[Proposition 5.2]{BDF10}), it is also stated that
$$
\Var(\des(w)) = \frac{(n+1)q(1-q+q^2)}{(1+q)^2(1+q+q^2)} - \frac{q(1-3q+q^2)}{(1+q)^2(1+q+q^2)}.
$$
We have,
$$
-(1-3q+q^2) = q - (1-q)^2 \le q + (1-q)^2 = 1 -q + q^2
$$
so that
$$
\Var(\des(w)) \le \frac{(n+2)q(1-q+q^2)}{(1+q)^2(1+q+q^2)}.
$$
Then, noting that
\begin{align*}
    \Var(\des(w) + \des(w^{-1})) &= \Var(\des(w)) + \Var(\des(w^{-1})) + 2\Cov(\des(w), \des(w^{-1})) \\
    &= 2\Var(\des(w)) + 2\Cov(\des(w), \des(w^{-1}))
\end{align*}
yields the result.
\end{proof}

\begin{corollary} \label{Cor: A_n simple_var_bound}
Let $w \in A_n$ be Mallows distributed with parameter $q \in (0, \infty)$ and $n \ge 2$. Then,
$$
\frac{n\min(q, 1/q)}{6} \le \mathrm{Var}(\mathrm{des}(w) + \mathrm{des}(w^{-1})) \le 8n\min(q, 1/q).
$$
\end{corollary}
\begin{proof}
Without loss of generality assume $q \le 1$ (since the other case follows by Corollary \ref{cor: q_invert}). The lower bound follows from noting that $\frac{1}{(1+q)^2} \ge \frac{1}{4}$ and 
$$
\frac{1-q+q^2}{1+q+q^2} = \frac{(1-q)^2 + q}{(1-q)^2 + 3q} \ge \frac{1}{3}.
$$
For the upper bound, we have,
$$
\frac{2nq(1-q)(1+q)^2}{1-q^n} + \frac{2(n+2)q(1-q+q^2)}{(1+q)^2(1+q+q^2)} \le 2nq(1+q) + 2(n+2)q \le 8nq.
$$
\end{proof}

\begin{lemma} \label{Moments: Type B}
Let $w \in B_n$ be Mallows distributed with parameter $q \le 1$ and $n \ge 2$. Then,
$$
\frac{2nq(1-q+q^2)}{(1+q)^2(1+q+q^2)} \le \mathrm{Var}(\mathrm{des}(w) + \mathrm{des}(w^{-1})) \le \frac{2nq(2+q+2q^2)}{(1+q)^2(1+q+q^2)}.
$$
\end{lemma}
\begin{proof}
Note that
\begin{align}
    \mathrm{Var}(\mathrm{des}(w) + \mathrm{des}(w^{-1})) &= \mathrm{Var}\left(\sum_{i=0}^{n-1} \mathrm{des}_i(w) + \sum_{i=0}^{n-1} \mathrm{des}_i(w^{-1})\right) \nonumber \\ \label{variance_formula}
    &= 2\sum_{0 \le i,j \le n-1} \mathrm{Cov}(\mathrm{des}_i(w), \mathrm{des}_j(w)) + 2\sum_{0 \le i,j \le n-1} \mathrm{Cov}(\mathrm{des}_i(w), \mathrm{des}_j(w^{-1})). 
\end{align}
For $|i-j|>1$, $s_i$ and $s_j$ commute, so $\mathrm{des}_i(w)$ and $\mathrm{des}_j(w)$ are independent events by Lemma \ref{Lemma: general_parabolic_indep} since $\mathrm{des}_i(w)$ is a function of $w_{\{s_i\}}$. It follows that the covariance of these terms are $0$.

Now, we compute $\mathrm{Cov}(\mathrm{des}_i(w), \mathrm{des}_{i+1}(w))$. If $i \ge 1$, then we have that $(B_n)_{\{s_i, s_{i+1}\}} \cong S_3$. Then, $\mathrm{des}_i(w) = 1$ if and only if $w_{\{s_i, s_{i+1}\}} = 213, 312, 321$ and $\mathrm{des}_{i+1}(w) = 1$ if and only if $w_{\{s_i, s_{i+1}\}} = 321, 231, 132$. Now, the expectation of each of $\des_i(w)$ and $\des_{i+1}(w)$ is $\frac{q}{1+q}$. Hence,
\begin{align*}
    \mathrm{Cov}(\mathrm{des}_i(w), \mathrm{des}_{i+1}(w)) &= \mathbb{E}(\mathrm{des}_i(w) \mathrm{des}_{i+1}(w)) - \mathbb{E}(\mathrm{des}_i(w)) \mathbb{E}(\mathrm{des}_{i+1}(w)) \\
    &= \mathbb{P}(w_{\{s_i, s_{i+1}\}} = 321) - \left( \frac{q}{1+q} \right)^2 \\
    &= \frac{q^3}{(1+q)(1+q+q^2)} - \left( \frac{q}{1+q} \right)^2 \\
    &= -\frac{q^2}{(1+q)^2(1+q+q^2)}.
\end{align*}
We have used the fact that the normalizing constant for the Mallows distribution over $S_3$ is 
\[
Z_{S_3}(q) = [1]_q[2]_q[3]_q = \frac{1-q^2}{1-q} \cdot \frac{1-q^3}{1-q} = (1+q)(1+q+q^2).
\]
To compute $\mathrm{Cov}(\mathrm{des}_0(w), \mathrm{des}_{1}(w))$, note that $\mathrm{des}_0(w) = 1$ if and only if $w_{\{s_0, s_1\}} = (\pm 2, 1, -1, \mp 2)$ or $ (\pm 1, 2, -2, \mp 1)$ (noticing that $(B_n)_{\{s_0, s_1\}}$ is isomorphic to $B_2$). It follows that
\begin{align*}
    \mathrm{Cov}(\mathrm{des}_0(w), \mathrm{des}_{1}(w)) &= \mathbb{E}(\mathrm{des}_0(w) \mathrm{des}_{1}(w)) - \mathbb{E}(\mathrm{des}_0(w)) \mathbb{E}(\mathrm{des}_{1}(w)) \\
    &= \mathbb{P}(w_{\{s_0, s_1\}} = \{2, 1, -1, -2\}) - \left( \frac{q}{1+q} \right)^2 \\
    &= \frac{q^4}{(1+q)(1+q+q^2+q^3)} - \frac{q^2}{(1+q)^2} \\
    &= - \frac{q^2}{(1+q)^2(1+q^2)}
\end{align*}
where we have used the normalization constant for $B_2$
\[
Z_{B_2}(q) = [2]_q[4]_q = (1+q)(1+q+q^2+q^3).
\]
For all $i$, we have
$$
\mathrm{Cov}(\mathrm{des}_i(w), \mathrm{des}_{i}(w)) = \mathbb{E}(\mathrm{des}_i(w)^2) - \mathbb{E}(\mathrm{des}_i(w))^2 = \frac{q}{1+q} - \left( \frac{q}{1+q} \right)^2 = \frac{q}{(1+q)^2}.
$$
We can now compute:
\begin{align}
    \sum_{0 \le i,j \le n-1} \mathrm{Cov}(\mathrm{des}_i(w), \mathrm{des}_j(w)) &= n\frac{q}{(1+q)^2} - 2(n-2)\frac{q^2}{(1+q)^2(1+q+q^2)} - 2\frac{q^2}{(1+q)^2(1+q^2)} \label{B_n variance}\\
    & \ge \frac{nq(1-q+q^2)}{(1+q)^2(1+q+q^2)}. \nonumber
\end{align}
Now, note that for any $i,j$ we have that $\mathrm{des}_i(w)$ and $\mathrm{des}_j(w^{-1})$ are independent conditioned on $\{w(k): k \in \overline{\{s_i\}}\} \cap \overline{\{s_j\}}$ and $\{-w(k): k \in \overline{\{s_i\}}\} \cap \overline{\{s_j\}}$ each containing at most $1$ element. This is by Lemma \ref{inverse_indep}. Otherwise, these sets are the same (as each contains $2$ elements), so for $i,j \ge 1$ this yields the event
$$
E = \{\{w(i), w(i+1)\} = \{j, j+1\}, \{-j, -j-1\}\}.
$$
Conditioning on $E_1 = \{\{w(i), w(i+1)\} = \{j, j+1\}\}$ yields that $\mathrm{des}_i(w) = \mathrm{des}_j(w^{-1})$. Moreover, it is clear that right multiplication by $s_i$ and left multiplication by $s_j$ preserve this event. Thus, by Lemma \ref{First_indep_result}, it follows that these descent random variables are independent from the event, and conditioning on it we still have that each has expectation $\frac{q}{1+q}$. The same is true for the event $E_2 = \{w(i), w(i+1)\} = \{-j, -j-1\}$ but we instead have that $\mathrm{des}_i(w) \neq \mathrm{des}_j(w^{-1})$. Hence,
\begin{align*}
    \mathrm{Cov}(\mathrm{des}_i(w), \mathrm{des}_{j}(w^{-1})) &= \E(\des_i(w)\des_j(w^{-1})) - \frac{q^2}{(1+q)^2} \\
    &= \E(\des_i(w)\des_j(w^{-1}) | E_1)\Prob(E_1) + \E(\des_i(w)\des_j(w^{-1}) | E_2)\Prob(E_2) \\
    &\quad - \frac{q^2}{(1+q)^2}\Prob(E) \\
    &= \frac{q}{1+q}\Prob(E_1) - \frac{q^2}{(1+q)^2}\Prob(E) = \frac{q}{(1+q)^2}\Prob(E_1) - \frac{q^2}{(1+q)^2}\Prob(E_2).
\end{align*}
There is an obvious bijection from $E_1$ to $E_2$ that switches $w(i), w(-i)$ and $w(i+1), w(-i-1)$ and strictly increases the length of $w$. Thus, $\Prob(E_2) < \Prob(E_1)$ so the above is positive. Now, if $i= 0, j \neq 0$, then $\{w(k): k \in \overline{\{s_i\}}\} = \{-w(k): k \in \overline{\{s_i\}}\} = \{w(-1), w(1)\}$ and $\{w(-1), w(1)\} \cap \overline{\{s_j\}}$ contains at most $1$ element, since one set contains two additive inverses while the other does not. Hence, the covariance term is $0$ in this case by Lemma \ref{inverse_indep}. A similar argument can be made to show that the covariance term is $0$ when $i \neq 0, j = 0$. If $i = j = 0$, $\{w(k): k \in \overline{\{s_i\}}\} \cap \overline{\{s_j\}}$ contains two elements if and only if 
$$
\{w(1), w(-1)\} = \{1, -1 \} \iff w(1) = \pm 1.
$$
Furthermore, Lemma \ref{First_indep_result} again implies that each descent statistic is independent of this event and so
$$
\mathrm{Cov}(\mathrm{des}_0(w), \mathrm{des}_{0}(w^{-1})) = \left( \frac{q}{1+q} - \frac{q^2}{(1+q)^2} \right) \mathbb{P}(w(1) = \pm 1).
$$
In all cases, $\mathrm{Cov}(\mathrm{des}_i(w), \mathrm{des}_{j}(w^{-1})) \ge 0$. Thus, by (\ref{variance_formula}), we have
$$
\mathrm{Var}(\mathrm{des}(w) + \mathrm{des}(w^{-1})) \ge \frac{2nq(1-q+q^2)}{(1+q)^2(1+q+q^2)}.
$$
For the upper bound, note that \eqref{B_n variance} implies that
\begin{align*}
    \sum_{0 \le i,j \le n-1} \mathrm{Cov}(\mathrm{des}_i(w), \mathrm{des}_j(w)) &\le n\frac{q}{(1+q)^2} - (n-2)\frac{q^2}{(1+q)^2(1+q+q^2)} - 2\frac{q^2}{(1+q)^2(1+q^2)} \\
    &\le \frac{nq(1+q^2)}{(1+q)^2(1+q+q^2)}.
\end{align*}
We also have that, from the above work,
\begin{align*}
    \sum_{0 \le i,j \le n-1} \mathrm{Cov}(\mathrm{des}_i(w), \mathrm{des}_j(w^{-1})) &\le \sum_{0 \le i,j \le n-1} \frac{q}{(1+q)^2} \mathbb{P}(\{w(i), w(i+1)\} = \{j, j+1\}) \\
    &\quad + \frac{q}{(1+q)^2}\mathbb{P}(w(1) = \pm 1) \\
    &\le (n-1)\frac{q}{(1+q)^2} + \frac{q}{(1+q)^2} = \frac{nq}{(1+q)^2}.
\end{align*}
Summing these bounds gives the result.
\end{proof}

\begin{corollary} \label{Cor: B_n simple_var_bound}
Let $w \in B_n$ be Mallows distributed with parameter $q \in (0, \infty)$ and $n \ge 2$. Then,
$$
\frac{n\min(q, 1/q)}{6} \le \Var(\des(w) + \des(w^{-1})) \le 4n\min(q,1/q).
$$
\end{corollary}
\begin{proof}
When $q \le 1$, this is a simple consequence of Lemma \ref{Moments: Type B}. The case of $q \ge 1$ follows from Corollary \ref{cor: q_invert}.
\end{proof}

\begin{lemma} \label{Moments: Type D}
Let $w \in D_n$ be Mallows distributed with parameter $q \le 1$ and $n \ge 2$. Then,
$$
\frac{2nq(1-q+q^2)}{(1+q)^2(1+q+q^2)} - \frac{10q^2}{(1+q)^2} \le \mathrm{Var}(\mathrm{des}(w) + \mathrm{des}(w^{-1})) \le \frac{4nq(2+2q+3q^2+q^3)}{(1+q)^2(1+q+q^2)}.
$$
\end{lemma}
\begin{proof}
We have, as before,
\begin{align}
    \mathrm{Var}(\mathrm{des}(w) + \mathrm{des}(w^{-1})) &= \mathrm{Var}\left(\sum_{i=0}^{n-1} \mathrm{des}_i(w) + \sum_{i=0}^{n-1} \mathrm{des}_i(w^{-1})\right) \nonumber \\
    &= 2\sum_{0 \le i,j \le n-1} \mathrm{Cov}(\mathrm{des}_i(w), \mathrm{des}_j(w)) + 2\sum_{0 \le i,j \le n-1} \mathrm{Cov}(\mathrm{des}_i(w), \mathrm{des}_j(w^{-1})). 
\end{align}
As in Lemma \ref{Moments: Type B}, we also have that $\des_i(w)$ and $\des_j(w)$ are independent if $s_i$ and $s_j$ commute and $i \neq j$. The computation for $\mathrm{Cov}(\des_i(w), \des_{i+1}(w))$ for $i \ge 1$ is identical as in Lemma \ref{Moments: Type B}. The same is true for $\mathrm{Cov}(\des_i(w), \des_{i}(w))$ for all $i$. Explicitly,
$$
\mathrm{Cov}(\des_i(w), \des_{i+1}(w)) = -\frac{q^2}{(1+q)^2(1+q+q^2)} \ \forall i \ge 1,
$$
$$
\mathrm{Cov}(\des_i(w), \des_{i}(w)) = \frac{q}{(1+q)^2} \ \forall i.
$$
Now, we have that $s_2$ is the only generator that $s_0$ does not commute with. And, $(s_0s_2)^3 = 1$, so $W_{\{s_0, s_2\}} \cong S_3$. Hence,
$$
\Cov(\des_0(w), \des_2(w)) = -\frac{q^2}{(1+q)^2(1+q+q^2)}
$$
and
\begin{align}
    \sum_{0 \le i,j \le n-1} \mathrm{Cov}(\mathrm{des}_i(w), \mathrm{des}_j(w)) &= n\frac{q}{(1+q)^2} - 2n\frac{q^2}{(1+q)^2(1+q+q^2)} \nonumber\\
    &= \frac{nq(1-q+q^2)}{(1+q)^2(1+q+q^2)}. \label{D_n var}
\end{align}
Now, we can use the same logic as in the proof of Lemma \ref{Moments: Type B} to show that, for $i,j \ge 1$,
\begin{equation} \label{D_n var most_cross_terms}
\mathrm{Cov}(\mathrm{des}_i(w), \mathrm{des}_{j}(w^{-1})) = \frac{q}{1+q}\mathbb{P}(\{w(i), w(i+1)\} = \{j, j+1\}) - \frac{q^2}{(1+q)^2}\mathbb{P}(\{w(i), w(i+1)\} = \{-j, -j-1\}) 
\end{equation}
We now consider the case $i=0, j \neq 0$. Then, by Lemma \ref{inverse_indep}, $\des_0(w)$ and $\des_j(w^{-1})$ are independent conditioning on the event that $\{w(1), w(2), w(-1), w(-2)\} \cap \{j , j+1\}$ contains at most $1$ element. Now, if the event $\{j , j+1\} \subseteq \{w(1), w(2), w(-1), w(-2)\}$ is true, we have that $\{w(1), w(2)\} = \{\pm j, \pm (j+1) \}$. Call this event $E_j$. Then, letting $p = \Prob(E_j)$,
\begin{align*}
    \Cov(\des_0(w), \des_j(w^{-1})) &= \mathbb{E}(\des_0(w)\des_j(w^{-1})) - \mathbb{E}(\des_0(w))\mathbb{E}(\des_j(w^{-1})) \\
    &= \mathbb{E}(\des_0(w)\des_j(w^{-1}) | E_j)p + \mathbb{E}(\des_0(w)|E_j^c)\E(\des_j(w^{-1}) | E_j^c)(1-p) - \frac{q^2}{(1+q)^2}\\
    &= \mathbb{E}(\des_0(w)\des_j(w^{-1}) | E_j)p + \frac{q^2}{(1+q)^2}(1-p) - \frac{q^2}{(1+q)^2} \\
    &\ge -\frac{2pq^2}{(1+q)^2}.
\end{align*}
We have used the fact that $\des_0(w)$ and $E_j$ are independent since there exists a bijection $w \to ws_0$ that preserves $E_j$ (and thus $E_j^c$). That is, in $E_j$, permutations come in pairs, differing in length by $1$, such that the shorter one does not have a descent at $s_0$ while the longer one does. A similar result is true for $\des_j(w^{-1})$. Now, the events $E_j$ are pairwise disjoint. Thus,
$$
\sum_{j = 1}^{n-1} \Cov(\des_0(w), \des_j(w^{-1})) \ge -\frac{2q^2}{(1+q)^2}.
$$
By similar reasoning (or simply subtituting $w^{-1}$ for $w$ above), $\sum\limits_{i\ge 1} \Cov(\des_i(w), \des_0(w^{-1})) \ge -\frac{2q^2}{(1+q)^2}$ as well. Since $\des_0$ is Bernoulli, we also have that $\Cov(\des_0(w), \des_0(w^{-1})) \ge -\E(\des_0(w))\E(\des_0(w^{-1})) = -\frac{q^2}{(1+q)^2}$. Hence,
$$
\sum_{0 \le i,j \le n-1} \mathrm{Cov}(\mathrm{des}_i(w), \mathrm{des}_j(w^{-1})) \ge -\frac{5q^2}{(1+q)^2}
$$
and we have the lower bound. To see the upper bound, note that
$$
\Cov(\des_i(w), \des_j(w^{-1})) \le \mathbb{E}(\des_0(w)\des_j(w^{-1})) \le \mathbb{E}(\des_0(w)) = \frac{q}{1+q}
$$
for any $i,j$ (possibly $0$). Thus, by (\ref{D_n var most_cross_terms}),
\begin{align*}
    \sum_{0 \le i,j \le n-1} \mathrm{Cov}(\mathrm{des}_i(w), \mathrm{des}_j(w^{-1})) &\le (2n-1)\frac{q}{1+q} + \sum_{1 \le i,j \le n-1} \frac{q}{(1+q)^2} \mathbb{P}(\{w(i), w(i+1)\} = \{j, j+1\}) \\
    &\le (n-1)\frac{q}{(1+q)^2} + (2n-1)\frac{q}{1+q} = \frac{nq(3+2q)}{(1+q)^2}.
\end{align*}
Combining this with \ref{D_n var} yields the desired upper bound.
\end{proof}

\begin{corollary} \label{Cor: D_n simple_var_bound}
Let $w \in D_n$ be Mallows distributed with parameter $q \in (0, \infty)$ and $n \ge 30$. Then,
$$
\frac{n\min(q, 1/q)}{12} \le \Var(\des(w), \des(w^{-1})) \le 8n\min(q, 1/q).
$$
\end{corollary}
\begin{proof}
First, suppose that $q \le 1$. The upper bound is not hard to see from Lemma \ref{Moments: Type D} since $2 + 2q + 3q^2 + q^2 \le 2(1+q)(1+q+q^2)$. For the lower bound, note that
$$
\frac{2nq(1-q+q^2)}{(1+q)^2(1+q+q^2)} - \frac{10q^2}{(1+q)^2} = \frac{q}{(1+q)^2}\left( \frac{2n(1-q+q^2)}{1+q+q^2} - 10q \right) \ge \frac{q}{(1+q)^2}\left( \frac{2n}{3} - \frac{nq}{3} \right) \ge \frac{nq}{12}.
$$
The case when $q \ge 1$ follows from Corollary \ref{cor: q_invert}.
\end{proof}

\begin{lemma} \label{Lemma: descent_cubed_bound}
Let $w \in W$ be Mallows distributed for $W = A_n, B_n$ or $D_n$ with parameter $q$, a positive real. We have the following bound on the expectation of the cube of the one-sided descent:
$$
\E(\des(w)^3) \le \frac{n^3q^3}{(1+q)^3} + \frac{24n^2q^2}{(1+q)^2} + \frac{16nq}{1+q}.
$$
\end{lemma}
\begin{proof}
We have,
$$
\E(\des(w)^3) = \E\left( \sum_{0 \le i,j,k \le n-1} \des_i(w)\des_j(w)\des_k(w) \right) = \sum_{0 \le i,j,k \le n-1} \E(\des_i(w)\des_j(w)\des_k(w)).
$$
We now do casework on which pairs of $s_i, s_j, s_k$ commute, i.e. the subgraph of the Coxeter graph of $W$ that $s_i, s_j, s_k$ form.
\begin{itemize}
    \item All pairs of the $s_i, s_j, s_k$ commute, i.e. the subgraph has no edges and the $s_i, s_j, s_k$ are distinct: We claim that the random variables $\des_i(w), \des_j(w), \des_k(w)$ are independent. To see this, condition on the values of $\des_i(w), \des_j(w)$. Then, these values are invariant after multiplication on the right by $s_k$ by Lemma \ref{Lemma: Descents_invariant_commute}. It follows that $w \mapsto ws_k$ is an involution on the subset of $W$ that results from the conditioning. Hence, this subset can be partitioned into pairs of the form $(w, ws_k)$ so that $\des_k(w)$ is still Bernoulli with expectation $\frac{q}{1+q}$. The analogous results are true for $\des_i(w), \des_j(w)$, so the three random variables are independent. The amount contributed to the sum is therefore
    $$
    \E(\des_i(w)\des_j(w)\des_k(w)) = \E(\des_i(w))\E(\des_j(w))\E(\des_k(W)) = \frac{q^3}{(1+q)^3}.
    $$
    There are clearly at most $n^3$ such triples $(i,j,k)$ that satisfy this case.

    \item Two of the generators are distinct and commute but a different pair of the generators does not satisfy this: Without loss of generality suppose that $s_i, s_j$ are distinct and commute. Then, the contribution to the sum above is at most
    $$
    \E(\des_i(w)\des_j(w)\des_k(w)) \le \E(\des_i(w)\des_j(w)) = \E(\des_i(w))\E(\des_j(w)) = \frac{q^2}{(1+q)^2}.
    $$
    We now bound the number of such triples. First, there are $3$ choices for the pair of indices that correspond to commuting generators, and then at most $n^2$ possibilities for the pair of generators. There are then $8$ choices for the last generator since it is either a neighbor of or equal to one of the first two. This results an an upper bound of $24n^2$ triples.

    \item No pair of generators in $s_i, s_j, s_k$ satisfy that they are distinct and commute: The contribution to the sum is trivially bounded above by
    $$
    \E(\des_i(w)\des_j(w)\des_k(w)) \le \E(\des_i(w)) = \frac{q}{1+q}.
    $$
    We have each of $s_j$ and $s_k$ are either equal to or a neighbor of $s_i$, so there are at most $4$ possibilities for each. Hence, there are at most $16n$ triples that satisfy this condition.
\end{itemize}
From the above work, we have that
$$
\E(\des(w)^3) \le n^3 \frac{q^3}{(1+q)^3} + 24n^2 \frac{q^2}{(1+q)^2} + 16n \frac{q}{1+q}.
$$
\end{proof}

\printbibliography

@article{He_2022,
	doi = {10.1214/21-aihp1167},
  
	url = {https://doi.org/10.1214%2F21-aihp1167},
  
	year = 2022,
	month = {5},
  
	publisher = {Institute of Mathematical Statistics},
  
	volume = {58},
  
	number = {2},
  
	author = {Jimmy He},
  
	title = {A central limit theorem for descents of a Mallows permutation and its inverse},
  
	journal = {Annales de l{\textquotesingle}Institut Henri Poincar{\'{e}
},
   Probabilit{\'{e}}s et Statistiques}
}

@misc{ferayCLT,
      title={On the central limit theorem for the two-sided descent statistics in Coxeter groups}, 
      author={Valentin Féray},
      year={2020},
      eprint={1911.10939},
      archivePrefix={arXiv},
      primaryClass={math.PR}
}

@incollection{BRS89,
  author = {P. Baldi and Y. Rinott and C. Stein},
  title = {A Normal Approximation for the Number of Local Maxima of a Random Function on a Graph},
  booktitle = {Probability, Statistics, and Mathematics},
  publisher = {Academic Press},
  address = {Boston, MA},
  year = {1989},
  pages = {59--81}
}

@book {Bil95,
    AUTHOR = {Billingsley, Patrick},
     TITLE = {Probability and measure},
    SERIES = {Wiley Series in Probability and Mathematical Statistics},
   EDITION = {Third},
      NOTE = {A Wiley-Interscience Publication},
 PUBLISHER = {John Wiley \& Sons, Inc., New York},
      YEAR = {1995},
     PAGES = {xiv+593},
      ISBN = {0-471-00710-2},
   MRCLASS = {60-01 (28-01)},
  MRNUMBER = {1324786},
}

@article{BR22,
	doi = {10.37236/10744},
  
	url = {https://doi.org/10.37236%2F10744},
  
	year = 2022,
	month = {1},
  
	publisher = {The Electronic Journal of Combinatorics},
  
	volume = {29},
  
	number = {1},
  
	author = {Benjamin Brück and Frank Röttger},
  
	title = {A Central Limit Theorem for the Two-Sided Descent Statistic on Coxeter Groups},
  
	journal = {The Electronic Journal of Combinatorics}
}

@article{BDF10,
  author = {Alexei Borodin and Persi Diaconis and Jason Fulman},
  title = {On Adding a List of Numbers (and Other One-Dependent Determinantal Processes)},
  journal = {Bull. Amer. Math. Soc. (N.S.)},
  volume = {47},
  number = {4},
  year = {2010},
  pages = {639--670}
}

@article{CV07,
  author = {Mark Conger and D. Viswanath},
  title = {Normal Approximations for Descents and Inversions of Permutations of Multisets},
  journal = {J. Theoret. Probab.},
  volume = {20},
  number = {2},
  year = {2007},
  pages = {309--325}
}

@article{G05,
  author = {Larry Goldstein},
  title = {Berry-Esseen Bounds for Combinatorial Central Limit Theorems and Pattern Occurrences, Using Zero and Size Biasing},
  journal = {J. Appl. Probab.},
  volume = {42},
  number = {3},
  year = {2005},
  pages = {661--683}
}

@article{GR96,
  author = {Larry Goldstein and Yosef Rinott},
  title = {Multivariate Normal Approximations by Stein's Method and Size Bias Couplings},
  journal = {J. Appl. Probab.},
  volume = {33},
  number = {1},
  year = {1996},
  pages = {1--17}
}

@article{KS20,
  author = {Thomas Kahle and Christian Stump},
  title = {Counting Inversions and Descents of Random Elements in Finite {C}oxeter Groups},
  journal = {Math. Comp.},
  volume = {89},
  number = {321},
  year = {2020},
  pages = {437--464}
}

@book{BB05,
  author = {Anders Björner and Francesco Brenti},
  title = {Combinatorics of {C}oxeter Groups},
  series = {Graduate Texts in Mathematics},
  volume = {231},
  publisher = {Springer},
  address = {New York},
  year = {2005}
}

@article{Mal72,
author = {C. L. Mallows},
title = {{A Note on Asymptotic Joint Normality}},
volume = {43},
journal = {The Annals of Mathematical Statistics},
number = {2},
publisher = {Institute of Mathematical Statistics},
pages = {508 -- 515},
year = {1972},
doi = {10.1214/aoms/1177692631},
URL = {https://doi.org/10.1214/aoms/1177692631}
}

@article{R11b,
  author = {Nathan Ross},
  title = {Fundamentals of {S}tein's Method},
  journal = {Probab. Surv.},
  volume = {8},
  year = {2011},
  pages = {210--293}
}

@article {Mal57,
    AUTHOR = {Mallows, C. L.},
     TITLE = {Non-null ranking models. {I}},
   JOURNAL = {Biometrika},
  FJOURNAL = {Biometrika},
    VOLUME = {44},
      YEAR = {1957},
     PAGES = {114--130},
      ISSN = {0006-3444,1464-3510},
   MRCLASS = {62.0X},
  MRNUMBER = {87267},
       DOI = {10.1093/biomet/44.1-2.114},
       URL = {https://doi.org/10.1093/biomet/44.1-2.114},
}

@article {Muk16a,
    AUTHOR = {Mukherjee, Sumit},
     TITLE = {Estimation in exponential families on permutations},
   JOURNAL = {Ann. Statist.},
  FJOURNAL = {The Annals of Statistics},
    VOLUME = {44},
      YEAR = {2016},
    NUMBER = {2},
     PAGES = {853--875},
      ISSN = {0090-5364,2168-8966},
   MRCLASS = {62F12 (05A05 60F10)},
  MRNUMBER = {3476619},
       DOI = {10.1214/15-AOS1389},
       URL = {https://doi.org/10.1214/15-AOS1389},
}

@article {HHL,
    AUTHOR = {Holroyd, Alexander E. and Hutchcroft, Tom and Levy, Avi},
     TITLE = {Mallows permutations and finite dependence},
   JOURNAL = {Ann. Probab.},
  FJOURNAL = {The Annals of Probability},
    VOLUME = {48},
      YEAR = {2020},
    NUMBER = {1},
     PAGES = {343--379},
      ISSN = {0091-1798,2168-894X},
   MRCLASS = {60G10 (05A05 05C15)},
  MRNUMBER = {4079439},
MRREVIEWER = {Eugenijus\ Manstavi\v{c}ius},
       DOI = {10.1214/19-AOP1363},
       URL = {https://doi.org/10.1214/19-AOP1363},
}

@InProceedings{Tang19,
  title = 	 {Mallows ranking models: maximum likelihood estimate and regeneration},
  author =       {Tang, Wenpin},
  booktitle = 	 {Proceedings of the 36th International Conference on Machine Learning},
  pages = 	 {6125--6134},
  year = 	 {2019},
  editor = 	 {Chaudhuri, Kamalika and Salakhutdinov, Ruslan},
  volume = 	 {97},
  series = 	 {Proceedings of Machine Learning Research},
  month = 	 {09--15 Jun},
  publisher =    {PMLR},
  url = 	 {https://proceedings.mlr.press/v97/tang19a.html}
}

@article {AHHL21,
    AUTHOR = {Angel, Omer and Holroyd, Alexander E. and Hutchcroft, Tom and Levy, Avi},
     TITLE = {Mallows permutations as stable matchings},
   JOURNAL = {Canad. J. Math.},
  FJOURNAL = {Canadian Journal of Mathematics. Journal Canadien de
              Math\'{e}matiques},
    VOLUME = {73},
      YEAR = {2021},
    NUMBER = {6},
     PAGES = {1531--1555},
      ISSN = {0008-414X,1496-4279},
   MRCLASS = {60C05 (05A05 05C70)},
  MRNUMBER = {4350552},
       DOI = {10.4153/S0008414X20000590},
       URL = {https://doi.org/10.4153/S0008414X20000590},
}

@article{Gon44,
  title={Some facts from combinatorics},
  author={Goncharov, V. L.},
  journal={Izvestiya Akademii Nauk SSSR, Seriya Matematicheskaya},
  volume={8},
  pages={3--48},
  year={1944},
  note={Translated as: On the field of combinatory analysis. In Transl. Amer. Math. Soc., Vol. 19, pp. 1-46, 1962}
}

@article{Ho51,
  title={A combinatorial central limit theorem},
  author={Hoeffding, W.},
  journal={Annals of Mathematical Statistics},
  volume={22},
  number={4},
  pages={558--566},
  year={1951}
}

@article {CGJ18,
    AUTHOR = {Cook, Nicholas and Goldstein, Larry and Johnson, Tobias},
     TITLE = {Size biased couplings and the spectral gap for random regular
              graphs},
   JOURNAL = {Ann. Probab.},
  FJOURNAL = {The Annals of Probability},
    VOLUME = {46},
      YEAR = {2018},
    NUMBER = {1},
     PAGES = {72--125},
      ISSN = {0091-1798,2168-894X},
   MRCLASS = {05C80 (60B20 60E15)},
  MRNUMBER = {3758727},
       DOI = {10.1214/17-AOP1180},
       URL = {https://doi.org/10.1214/17-AOP1180},
}

@article {AB15,
    AUTHOR = {Arratia, Richard and Baxendale, Peter},
     TITLE = {Bounded size bias coupling: a {G}amma function bound, and
              universal {D}ickman-function behavior},
   JOURNAL = {Probab. Theory Related Fields},
  FJOURNAL = {Probability Theory and Related Fields},
    VOLUME = {162},
      YEAR = {2015},
    NUMBER = {3-4},
     PAGES = {411--429},
      ISSN = {0178-8051,1432-2064},
   MRCLASS = {60E15},
  MRNUMBER = {3383333},
MRREVIEWER = {Xiequan\ Fan},
       DOI = {10.1007/s00440-014-0572-x},
       URL = {https://doi.org/10.1007/s00440-014-0572-x},
}

@article {LL19,
    AUTHOR = {Labb\'{e}, Cyril and Lacoin, Hubert},
     TITLE = {Cutoff phenomenon for the asymmetric simple exclusion process
              and the biased card shuffling},
   JOURNAL = {Ann. Probab.},
  FJOURNAL = {The Annals of Probability},
    VOLUME = {47},
      YEAR = {2019},
    NUMBER = {3},
     PAGES = {1541--1586},
      ISSN = {0091-1798,2168-894X},
   MRCLASS = {60J27 (37A25 82C22)},
  MRNUMBER = {3945753},
MRREVIEWER = {Leonid\ Petrov},
       DOI = {10.1214/18-AOP1290},
       URL = {https://doi.org/10.1214/18-AOP1290},
}

@misc{B20,
      title={Interacting particle systems and random walks on Hecke algebras}, 
      author={Alexey Bufetov},
      year={2020},
      eprint={2003.02730},
      archivePrefix={arXiv},
      primaryClass={math.PR}
}

@incollection {DR00,
    AUTHOR = {Diaconis, Persi and Ram, Arun},
     TITLE = {Analysis of systematic scan {M}etropolis algorithms using
              {I}wahori-{H}ecke algebra techniques},
      NOTE = {Dedicated to William Fulton on the occasion of his 60th
              birthday},
   JOURNAL = {Michigan Math. J.},
  FJOURNAL = {Michigan Mathematical Journal},
    VOLUME = {48},
      YEAR = {2000},
     PAGES = {157--190},
      ISSN = {0026-2285,1945-2365},
   MRCLASS = {60J10 (20C08 60B15 60F05)},
  MRNUMBER = {1786485},
MRREVIEWER = {Martin\ V.\ Hildebrand},
       DOI = {10.1307/mmj/1030132713},
       URL = {https://doi.org/10.1307/mmj/1030132713},
}

@article {CD17,
    AUTHOR = {Chatterjee, Sourav and Diaconis, Persi},
     TITLE = {A central limit theorem for a new statistic on permutations},
   JOURNAL = {Indian J. Pure Appl. Math.},
  FJOURNAL = {Indian Journal of Pure and Applied Mathematics},
    VOLUME = {48},
      YEAR = {2017},
    NUMBER = {4},
     PAGES = {561--573},
      ISSN = {0019-5588,0975-7465},
   MRCLASS = {60F05 (05A05)},
  MRNUMBER = {3741694},
MRREVIEWER = {Anant\ P.\ Godbole},
       DOI = {10.1007/s13226-017-0246-3},
       URL = {https://doi.org/10.1007/s13226-017-0246-3},
}

\end{document}